\documentclass[10pt]{amsart}

\newif\iffinalrun

\usepackage{amsmath}
\usepackage{latexsym}
\usepackage{amsfonts}
\usepackage{amssymb}
\usepackage{graphicx}
\usepackage{mathrsfs}
\usepackage{graphpap}
\usepackage{color}
\usepackage[margin=1.2in]{geometry}
\usepackage[all]{xy}
\theoremstyle{definition}

\numberwithin{equation}{section}

\usepackage[all]{xy}
\theoremstyle{theorem}
\newtheorem{theo}[equation]{Theorem}
\newtheorem*{theo*}{Theorem}

\newtheorem{lemm}[equation]{Lemma}
\newtheorem{coro}[equation]{Corollary}
\newtheorem{prop}[equation]{Proposition}
\newtheorem{rema}[equation]{Remark}

\iffinalrun
  \newcommand{\need}[1]{}
  \newcommand{\mar}[1]{}
\else
  \newcommand{\need}[1]{{\tiny *** #1}}
  \newcommand{\mar}[1]{\marginpar{\raggedright\tiny #1}}
\fi

\begin{document}

\title{Dilogarithm and Higher $\mathscr{L}$\text{-}invariants for $\mathrm{GL}_3(\mathbb{Q}_p)$}
%\date{\today}

\author{Zicheng Qian}
\address{D\'epartement de Math\'ematiques Batiment 425, Facult\'e des Sciences d'Orsay Universit\'e Paris-Sud, 91405 Orsay, France}
\email{zicheng.qian@u-psud.fr}


\maketitle

\begin{abstract}
The primary purpose of this paper is to clarify the relation between previous results in \cite{Schr11}, \cite{Bre17} and \cite{BD18}. Let $E$ be a sufficiently large finite extension of $\mathbb{Q}_p$ and $\rho_p$ be a $p$\text{-}adic semi-stable representation $\mathrm{Gal}(\overline{\mathbb{Q}_p}/\mathbb{Q}_p)\rightarrow\mathrm{GL}_3(E)$ such that the Weil--Deligne representation $\mathrm{WD}(\rho_p)$ associated with it has rank two monodromy operator $N$ and the Hodge filtration associated with it is non-critical. Then by a computation of extensions of rank one $(\varphi, \Gamma)$-modules we know that the Hodge filtration of $\rho_p$ depends on three invariants in $E$. We construct a family of locally analytic representations $\Sigma^{\rm{min}}(\lambda, \mathscr{L}_1, \mathscr{L}_2, \mathscr{L}_3)$ of $\mathrm{GL}_3(\mathbb{Q}_p)$ depending on three invariants $\mathscr{L}_1, \mathscr{L}_2, \mathscr{L}_3 \in E$ with each of the representation containing the locally algebraic representation $\mathrm{Alg}\otimes\mathrm{Steinberg}$ determined by $\mathrm{WD}(\rho_p)$ via classical local Langlands correspondence for $\mathrm{GL}_3(\mathbb{Q}_p)$  and by the Hodge--Tate weights of $\rho_p$. When $\rho_p$ comes from an automorphic representation $\pi$ of $G(\mathbb{A}_{\mathbb{Q}_p})$ with a fixed level $U^p$ prime to $p$ for a suitable unitary group $G_{/\mathbb{Q}}$, we show (under some technical assumption) that there is a unique locally analytic representation in the above family that occurs as a subrepresentation of the associated Hecke-isotypic subspace in the completed cohomology with level $U^p$. We recall that \cite{Bre17} constructed a family of locally analytic representations depending on four invariants ( cf. (4) in \cite{Bre17}) and proved that there is a unique representation in the family that embeds into the fixed Hecke-isotypic space above. We prove that if a representation $\Pi$ in Breuil's family embeds into a certain Hecke-isotypic subspace of completed cohomology, then it must equally embed into $\Sigma^{\rm{min}}(\lambda, \mathscr{L}_1, \mathscr{L}_2, \mathscr{L}_3)$ for certain choices of $\mathscr{L}_1, \mathscr{L}_2, \mathscr{L}_3\in E$ determined explicitly by $\Pi$. This gives a purely representation theoretic necessary condition for $\Pi$ to embed into completed cohomology. Moreover, certain natural subquotients of $\Sigma^{\rm{min}}(\lambda, \mathscr{L}_1, \mathscr{L}_2, \mathscr{L}_3)$ give a true complex of locally analytic representations that realizes the derived object $\Sigma(\lambda, \underline{\mathscr{L}})$ in (1.14) of \cite{Schr11}. Consequently, the locally analytic representation $\Sigma^{\rm{min}}(\lambda, \mathscr{L}_1, \mathscr{L}_2, \mathscr{L}_3)$ gives a relation between the higher $\mathscr{L}$\text{-}invariants studied in \cite{Bre17} as well as \cite{BD18} and the $p$\text{-}adic dilogarithm function which appears in the construction of $\Sigma(\lambda, \underline{\mathscr{L}})$ in \cite{Schr11}.
\end{abstract}

\setcounter{tocdepth}{2}

\tableofcontents

\section{Introduction}\label{3section: introduction}
Let $p$ be a prime number and $F$ an imaginary quadratic extension of $\mathbb{Q}$ such that $p$ splits in $F$. We fix a unitary algebraic group $G$ over $\mathbb{Q}$ which becomes $\mathrm{GL}_n$ over $F$ and such that $G(\mathbb{R})$ is compact and $G$ is split at all places of $F$ above $p$. Then to each finite extension $E$ of $\mathbb{Q}_p$ and to each prime\text{-}to\text{-}$p$ level $U^p$ in $G(\mathbb{A}_{\mathbb{Q}}^{\infty, p})$, one can associate the Banach space of $p$\text{-}adic automorphic forms $\widehat{S}(U^p, E)$. One can also associate with $U^p$ a set of finite places $D(U^p)$ of $\mathbb{Q}$ and a Hecke algebra $\mathbb{T}(U^p)$ which is the polynomial algebra freely generated by Hecke operators at places of $F$ lying above $D(U^p)$. In particular, the commutative algebra $\mathbb{T}(U^p)$ acts on $\widehat{S}(U^p, E)$ and commutes with the action of $G(\mathbb{Q}_p)\cong\mathrm{GL}_n(\mathbb{Q}_p)$ coming from translations on $G(\mathbb{A}_{\mathbb{Q}}^{\infty})$.

If $\rho:\mathrm{Gal}(\overline{F}/F)\rightarrow\mathrm{GL}_n(E)$ is a continuous irreducible representation, one considers the associated Hecke isotypic subspace $\widehat{S}(U^p, E)[\mathfrak{m}_{\rho}]$, which is a continuous admissible representation of $G(\mathbb{Q}_p)\cong\mathrm{GL}_n(\mathbb{Q}_p)$ over $E$, or its locally $\mathbb{Q}_p$\text{-}analytic vectors $\widehat{S}(U^p, E)[\mathfrak{m}_{\rho}]^{\rm{an}}$, which is an admissible locally $\mathbb{Q}_p$\text{-}analytic representation of $\mathrm{GL}_n(\mathbb{Q}_p)$. We fix $w_p$ a place of $F$ above $p$ and it is widely wished that $\widehat{S}(U^p, E)[\mathfrak{m}_{\rho}]$ (and its subspace $\widehat{S}(U^p, E)[\mathfrak{m}_{\rho}]^{\rm{an}}$ as well) determines and depends only on $\rho_p:=\rho|_{\mathrm{Gal}(\overline{F_{w_p}}/F_{w_p})}$. The case $n=2$ is well-known essentially due to various results in \cite{Col10}, \cite{Eme}. The case $n\geq 3$ is much more difficult and only some partial results are known. We are particularly interested in the case when the subspace of locally algebraic vectors $\widehat{S}(U^p, E)[\mathfrak{m}_{\rho}]^{\rm{alg}}\subsetneq \widehat{S}(U^p, E)[\mathfrak{m}_{\rho}]$ is non-zero, which implies that $\rho_p$ is potentially semi-stable. Certain cases when $n=3$ and $\rho_p$ is semi-stable and non-crystalline have been studied in \cite{Bre17} and \cite{BD18}. We are going to continue their work and obtain some interesting relation between results in \cite{Bre17}, \cite{BD18} and previous results in \cite{Schr11} which involve the $p$\text{-}adic dilogarithm function.

We use the notation $\lambda\in X(T)_+$ for a weight $\lambda=(\lambda_1, \lambda_2, \lambda_3)$ (of the diagonal split torus $T$ of $\mathrm{GL}_3$) which is dominant with respect to the upper-triangular Borel subgroup $\overline{B}$ and hence satisfies $\lambda_1\geq \lambda_2\geq\lambda_3$. Given two locally analytic representations $V, W$ of $\mathrm{GL}_3(\mathbb{Q}_p)$, we use the shorten notation $\begin{xy}
(0,0)*+{V}="a"; (10,0)*+{W}="b";
{\ar@{-}"a";"b"};
\end{xy}$ (resp. the shorten notation $\begin{xy}
(0,0)*+{V}="a"; (10,0)*+{W}="b";
{\ar@{--}"a";"b"};
\end{xy}$) for a locally analytic representation determined by a non-zero (resp. possibly zero) element in $\mathrm{Ext}^1_{\mathrm{GL}_3(\mathbb{Q}_p)}\left(W, ~V\right)$.
\begin{theo}\label{3theo: construction introduction}[Proposition~\ref{3prop: main dim}, Proposition~\ref{3prop: criterion of existence}]
For each choice of $\lambda\in X(T)_+$ and $\mathscr{L}_1, \mathscr{L}_2, \mathscr{L}_3\in E$, there exists a locally analytic representation
$\Sigma^{\rm{min}}(\lambda, \mathscr{L}_1, \mathscr{L}_2, \mathscr{L}_3)$ of $\mathrm{GL}_3(\mathbb{Q}_p)$ of the form:
\begin{equation}
\begin{xy}
(0,0)*+{\mathrm{St}_3^{\rm{an}}(\lambda)}="a"; (20,5)*+{v_{P_1}^{\rm{an}}(\lambda)}="b"; (40,8)*+{C_{s_1,s_1}}="d"; (60,8)*+{\overline{L}(\lambda)\otimes_Ev_{P_2}^{\infty}}="f"; (20,-5)*+{v_{P_2}^{\rm{an}}(\lambda)}="c"; (40,-8)*+{C_{s_2,s_2}}="e"; (60,-8)*+{\overline{L}(\lambda)\otimes_Ev_{P_1}^{\infty}}="g"; (80,3)*+{\overline{L}(\lambda)}="h1"; (80,-3)*+{\overline{L}(\lambda)}="h2";
{\ar@{-}"a";"b"}; {\ar@{-}"a";"c"}; {\ar@{-}"b";"d"}; {\ar@{-}"c";"e"}; {\ar@{-}"d";"f"}; {\ar@{-}"e";"g"}; {\ar@{-}"f";"h1"}; {\ar@{-}"b";"h1"}; {\ar@{-}"g";"h1"}; {\ar@{-}"f";"h2"}; {\ar@{-}"c";"h2"}; {\ar@{-}"g";"h2"};
\end{xy}
\end{equation}
where $\mathrm{St}_3^{\rm{an}}(\lambda)$, $v_{P_1}^{\rm{an}}(\lambda)$, $v_{P_2}^{\rm{an}}(\lambda)$, $\overline{L}(\lambda)$ and $C^{\ast}_{w^{\prime},w}$ for $w,w^{\prime}\in \{s_1, s_2, s_1s_2, s_2s_1\}$ and $\ast\in \{\varnothing, 1, 2\}$ are various explicit locally analytic representations defined in Section~\ref{3subsection: main notation}. Moreover, different choices of $\mathscr{L}_1, \mathscr{L}_2, \mathscr{L}_3\in E$ give non-isomorphic representations.
\end{theo}
We will see in Lemma~\ref{3lemm: distinguish mult two} and (\ref{3main picture for min}) that $\Sigma^{\rm{min}}(\lambda, \mathscr{L}_1, \mathscr{L}_2, \mathscr{L}_3)$ is \emph{the minimal locally analytic representation that involves $p$\text{-}adic dilogarithm}, hence explains the notation `\rm{min}'. We also construct a locally analytic representation $\Sigma^{\rm{min}, +}(\lambda, \mathscr{L}_1, \mathscr{L}_2, \mathscr{L}_3)$ of the form
$$
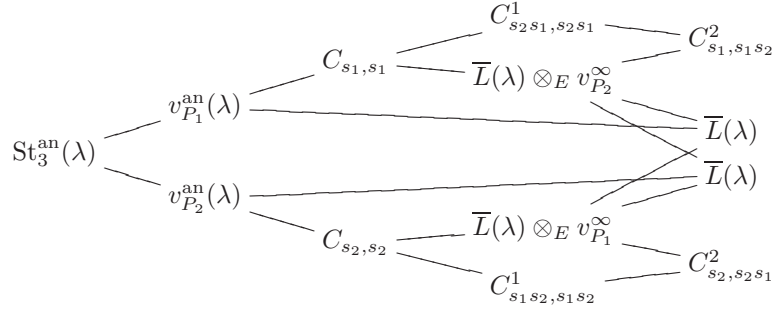
\begin{xy}
(0,0)*+{\mathrm{St}_3^{\rm{an}}(\lambda)}="a"; (20,6)*+{v_{P_1}^{\rm{an}}(\lambda)}="b"; (40,12)*+{C_{s_1,s_1}}="d"; (65,10)*+{\overline{L}(\lambda)\otimes_Ev_{P_2}^{\infty}}="f"; (20,-6)*+{v_{P_2}^{\rm{an}}(\lambda)}="c"; (40,-12)*+{C_{s_2,s_2}}="e"; (65,-10)*+{\overline{L}(\lambda)\otimes_Ev_{P_1}^{\infty}}="g"; (90,3)*+{\overline{L}(\lambda)}="h1"; (90,-3)*+{\overline{L}(\lambda)}="h2"; (65,18)*+{C^1_{s_2s_1,s_2s_1}}="i";  (65,-18)*+{C^1_{s_1s_2,s_1s_2}}="j";  (90,15)*+{C^2_{s_1,s_1s_2}}="k";  (90,-15)*+{C^2_{s_2,s_2s_1}}="l";
{\ar@{-}"a";"b"}; {\ar@{-}"a";"c"}; {\ar@{-}"b";"d"}; {\ar@{-}"c";"e"}; {\ar@{-}"d";"f"}; {\ar@{-}"e";"g"}; {\ar@{-}"f";"h1"}; {\ar@{-}"b";"h1"}; {\ar@{-}"g";"h1"}; {\ar@{-}"f";"h2"}; {\ar@{-}"c";"h2"}; {\ar@{-}"g";"h2"}; {\ar@{-}"d";"i"}; {\ar@{-}"e";"j"}; {\ar@{-}"i";"k"}; {\ar@{-}"f";"k"}; {\ar@{-}"j";"l"}; {\ar@{-}"g";"l"};
\end{xy}
$$
which contains and is uniquely determined by $\Sigma^{\rm{min}}(\lambda, \mathscr{L}_1, \mathscr{L}_2, \mathscr{L}_3)$.
\begin{theo}\label{3theo: main introduction}[Theorem~\ref{3theo: main}]
Assume that $p\geq 5$ and $n=3$. Assume moreover that
\begin{enumerate}
\item $\rho$ is unramified at all finite places of $F$ above $D(U^p)$;
\item $\widehat{S}(U^p, E)[\mathfrak{m}_{\rho}]^{\rm{alg}}\neq 0$;
\item $\rho_p$ is semi-stable with Hodge--Tate weights $\{k_1>k_2>k_3\}$ such that $N^2\neq 0$;
\item $\rho_p$ is non-critical in the sense of Remark~6.1.4 of \cite{Bre17};
\item only one automorphic representation contributes to $\widehat{S}(U^p, E)[\mathfrak{m}_{\rho}]^{\rm{alg}}$.
\end{enumerate}
Then there exists a unique choice of $\mathscr{L}_1, \mathscr{L}_2, \mathscr{L}_3\in E$ such that $\widehat{S}(U^p, E)[\mathfrak{m}_{\rho}]^{\rm{an}}$ contains (copies of) the locally analytic representation
$$\Sigma^{\rm{min}, +}(\lambda, \mathscr{L}_1, \mathscr{L}_2, \mathscr{L}_3)\otimes_E(\mathrm{ur}(\alpha)\otimes_E\varepsilon^2)\circ\mathrm{det}$$
where $\lambda=(\lambda_1, \lambda_2, \lambda_3)=(k_1-2, k_2-1, k_3)$ and $\alpha\in E^{\times}$ is determined by the Weil--Deligne representation $\mathrm{WD}(\rho_p)$ associated with $\rho_p$. Moreover, we have
\begin{multline}
\mathrm{Hom}_{\mathrm{GL}_3(\mathbb{Q}_p)}\left(\Sigma^{\rm{min},+}(\lambda, \mathscr{L}_1, \mathscr{L}_2, \mathscr{L}_3)\otimes_E(\mathrm{ur}(\alpha)\otimes_E\varepsilon^2)\circ\mathrm{det}, \widehat{S}(U^p, E)^{\rm{an}}[\mathfrak{m}_{\rho}]\right)\\
\xrightarrow{\sim}\mathrm{Hom}_{\mathrm{GL}_3(\mathbb{Q}_p)}\left(\overline{L}(\lambda)\otimes_E\mathrm{St}_3^{\infty}\otimes_E(\mathrm{ur}(\alpha)\otimes_E\varepsilon^2)\circ\mathrm{det}, \widehat{S}(U^p, E)^{\rm{an}}[\mathfrak{m}_{\rho}]\right).\\
\end{multline}
\end{theo}
The assumptions of our Theorem~\ref{3theo: main introduction} are the same as that of Theorem~1.3 of \cite{Bre17}. We do not attempt to obtain any explicit relation between $\mathscr{L}_1, \mathscr{L}_2, \mathscr{L}_3\in E$ and $\rho_p$, which is similar in flavor to Theorem~1.3 of \cite{Bre17}. On the other hand, Theorem~7.52 of \cite{BD18} does care about the explicit relation between invariants of the locally analytic representation associated with $\rho_p$, under further technical assumptions such as $\rho_p$ is ordinary with consecutive Hodge--Tate weights and has an irreducible mod $p$ reduction but without assuming our condition (v). The improvement of our Theorem~\ref{3theo: main introduction} upon Theorem~1.3 of \cite{Bre17} will be explained in Section~\ref{3subsection: criterion of global embedding}. One can naturally wish that there is a common refinement or generalization of our Theorem~\ref{3theo: main introduction} and Theorem~7.52 of \cite{BD18} by removing as many technical assumptions as possible.

\begin{rema}\label{3rema: construction max}
It is actually possible to construct a locally analytic representation $\Sigma^{\rm{max}}(\lambda, \mathscr{L}_1, \mathscr{L}_2, \mathscr{L}_3)$ of $\mathrm{GL}_3(\mathbb{Q}_p)$ containing $\Sigma^{\rm{min}, +}(\lambda, \mathscr{L}_1, \mathscr{L}_2, \mathscr{L}_3)$ which is characterized by the fact that it is maximal (for inclusion) among the locally analytic representations $V$ satisfying the following conditions:
\begin{enumerate}
\item $\mathrm{soc}_{\mathrm{GL}_3(\mathbb{Q}_p)}(V)=V^{\rm{alg}}=\overline{L}(\lambda)\otimes_E\mathrm{St}_3^{\infty}$;
\item each constituent of $V$ is a subquotient of a locally analytic principal series
\end{enumerate}
where $V^{\rm{alg}}$ is the subspace of locally algebraic vectors in $V$. Moreover, one can use an immediate generalization of the arguments in the proof of Theorem~\ref{3theo: main introduction} (and thus of Theorem~1.1 of \cite{Bre17}) to show that
\begin{multline}
\mathrm{Hom}_{\mathrm{GL}_3(\mathbb{Q}_p)}\left(\Sigma^{\rm{max}}(\lambda, \mathscr{L}_1, \mathscr{L}_2, \mathscr{L}_3)\otimes_E(\mathrm{ur}(\alpha)\otimes_E\varepsilon^2)\circ\mathrm{det}, \widehat{S}(U^p, E)^{\rm{an}}[\mathfrak{m}_{\rho}]\right)\\
\xrightarrow{\sim}\mathrm{Hom}_{\mathrm{GL}_3(\mathbb{Q}_p)}\left(\overline{L}(\lambda)\otimes_E\mathrm{St}_3^{\infty}\otimes_E(\mathrm{ur}(\alpha)\otimes_E\varepsilon^2)\circ\mathrm{det}, \widehat{S}(U^p, E)^{\rm{an}}[\mathfrak{m}_{\rho}]\right).\\
\end{multline}
We can also show that
$$\Sigma^{\rm{max}}(\lambda, \mathscr{L}_1, \mathscr{L}_2, \mathscr{L}_3)/\overline{L}(\lambda)\otimes_E\mathrm{St}_3$$
is independent of the choice of $\mathscr{L}_1, \mathscr{L}_2, \mathscr{L}_3\in E$, which is compatible with the fact that
$$\Sigma^{\rm{min}, \ast}(\lambda, \mathscr{L}_1, \mathscr{L}_2, \mathscr{L}_3)/\overline{L}(\lambda)\otimes_E\mathrm{St}_3$$
is independent of the choice of $\mathscr{L}_1, \mathscr{L}_2, \mathscr{L}_3\in E$ for each $\ast\in \{\varnothing, +\}$ as mentioned in Remark~\ref{3quotient independent of invariants}. However, the full construction of $\Sigma^{\rm{max}}(\lambda, \mathscr{L}_1, \mathscr{L}_2, \mathscr{L}_3)$ is lengthy and technical and thus we decided not to put it in the present article.
\end{rema}
\subsection{Derived object and dilogarithm}\label{3subsection: dilog}
We consider the bounded derived category
$$\mathcal{D}^b\left(\mathrm{Mod}_{D(\mathrm{GL}_3(\mathbb{Q}_p), E)}\right)$$
associated with the abelian category $\mathrm{Mod}_{D(\mathrm{GL}_3(\mathbb{Q}_p), E)}$ of abstract modules over the algebra $D(\mathrm{GL}_3(\mathbb{Q}_p), E)$ of locally $\mathbb{Q}_p$\text{-}analytic distributions on $\mathrm{GL}_3(\mathbb{Q}_p)$. An object
$$\Sigma(\lambda, \underline{\mathscr{L}})^{\prime}\in \mathcal{D}^b\left(\mathrm{Mod}_{D(\mathrm{GL}_3(\mathbb{Q}_p), E)}\right)$$
(one should not confuse this notation $\Sigma(\lambda, \underline{\mathscr{L}})^{\prime}$ borrowed directly from \cite{Schr11} with our notation $\Sigma^+(\lambda, \underline{\mathscr{L}})$ before Lemma~\ref{3lemm: ext13}) has been constructed in \cite{Schr11} and plays a key role in Theorem~1.2 of \cite{Schr11}. An interesting feature of \cite{Schr11} is the appearance of the $p$\text{-}adic dilogarithm function in the construction of $\Sigma(\lambda, \underline{\mathscr{L}})^{\prime}$ in Definition~5.19 of \cite{Schr11}. Roughly, the object $\Sigma(\lambda, \underline{\mathscr{L}})^{\prime}$ was constructed from the choice of an element in $\mathrm{Ext}^2_{\mathrm{GL}_3(\mathbb{Q}_p), \lambda}\left(\overline{L}(\lambda),~\Sigma(\lambda, \mathscr{L}_1, \mathscr{L}_2)\right)$ together with general formal arguments in triangulated categories ( cf. Proposition~3.2 of \cite{Schr11}). In particular, $\Sigma(\lambda, \underline{\mathscr{L}})^{\prime}$ fits into the following distinguished triangle:
$$F_\lambda^{\prime}\longrightarrow ~\Sigma(\lambda, \underline{\mathscr{L}})^{\prime}\longrightarrow \Sigma(\lambda, \mathscr{L}, \mathscr{L}^{\prime})^{\prime}[-1]\xrightarrow{~+1}$$
as illustrated in (5.99) of \cite{Schr11}. However, it was not clear in \cite{Schr11} whether there is an explicit complex $\left[C_{\bullet}\right]$ of locally analytic representations of $\mathrm{GL}_3(\mathbb{Q}_p)$ such that the object $$\mathcal{D}^{\prime}\in\mathcal{D}^b\left(\mathrm{Mod}_{D(\mathrm{GL}_3(\mathbb{Q}_p), E)}\right)$$
associated with $\left[C^{\prime}_{-\bullet}\right]$ satisfies
$$\mathcal{D}^{\prime}\cong \Sigma(\lambda, \underline{\mathscr{L}})^{\prime}\in \mathcal{D}^b\left(\mathrm{Mod}_{D(\mathrm{GL}_3(\mathbb{Q}_p), E)}\right).$$
Although our notation are slightly different from \cite{Schr11} in the sense that the notation $\Sigma(\lambda, \mathscr{L}, \mathscr{L}^{\prime})$ (resp. the notation $F_\lambda$) is replaced with $\Sigma(\lambda, \mathscr{L}_1, \mathscr{L}_2)$ (resp. with $\overline{L}(\lambda)$), we show that
\begin{theo}\label{3theo: complex introduction}[Proposition~\ref{3prop: relation with derived object}, (\ref{3sign of L invariant}) and Lemma~\ref{3lemm: easy normalization of higher invariant}]
The complex
\begin{equation}\label{3key complex}
\left[\left(\begin{xy}
(0,0)*+{\overline{L}(\lambda)\otimes_Ev_{P_{3-i}}^{\infty}}="a"; (20,0)*+{\overline{L}(\lambda)}="b";
{\ar@{-}"a";"b"};
\end{xy}\right)^{\prime}\longrightarrow\Sigma^{\sharp, +}_i(\lambda, \mathscr{L}_1, \mathscr{L}_2, \mathscr{L}_3)^{\prime} \right]
\end{equation}
realizes the object $\Sigma(\lambda, \underline{\mathscr{L}})^{\prime}$ where $\begin{xy}
(0,0)*+{\overline{L}(\lambda)\otimes_Ev_{P_{3-i}}^{\infty}}="a"; (20,0)*+{\overline{L}(\lambda)}="b";
{\ar@{-}"a";"b"};
\end{xy}$ is the unique non-split extension of $\overline{L}(\lambda)\otimes_Ev_{P_{3-i}}^{\infty}$ by $\overline{L}(\lambda)$ thanks to Proposition~\ref{3prop: locally algebraic extension}, $\Sigma^{\sharp, +}_i(\lambda, \mathscr{L}_1, \mathscr{L}_2, \mathscr{L}_3)$ is the locally analytic subrepresentation of $\Sigma^{\rm{min}}(\lambda, \mathscr{L}_1, \mathscr{L}_2, \mathscr{L}_3)$ of the form
$$
\begin{xy}
(0,0)*+{\mathrm{St}_3^{\rm{an}}(\lambda)}="a"; (20,4)*+{v_{P_i}^{\rm{an}}(\lambda)}="b"; (40,8)*+{C_{s_i,s_i}}="d"; (65,8)*+{\overline{L}(\lambda)\otimes_Ev_{P_{3-i}}^{\infty}}="f"; (20,-4)*+{v_{P_{3-i}}^{\rm{an}}(\lambda)}="c"; (40,-8)*+{C_{s_{3-i},s_{3-i}}}="e"; (40,0)*+{\overline{L}(\lambda)}="h";
{\ar@{-}"a";"b"}; {\ar@{-}"a";"c"}; {\ar@{-}"b";"d"}; {\ar@{-}"c";"e"}; {\ar@{-}"d";"f"}; {\ar@{-}"b";"h"}; {\ar@{-}"c";"h"};
\end{xy}
$$
and the invariants $\mathscr{L}_1, \mathscr{L}_2, \mathscr{L}_3\in E$ are determined by the formula
$$\mathscr{L}_1=-\mathscr{L}^{\prime},~\mathscr{L}_2=-\mathscr{L},~\mathscr{L}_3=\gamma(\mathscr{L}^{\prime\prime}-\frac{1}{2}\mathscr{L}\mathscr{L}^{\prime})$$
with the constant $\gamma\in E^{\times}$ defined in Lemma~\ref{3lemm: ext2 prime}.
\end{theo}
\begin{rema}\label{3rema: canonical isomorphism}
Strictly speaking, the complex (\ref{3key complex}) realizes an object in $\mathcal{D}^b\left(\mathrm{Mod}_{D(\mathrm{GL}_3(\mathbb{Q}_p), E)}\right)$ characterized by an element in
$$\mathrm{Ext}^2_{\mathrm{GL}_3(\mathbb{Q}_p), \lambda}\left(\overline{L}(\lambda),~\Sigma^{\sharp, +}(\lambda, \mathscr{L}_1, \mathscr{L}_2)\right)$$
due to formal arguments from Proposition~3.2 of \cite{Schr11}. However, we can prove that there is a canonical isomorphism
$$\mathrm{Ext}^2_{\mathrm{GL}_3(\mathbb{Q}_p), \lambda}\left(\overline{L}(\lambda),~\Sigma(\lambda, \mathscr{L}_1, \mathscr{L}_2)\right)\xrightarrow{\sim}\mathrm{Ext}^2_{\mathrm{GL}_3(\mathbb{Q}_p), \lambda}\left(\overline{L}(\lambda),~\Sigma^{\sharp, +}(\lambda, \mathscr{L}_1, \mathscr{L}_2)\right)$$
and hence we can equally say that (\ref{3key complex}) realizes $\Sigma(\lambda, \underline{\mathscr{L}})^{\prime}$ for a suitable normalization of notation as $\Sigma(\lambda, \underline{\mathscr{L}})$ has been constructed by choosing a non-zero element in $\mathrm{Ext}^2_{\mathrm{GL}_3(\mathbb{Q}_p), \lambda}\left(\overline{L}(\lambda),~\Sigma(\lambda, \mathscr{L}, \mathscr{L}^{\prime})\right)$ via Proposition~3.2 of \cite{Schr11}. Note that we have
$$\Sigma(\lambda, \mathscr{L}, \mathscr{L}^{\prime})\cong\Sigma(\lambda, \mathscr{L}_1, \mathscr{L}_2)$$
by (\ref{3normalization of notation}).
\end{rema}
\subsection{Higher $\mathscr{L}$\text{-}invariants for $\mathrm{GL}_3(\mathbb{Q}_p)$}\label{3subsection: criterion of global embedding}
It follows from (\ref{3main picture for min}) and (\ref{3main picture for Ext1}) that $\Sigma^{\rm{min},+}(\lambda, \mathscr{L}_1, \mathscr{L}_2, \mathscr{L}_3)$ can be described more precisely by the following picture:
$$
\begin{xy}
(-20,0)*+{\overline{L}(\lambda)\otimes_E\mathrm{St}_3^{\infty}}="a1";
(-6,9)*+{C^2_{s_1,1}}="a2"; (13.5,22)*+{C^1_{s_2s_1,1}}="a4"; (51,21)*+{C^2_{s_2s_1,1}}="a6"; (20,14)*+{\overline{L}(\lambda)\otimes_Ev_{P_1}^{\infty}}="b1"; (70,12)*+{C^1_{s_2,1}}="b2"; (42,30)*+{C_{s_1,s_1}}="d1"; (77,22)*+{\overline{L}(\lambda)\otimes_Ev_{P_2}^{\infty}}="f1"; (-6,-9)*+{C^2_{s_2,1}}="a3";(13.5,-22)*+{C^1_{s_1s_2,1}}="a5"; (51,-21)*+{C^2_{s_1s_2,1}}="a7"; (20,-14)*+{\overline{L}(\lambda)\otimes_Ev_{P_2}^{\infty}}="c1"; (70,-12)*+{C^1_{s_1,1}}="c2"; (42,-30)*+{C_{s_2,s_2}}="e1"; (77,-22)*+{\overline{L}(\lambda)\otimes_Ev_{P_1}^{\infty}}="g1"; (100,5)*+{\overline{L}(\lambda)^1}="h1"; (100,-5)*+{\overline{L}(\lambda)^2}="h2"; (73,-0)*+{C_{s_1s_2s_1,1}}="a8"; (75,35)*+{C^1_{s_2s_1,s_2s_1}}="i1"; (75,-35)*+{C^1_{s_1s_2,s_1s_2}}="j1"; (100,30)*+{C^2_{s_1,s_1s_2}}="i2"; (100,-30)*+{C^2_{s_2,s_2s_1}}="j2";
{\ar@{-}"a1";"a2"}; {\ar@{-}"a2";"a4"}; {\ar@{-}"a2";"b1"}; {\ar@{-}"a2";"a7"}; {\ar@{--}"a4";"b2"}; {\ar@{-}"a4";"a6"}; {\ar@{-}"a4";"d1"}; {\ar@{--}"a6";"c2"}; {\ar@{--}"a6";"a8"}; {\ar@{-}"b1";"b2"}; {\ar@{-}"b1";"d1"}; {\ar@{-}"b2";"h1"}; {\ar@{-}"d1";"f1"}; {\ar@{-}"f1";"h2"}; {\ar@{-}"f1";"h1"}; {\ar@{-}"a1";"a3"}; {\ar@{-}"a3";"a5"}; {\ar@{-}"a3";"c1"}; {\ar@{-}"a3";"a6"}; {\ar@{--}"a5";"c2"}; {\ar@{-}"a5";"e1"}; {\ar@{-}"a5";"a7"}; {\ar@{--}"a7";"b2"}; {\ar@{--}"a7";"a8"}; {\ar@{-}"c1";"c2"}; {\ar@{-}"c1";"e1"}; {\ar@{-}"c2";"h2"}; {\ar@{-}"e1";"g1"}; {\ar@{-}"g1";"h1"}; {\ar@{-}"g1";"h2"}; {\ar@{-}"d1";"i1"}; {\ar@{-}"e1";"j1"}; {\ar@{-}"i1";"i2"}; {\ar@{-}"j1";"j2"}; {\ar@{-}"f1";"i2"}; {\ar@{-}"g1";"j2"};
\end{xy}
$$
and therefore contains a unique subrepresentation of the form
$$
\begin{xy}
(0,0)*+{\overline{L}(\lambda)\otimes_E\mathrm{St}_3^{\infty}}="a1";
(20,8)*+{C^2_{s_1,1}}="a2"; (45,12)*+{C^1_{s_2s_1,1}}="a4"; (45,4)*+{\overline{L}(\lambda)\otimes_Ev_{P_1}^{\infty}}="b1"; (70,8)*+{C_{s_1,s_1}}="d1"; (95,4)*+{\overline{L}(\lambda)\otimes_Ev_{P_2}^{\infty}}="f1"; (20,-8)*+{C^2_{s_2,1}}="a3";(45,-12)*+{C^1_{s_1s_2,1}}="a5"; (45,-4)*+{\overline{L}(\lambda)\otimes_Ev_{P_2}^{\infty}}="c1"; (70,-8)*+{C_{s_2,s_2}}="e1"; (95,-4)*+{\overline{L}(\lambda)\otimes_Ev_{P_1}^{\infty}}="g1"; (95,12)*+{C^1_{s_2s_1,s_2s_1}}="i1"; (95,-12)*+{C^1_{s_1s_2,s_1s_2}}="j1"; (120,8)*+{C^2_{s_1,s_1s_2}}="i2"; (120,-8)*+{C^2_{s_2,s_2s_1}}="j2";
{\ar@{-}"a1";"a2"}; {\ar@{-}"a2";"a4"}; {\ar@{-}"a2";"b1"}; {\ar@{-}"a4";"d1"}; {\ar@{-}"b1";"d1"}; {\ar@{-}"d1";"f1"}; {\ar@{-}"a1";"a3"}; {\ar@{-}"a3";"a5"}; {\ar@{-}"a3";"c1"}; {\ar@{-}"a5";"e1"}; {\ar@{-}"c1";"e1"}; {\ar@{-}"e1";"g1"}; {\ar@{-}"d1";"i1"}; {\ar@{-}"e1";"j1"}; {\ar@{-}"i1";"i2"}; {\ar@{-}"j1";"j2"}; {\ar@{-}"f1";"i2"}; {\ar@{-}"g1";"j2"};
\end{xy}
$$
which is denoted by
\begin{equation}\label{3representation ext 1}
\begin{xy}
(0,0)*+{\overline{L}(\lambda)\otimes_E\mathrm{St}_3^{\infty}}="a"; (30,4)*+{\Pi^1(\underline{k}, \underline{D})}="b"; (30,-4)*+{\Pi^2(\underline{k}, \underline{D})}="c";
{\ar@{-}"a";"b"}; {\ar@{-}"a";"c"};
\end{xy}
\end{equation}
in Theorem~1.1 of \cite{Bre17}. It follows from Theorem~1.2 of \cite{Bre17} that
$$\mathrm{dim}_E\mathrm{Ext}^1_{\mathrm{GL}_3(\mathbb{Q}_p),\lambda}\left(\Pi^i(\underline{k}, \underline{D}),~\overline{L}(\lambda)\otimes_E\mathrm{St}_3^{\infty}\right)=3$$
for $i=1,2$, and therefore a locally analytic representation of the form (\ref{3representation ext 1}) depends on four invariants. On the other hand, by a computation of extensions of rank one $(\varphi, \Gamma)$\text{-}modules we know that $\rho_p$ depends on three invariants. As a result, Theorem~1.1 of \cite{Bre17} predicts that not all representations of the form (\ref{3representation ext 1}) can be embedded into $\widehat{S}(U^p, E)^{\rm{an}}[\mathfrak{m}_{\rho}]$ for a certain pair of $U^p$ and $\rho_p$. This is actually the case as we show that
\begin{theo}\label{3theo: relation introduction}[Corollary~\ref{3coro: criterion}]
If a locally analytic representation $\Pi$ of the form (\ref{3representation ext 1}) can be embedded into $\widehat{S}(U^p, E)^{\rm{an}}[\mathfrak{m}_{\rho}]$ for a certain pair of $U^p$ and $\rho_p$, then it can be embedded into
$$\Sigma^{\rm{min},+}(\lambda, \mathscr{L}_1, \mathscr{L}_2, \mathscr{L}_3)$$
for a unique choice of $\mathscr{L}_1, \mathscr{L}_2, \mathscr{L}_3\in E$ determined by $\Pi$.
\end{theo}
\subsection{Sketch of content}\label{3subsection: sketch}
Section~\ref{3section: preliminary} recalls various well-known facts around locally analytic representations and our notation for a family of specific irreducible subquotients of locally analytic principal series to be used in the rest of the article. We emphasize that our definition of various $\mathrm{Ext}$\text{-}groups follows \cite{Bre17} closely and the only difference is that we use the dual notation compared to that of \cite{Bre17}. We also recall the $p$\text{-}adic dilogarithm function from Section~5.3 of \cite{Schr11} which is part of the main motivation of this article to relate \cite{Schr11} with \cite{Bre17} and \cite{BD18}.

Section~\ref{3section: GL2Qp} proves a crucial fact (Proposition~\ref{3prop: key result}) on the non-existence of locally analytic representations of $\mathrm{GL}_2(\mathbb{Q}_p)$ of a certain specific form using arguments involving infinitesimal characters of locally analytic representations. We learn such arguments essentially from Y. Ding.

Section~\ref{3section: computation} is a collection of various computational results necessary for the applications in Section~\ref{3section: exact sequence min}. These computations essentially make use of the formula in Section~5.2 and 5.3 of \cite{Bre17}.

Section~\ref{3section: technial min} serves as the preparation of Section~\ref{3section: exact sequence min} for the construction of $\Sigma^{\rm{min}}(\lambda, \mathscr{L}_1, \mathscr{L}_2, \mathscr{L}_3)$. It makes full use of the computational results from Section~\ref{3section: computation} to compute the dimension of various more complicated $\mathrm{Ext}$\text{-}groups to be crucially used in various important long exact sequences in Section~\ref{3section: exact sequence min}( cf. Lemma~\ref{3lemm: upper bound} and Proposition~\ref{3prop: main dim}).

Section~\ref{3section: exact sequence min} finishes the construction of $\Sigma^{\rm{min}}(\lambda, \mathscr{L}_1, \mathscr{L}_2, \mathscr{L}_3)$ as well as $\Sigma^{\rm{min},+}(\lambda, \mathscr{L}_1, \mathscr{L}_2, \mathscr{L}_3)$. Moreover, the construction of $\Sigma^{\rm{min}}(\lambda, \mathscr{L}_1, \mathscr{L}_2, \mathscr{L}_3)$ leads naturally to the construction of an explicit complex as in Theorem~\ref{3theo: complex introduction} that realizes the derived object $\Sigma(\lambda, \underline{\mathscr{L}})$ constructed in \cite{Schr11}.

Section~\ref{3section: local-global} finishes the proof of Theorem~\ref{3theo: main} by directly mimicking arguments from the proof of Theorem~6.2.1 of \cite{Bre17}. In particular, we give a purely representation theoretic criterion for a representation of the form (\ref{3representation ext 1}) to embed into completed cohomology as mentioned in Theorem~\ref{3theo: relation introduction}.
\subsection{Acknowledgement}\label{3subsection: acknowledgement}
The author expresses his gratefulness to Christophe Breuil for introducing the problem of relating \cite{Schr11} with \cite{Bre17} and \cite{BD18} and especially for his interest on the role played by the $p$\text{-}adic dilogarithm function. The author also benefited a lot from countless discussions with Y. Ding especially for Section~\ref{3section: GL2Qp} of this article. Finally, the author thanks B. Schraen for his beautiful thesis which improved the author's understanding on the subject.
\section{Preliminary}\label{3section: preliminary}
\subsection{Locally analytic representations}\label{3subsection: locally analytic rep}
In this section, we recall the definition of some well-known objects in the theory of locally analytic representations of $p$-adic reductive groups.

We fix a locally $\mathbb{Q}_p$-analytic group $H$ and denote the algebra of locally $\mathbb{Q}_p$-analytic distribution with coefficient $E$ on $H$ by $\mathcal{D}(H,E)$, which is defined as the strong dual of the locally convex $E$-vector space $C^{\rm{an}}(H,E)$ consisting of locally $\mathbb{Q}_p$-analytic functions on $H$. We use the notation $\mathrm{Rep}^{\rm{la}}_{H, E}$ (resp. $\mathrm{Rep}^{\infty}_{H, E}$) for the additve category consisting of locally $\mathbb{Q}_p$-analytic representations of $H$ (resp. smooth representations of $H$) with coefficient $E$. Therefore taking strong dual induces a fully faithful contravariant functor from $\mathrm{Rep}^{\rm{la}}_{H, E}$ to the abelian category $\mathrm{Mod}_{\mathcal{D}(H,E)}$ of abstract modules over $\mathcal{D}(H,E)$. The $E$-vector space $\mathrm{Ext}^i_{\mathcal{D}(H,E)}(M_1,M_2)$ is well-defined for any two objects $M_1, M_2\in \mathrm{Mod}_{\mathcal{D}(H,E)}$, and therefore we define
$$\mathrm{Ext}^i_H(\Pi_1,\Pi_2):=\mathrm{Ext}^i_{\mathcal{D}(H,E)}(\Pi_2^{\prime},\Pi_1^{\prime})$$
for any two objects $\Pi_1, \Pi_2\in \mathrm{Rep}^{\rm{la}}_{H, E}$ where $\cdot^{\prime}$ is the notation for strong dual. We also define the cohomology of an object $M\in \mathrm{Mod}_{\mathcal{D}(H,E)}$ by
$$H^i(H, M):=\mathrm{Ext}^i_{\mathcal{D}(H,E)}(1,M)$$
where $1$ is the strong dual of the trivial representation of $H$. If $H^{\prime}$ is a closed locally $\mathbb{Q}_p$-analytic normal subgroup of $H$, then $H/H^{\prime}$ is also a locally $\mathbb{Q}_p$-analytic group. It follows from the fact
$$D(H, E)\otimes_{D(H^{\prime},E)}E\cong D(H/H^{\prime},E)$$
(see Section 5.1 of \cite{Bre17} for example) that $H^i(H^{\prime}, M)$ admits a structure of $\mathcal{D}(H/H^{\prime},E)$-module for each $M\in \mathrm{Mod}_{\mathcal{D}(H,E)}$. We define the $H^{\prime}$-homology of $\Pi\in\mathrm{Rep}^{\rm{la}}_{H, E}$ as the object (if it exists up to isomorphism) $H_i(H^{\prime}, \Pi)\in\mathrm{Rep}^{\rm{la}}_{H/H^{\prime}, E}$ such that
$$H_i(H^{\prime}, \Pi)^{\prime}\cong H^i(H^{\prime}, \Pi^{\prime}).$$
We emphasize that $H_i(H^{\prime}, \Pi)$ is well defined in the sense above only after we know its existence. We fix a subgroup $Z$ of the center of the group $H$, then the algebra $\mathcal{D}(Z,E)$ consisting of locally $\mathbb{Q}_p$-analytic distribution with coefficient $E$ on $Z$ is naturally contained in the center of $\mathcal{D}(H,E)$. For each locally $\mathbb{Q}_p$-analytic $E$-character $\chi$ of $Z$, we can define the abelian subcategory $\mathrm{Mod}_{\mathcal{D}(H,E), \chi^{\prime}}$ consisting of all the objects in $\mathrm{Mod}_{\mathcal{D}(H,E)}$ on which $\mathcal{D}(Z,E)$ acts by $\chi^{\prime}$. Then we consider the functors $\mathrm{Ext}^i_{\mathcal{D}(H,E)}(-,-)$ defined as $\mathrm{Ext}^i_{\mathrm{Mod}_{\mathcal{D}(H,E),\chi^{\prime}}}(-,-)$ which are extensions inside the abelian category $\mathrm{Mod}_{\mathcal{D}(H,E),\chi^{\prime}}$. Consequently we can define
$$\mathrm{Ext}^i_{H, \chi}(\Pi_1,\Pi_2):=\mathrm{Ext}^i_{\mathcal{D}(H,E), \chi^{\prime}}(\Pi_2^{\prime},\Pi_1^{\prime})$$
for any two objects $\Pi_1, \Pi_2\in \mathrm{Rep}^{\rm{la}}_{H, E}$ such that $\Pi_1^{\prime}, \Pi_2^{\prime} \in \mathrm{Mod}_{\mathcal{D}(H,E),\chi^{\prime}}$. In particular, if $Z$ is the center of $H$ and acts on $\Pi\in \mathrm{Rep}^{\rm{la}}_{H, E}$ via the character $\chi$, then $\Pi^{\prime}\in \mathrm{Mod}_{\mathcal{D}(H,E),\chi^{\prime}}$, and we usually say that $\Pi$ admits a central character $\chi$.

Assume now $H$ is the set of $\mathbb{Q}_p$-points of a split reductive group over $\mathbb{Q}_p$. We recall the category $\mathcal{O}$ together with its subcategory $\mathcal{O}^{\mathfrak{p}}_{\rm{alg}}$ for each parabolic subgroup $P\subseteq H$ from Section~9.3 of \cite{Hum08} or \cite{OS15}. The construction by Orlik--Strauch in \cite{OS15} gives us a functor
$$\mathcal{F}_P^H : \mathcal{O}^{\mathfrak{p}}_{\rm{alg}}\times \mathrm{Rep}^{\infty}_{L, E}\rightarrow \mathrm{Rep}^{\rm{la}}_{H, E}$$
for each parabolic subgroup $P\subseteq H$ with Levi quotient $L$. We use the notation $\mathrm{Rep}^{\mathcal{OS}}_{H, E}$ for the abelian full subcategory of $\mathrm{Rep}^{\rm{la}}_{H, E}$ generated by the image of $\mathcal{F}_P^H$ when $P$ varies over all possible parabolic subgroups of $H$. Here we say a full subcategory is generated by a family of objects if it is the minimal full subcategory that contains these objects and is stable under extensions.
\subsection{Formal properties}\label{3subsection: formal devissage}
In this section, we recall and prove some general formal properties of locally analytic representations of $p$-adic reductive groups.

We fix a split $p$-adic reductive group $H$ and a parabolic subgroup $P$ of $H$. We use the notation $N$ for the unipotent radical of $P$ and fix a Levi subgroup $L$ of $P$.
\begin{lemm}\label{3lemm: cohomology devissage}
We have a spectral sequece
$$\mathrm{Ext}^j_{L,\ast}\left(H_k(N,~\Pi_1),~\Pi_2\right)\Rightarrow\mathrm{Ext}^{j+k}_{H,\ast}\left(\Pi_1,~\mathrm{Ind}_P^H\left(\Pi_2\right)^{\rm{an}}\right).$$
which implies an isomorphism
$$\mathrm{Hom}_{L,\ast}\left(H_0(N,~\Pi_1),~\Pi_2\right)\xrightarrow{\sim}\mathrm{Hom}_{H,\ast}\left(\Pi_1,~\mathrm{Ind}_P^H\left(\Pi_2\right)^{\rm{an}}\right)$$
and a long exact sequence
\begin{multline*}
\mathrm{Ext}^1_{L,\ast}\left(H_0(N,~\Pi_1),~\Pi_2\right)\hookrightarrow\mathrm{Ext}^1_{H,\ast}\left(\Pi_1,~\mathrm{Ind}_P^H\left(\Pi_2\right)^{\rm{an}}\right)\\
\rightarrow\mathrm{Hom}_{L,\ast}\left(H_1(N,~\Pi_1),~\Pi_2\right)\rightarrow\mathrm{Ext}^2_{L,\ast}\left(H_0(N,~\Pi_1),~\Pi_2\right)
\end{multline*}
for each $\Pi_1\in \mathrm{Rep}^{\rm{la}}_{H, E}$, $\Pi_2\in \mathrm{Rep}^{\rm{la}}_{L, E}$ satisfying the ($\mathrm{FIN}$) condition in Section 6 of \cite{ST05}, $\ast\in\{\varnothing, \chi\}$ where $\chi$ is a locally analytic character of the center of $H$.
\end{lemm}
\begin{proof}
This follows directly from our definition of $\mathrm{Ext}^k$ and $H_k$ in Section~\ref{3subsection: locally analytic rep} for $k\geq 0$, the original dual version in (44) and (45) of \cite{Bre17}.
%and the fact that objects in $\mathcal{OS}_{L_i(\mathbb{Q}_p)}$ satisfies the ($\mathrm{FIN}$) condition in Section 6 of \cite{ST05} according to \red{\cite{}}.
\end{proof}

We fix a Borel subgroup $B\subseteq H$ together with its opposite Borel subgroup $\overline{B}$. We fix an irreducible object $M\in\mathcal{O}^{\overline{\mathfrak{b}}}_{\rm{alg}}$. We choose a parabolic subgroup $P\subseteq H$ such that $P$ is maximal among all the parabolic subgroups $Q\subseteq H$ such that $M\in\mathcal{O}^{\overline{\mathfrak{q}}}_{\rm{alg}}$ where $\overline{\mathfrak{q}}$ is the Lie algebra of the opposite parabolic subgroup $\overline{Q}$ associated with $Q$. We fix a smooth irreducible representation $\pi^{\infty}$ of $L$ and a smooth character $\delta$ of $H$. We know that \cite{OS15} constructed an irreducible locally analytic representation
$$\mathcal{F}_P^H(M,~\pi^{\infty})$$
of $H$.
\begin{lemm}\label{3lemm: det twist}
The functor
$$-\otimes_E\delta$$
induces an equivalence of category from $\mathrm{Rep}^{\rm{la}}_{H, E}$ to itself. Moreover, the restriction of $-\otimes_E\delta$ to the subcategory $\mathrm{Rep}^{\mathcal{OS}}_{H, E}$ is again an equivalence of category to itself and satisfies
\begin{equation}\label{3smooth det twist}
\mathcal{F}_P^H(M,~\pi^{\infty})\otimes_E\delta\cong \mathcal{F}_P^H(M,~\pi^{\infty}\otimes_E\delta|_L)
\end{equation}
for each irreducible object $\mathcal{F}_P^H(M,~\pi^{\infty})\in \mathrm{Rep}^{\mathcal{OS}}_{H, E}$.
\end{lemm}
\begin{proof}
The functor $-\otimes_E\delta$ is clearly an equivalence of category from $\mathrm{Rep}^{\rm{la}}_{H, E}$ to itself with quasi-inverse given by
$$-\otimes_E\delta^{-1}.$$
It is sufficient to prove the formula (\ref{3smooth det twist}) to finish the proof. First of all, we notice by formal reason (equivalence of category) that $\mathcal{F}_P^H(M,~\pi^{\infty})\otimes_E\delta$ is an irreducible object in $\mathrm{Rep}^{\rm{la}}_{H, E}$ since $\mathcal{F}_P^H(M,~\pi^{\infty})$ is. We use the notation $\overline{\mathfrak{n}}$ for the Lie algebra associated with the unipotent radical $\overline{N}$ of the opposite parabolic subgroup $\overline{P}$ of $P$. We define $M_L$ as the (finite dimensional) algebraic representation of $L$ whose dual is isomorphic to $M^{\overline{\mathfrak{n}}}$ as a representation of $\mathfrak{l}$ and note that we have a surjection
$$
U(\mathfrak{h})\otimes_{U(\overline{\mathfrak{p}})}M^{\overline{\mathfrak{n}}}\twoheadrightarrow M.
$$
We observe that $N$ acts trivially on $\delta$, and therefore we have
$$H_0\left(N,~\mathcal{F}_P^H(M,~\pi^{\infty})\otimes_E\delta\right)\cong H_0\left(N,~\mathcal{F}_P^H(M,~\pi^{\infty})\right)\otimes_E\delta|_L\twoheadrightarrow M_L\otimes_E\pi^{\infty}\otimes_E\delta|_L$$
which induces by Lemma~\ref{3lemm: cohomology devissage} a non-zero morphism
\begin{equation}
\mathcal{F}_P^H(M,~\pi^{\infty})\otimes_E\delta\rightarrow\mathrm{Ind}_P^H\left(M_L\otimes_E\pi^{\infty}\otimes_E\delta|_L\right)^{\rm{an}}\cong \mathcal{F}_P^H(U(\mathfrak{h})\otimes_{U(\overline{\mathfrak{p}})}M^{\overline{\mathfrak{n}}},~\pi^{\infty}\otimes_E\delta|_L).
\end{equation}
We finish the proof by the fact that $\mathcal{F}_P^H(M,~\pi^{\infty})\otimes_E\delta$ is irreducible and that
$$\mathcal{F}_P^H(M,~\pi^{\infty}\otimes_E\delta|_L)\cong\mathrm{soc}_H\left(\mathcal{F}_P^H(U(\mathfrak{h})\otimes_{U(\overline{\mathfrak{p}})}M^{\overline{\mathfrak{n}}},~\pi^{\infty}\otimes_E\delta|_L)\right).$$
due to Corollary~3.3 of \cite{Bre16}.
\end{proof}
We fix a finite length locally analytic representation $V\in\mathrm{Rep}^{\rm{la}}_{H, E}$ equipped with a increasing filtration of subrepresentations $\{\mathrm{Fil}_kV\}_{0\leq k\leq m}$ such that
$$\mathrm{Fil}_0(V)=0,~\mathrm{Fil}_m(V)=V\mbox{ and }\mathrm{gr}_{k+1}V:=\mathrm{Fil}_{k+1}V/\mathrm{Fil}_kV\neq 0\mbox{ for all }0\leq k\leq m-1.$$
Note that the assumption above automatically implies that
$$\ell(V)\geq m$$
where $\ell(V)$ is the length of $V$.
\begin{prop}\label{3prop: formal devissages}
Assume that $W$ is another object of $\mathrm{Rep}^{\rm{la}}_{H, E}$ and $\chi$ is a locally analytic character of the center of $H$.
\begin{enumerate}
\item If $\mathrm{Ext}^1_{H,\chi}\left(W, ~\mathrm{gr}_kV\right)=0$ for each $1\leq k\leq m$, then we have
$$\mathrm{Ext}^1_{H,\chi}\left(W, ~V\right)=0.$$
\item If there exists $1\leq k_0\leq m$ such that $\mathrm{Ext}^1_{H,\chi}\left(W, ~\mathrm{gr}_kV\right)=0$ for each $1\leq k\neq k_0\leq m$ and $\mathrm{dim}_E\mathrm{Ext}^1_{H,\chi}\left(W, ~\mathrm{gr}_{k_0}V\right)=1$, then we have
    $$\mathrm{dim}_E\mathrm{Ext}^1_{H,\chi}\left(W, ~V\right)\leq1;$$
    if moreover $\mathrm{Ext}^2_{H,\chi}\left(W, ~\mathrm{gr}_kV\right)=0$ for each $1\leq k\leq k_0-1$ and $\mathrm{Hom}_{H,\chi}\left(W, ~\mathrm{gr}_kV\right)=0$ for each $k_0+1\leq k\leq m$, then we have
$$\mathrm{dim}_E\mathrm{Ext}^1_{H,\chi}\left(W, V\right)=1.$$
\end{enumerate}
\end{prop}
\begin{proof}
The short exact sequence $\mathrm{Fil}_kV\hookrightarrow\mathrm{Fil}_{k+1}V\twoheadrightarrow\mathrm{gr}_{k+1}V$ induces a long exact sequence
$$\mathrm{Ext}^1_{H,\chi}\left(W, ~\mathrm{Fil}_kV\right)\rightarrow\mathrm{Ext}^1_{H,\chi}\left(W, ~\mathrm{Fil}_{k+1}V\right)\rightarrow\mathrm{Ext}^1_{H,\chi}\left(W, ~\mathrm{gr}_{k+1}V\right)$$
which implies
$$\mathrm{dim}_E\mathrm{Ext}^1_{H,\chi}\left(W, ~\mathrm{Fil}_{k+1}V\right)\leq \mathrm{dim}_E\mathrm{Ext}^1_{H,\chi}\left(W, ~\mathrm{Fil}_kV\right)+\mathrm{dim}_E\mathrm{Ext}^1_{H,\chi}\left(W, ~\mathrm{gr}_{k+1}V\right).$$
Therefore we finish the proof of part (i) and the first claim of part (ii) by induction on $k$ and the fact that $\mathrm{gr}_1V=\mathrm{Fil}_1V$.

It remains to show the second claim of part (ii). The same method as in the proof of part (i) shows that
\begin{equation}\label{3formal devissage1}
\mathrm{Ext}^1_{H,\chi}\left(W, ~\mathrm{Fil}_{k_0-1}V\right)=\mathrm{Ext}^2_{H,\chi}\left(W, ~\mathrm{Fil}_{k_0-1}V\right)=0
\end{equation}
and
\begin{equation}\label{3formal devissage2}
\mathrm{Ext}^1_{H,\chi}\left(W, ~V/\mathrm{Fil}_{k_0}V\right)=\mathrm{Hom}_{H,\chi}\left(W, ~V/\mathrm{Fil}_{k_0}V\right)=0
\end{equation}
The short exact sequence $\mathrm{Fil}_{k_0-1}V\hookrightarrow\mathrm{Fil}_{k_0}V\twoheadrightarrow\mathrm{gr}_{k_0}V$ induces the long exact sequence
\begin{multline*}
\mathrm{Ext}^1_{H,\chi}\left(W, ~\mathrm{Fil}_{k_0-1}V\right)\rightarrow\mathrm{Ext}^1_{H,\chi}\left(W, ~\mathrm{Fil}_{k_0}V\right)
\rightarrow\mathrm{Ext}^1_{H,\chi}\left(W, ~\mathrm{gr}_{k_0}V\right)\rightarrow \mathrm{Ext}^2_{H,\chi}\left(W, ~\mathrm{Fil}_{k_0-1}V\right)
\end{multline*}
which implies that
\begin{equation}\label{3formal devissage3}
\mathrm{dim}_E\mathrm{Ext}^1_{H,\chi}\left(W, ~\mathrm{Fil}_{k_0}V\right)=1
\end{equation}
by (\ref{3formal devissage1}). The short exact sequence $\mathrm{Fil}_{k_0}V\hookrightarrow V\twoheadrightarrow V/\mathrm{Fil}_{k_0}V$ induces the long exact sequence
\begin{multline*}
\mathrm{Hom}_{H,\chi}\left(W, ~V/\mathrm{Fil}_{k_0}V\right)\rightarrow\mathrm{Ext}^1_{H,\chi}\left(W, ~\mathrm{Fil}_{k_0}V\right)
\rightarrow\mathrm{Ext}^1_{H,\chi}\left(W, ~V\right)\rightarrow\mathrm{Ext}^1_{H,\chi}\left(W, ~V/\mathrm{Fil}_{k_0}V\right)
\end{multline*}
which finishes the proof by combining (\ref{3formal devissage2}) and (\ref{3formal devissage3}).
\end{proof}
\subsection{Some notation}\label{3subsection: main notation}
In this section, we are going to recall some standard notation for the $p$-adic reductive groups $\mathrm{GL}_2(\mathbb{Q}_p)$ and $\mathrm{GL}_3(\mathbb{Q}_p)$ as well as notation for some locally analytic representations of these groups.

We denote the lower-triangular Borel subgroup (resp. the diagonal maximal split torus) of $\mathrm{GL}_{2/\mathbb{Q}_p}$ by $B_2$ (resp. by $T_2$) and the unipotent radical of $B_2$ by $N_{\mathrm{GL}_2}$. We use the notation $s$ for the non-trivial element in the Weyl group of $\mathrm{GL}_2$. We fix a weight $\nu\in X(T_2)$ of $\mathrm{GL}_2$ of the following form
$$\nu=(\nu_1,\nu_2)\in\mathbb{Z}^2$$
which corresponds to an algebraic character of $T_2(\mathbb{Q}_p)$
$$\delta_{T_2,\nu}:=\left(\begin{array}{cc}
a&0\\
0&b\\
\end{array}\right)\mapsto a^{\nu_1}b^{\nu_2}.$$
We denote the upper-triangular Borel subgroup by $\overline{B_2}$. If $\nu$ is dominant with respect to $\overline{B_2}$, namely if $\nu_1\geq \nu_2$, we use the notation $\overline{L}_{\mathrm{GL}_2}(\nu)$ (resp. $L_{\mathrm{GL}_2}(-\nu)$) for the irreducible algebraic representation of $\mathrm{GL}_2(\mathbb{Q}_p)$ with highest weight $\nu$ (resp. $-\nu$) with respect to the positive roots determined by $\overline{B_2}$ (resp. $B_2$). In particular, $\overline{L}_{\mathrm{GL}_2}(\nu)$ and $L_{\mathrm{GL}_2}(-\nu)$ are the dual of each other. We use the shorten notation
$$I^{\mathrm{GL}_2}_{B_2}(\chi_{T_2}):=\left(\mathrm{Ind}_{B_2(\mathbb{Q}_p)}^{\mathrm{GL}_2(\mathbb{Q}_p)}\chi_{T_2}\right)^{\rm{an}}$$
for any locally analytic character $\chi_{T_2}$ of $T_2(\mathbb{Q}_p)$ and set
$$i^{\mathrm{GL}_2}_{B_2}(\chi_{T_2}):=\left(\mathrm{Ind}_{B_2(\mathbb{Q}_p)}^{\mathrm{GL}_2(\mathbb{Q}_p)}\chi^{\infty}_{T_2}\right)^{\infty}\otimes_E \overline{L}_{\mathrm{GL}_2}(\nu)$$
if $\chi_{T_2}=\delta_{T_2,\nu}\otimes_E\chi^{\infty}_{T_2}$ is locally algebraic where $\chi^{\infty}_{T_2}$ is a smooth character of $T_2(\mathbb{Q}_p)$. Then we define the locally analytic Steinberg representation as well as the smooth Steinberg representation for $\mathrm{GL}_2(\mathbb{Q}_p)$ as follows
$$\mathrm{St}^{\rm{an}}_2(\nu):=I^{\mathrm{GL}_2}_{B_2}(\delta_{T_2,\mu})/\overline{L}_{\mathrm{GL}_2}(\nu), ~ \mathrm{St}^{\infty}_2:=i^{\mathrm{GL}_2}_{B_2}(1_{T_2})/1_2$$
where $1_2$ (resp. $1_{T_2}$) is the trivial representation of $\mathrm{GL}_2(\mathbb{Q}_p)$ (resp. of $T_2(\mathbb{Q}_p)$).

We denote the lower-triangular Borel subgroup (resp. the diagonal maximal split torus) of $\mathrm{GL}_{3/\mathbb{Q}_p}$ by $B$ (resp. by $T$) and the unipotent radical of $B$ by $N$. We fix a weight $\lambda\in X(T)$ of $\mathrm{GL}_3$ of the following form
$$\lambda=(\lambda_1,\lambda_2, \lambda_3)\in\mathbb{Z}^3,$$
which corresponds to an algebraic character of $T(\mathbb{Q}_p)$
$$ \delta_{T,\lambda}:=\left(\begin{array}{ccc}
a&0&0\\
0&b&0\\
0&0&c\\
\end{array}\right)\mapsto a^{\lambda_1}b^{\lambda_2}c^{\lambda_3}.$$
We denote the center of $\mathrm{GL}_3$ by $Z$ and notice that $Z(\mathbb{Q}_p)\cong\mathbb{Q}_p^{\times}$. Hence the restriction of $\delta_{T,\lambda}$ to $Z(\mathbb{Q}_p)$ gives an algebraic character of $Z(\mathbb{Q}_p)$:
$$\delta_{Z,\lambda}:=\left(\begin{array}{ccc}
a&0&0\\
0&a&0\\
0&0&a\\
\end{array}\right)\mapsto a^{\lambda_1+\lambda_2+\lambda_3}.$$
We use the shorten notation
$$\mathrm{Ext}^i_{\ast, \lambda}(-,-):=\mathrm{Ext}^i_{\ast, \delta_{Z,\lambda}}(-,-)$$
for $\ast\in\{T(\mathbb{Q}_p), L_1(\mathbb{Q}_p), L_2(\mathbb{Q}_p), \mathrm{GL}_3(\mathbb{Q}_p)\}.$ In particular, the notation
$$\mathrm{Ext}^i_{\ast, 0}(-,-)$$
means (higher) extensions with the trivial central character. We denote the upper-triangular Borel subgroup of $\mathrm{GL}_3$ by $\overline{B}$. If $\lambda$ is dominant with respect to $\overline{B}$, namely if $\lambda_1\geq \lambda_2\geq\lambda_3$, we use the notation $\overline{L}(\lambda)$ (resp. $L(-\lambda)$) for the irreducible algebraic representation of $\mathrm{GL}_3(\mathbb{Q}_p)$ with highest weight $\lambda$ (resp. $-\lambda$) with respect to the positive roots determined by $\overline{B}$ (resp. $B$). In particular, $\overline{L}(\lambda)$ and $L(-\lambda)$ are dual of each other.
We use the notation $P_1:=\left(\begin{array}{ccc}
\ast&\ast&0\\
\ast&\ast&0\\
\ast&\ast&\ast\\
\end{array}\right)$ and $P_2:=\left(\begin{array}{ccc}
\ast&0&0\\
\ast&\ast&\ast\\
\ast&\ast&\ast\\
\end{array}\right)$ for the two standard maximal parabolic subgroups of $\mathrm{GL}_3$ with unipotent radical $N_1$ and $N_2$ respectively, and the notation $\overline{P_i}$ for the opposite parabolic subgroup of $P_i$ for $i=1,2$. We set
$$L_i:=P_i\cap \overline{P_i}$$
and set $s_i$ for the simple reflection in the Weyl group of $L_i$ for each $i=1,2$. In particular, the Weyl group $W$ of $\mathrm{GL}_3$ can be lifted to a subgroup of $\mathrm{GL}_3$ with the following elements
$$\{1, s_1, s_2, s_1s_2, s_2s_1, s_1s_2s_1\}.$$
We will usually use the shorten notation $N_i$ ( cf. Section~\ref{3section: computation}) for its set of $\mathbb{Q}_p$-points $N_i(\mathbb{Q}_p)$ if it does not cause any ambiguity. We use the notation $M(-\lambda)$ for the Verma module in $\mathcal{O}^{\mathfrak{b}}_{\rm{alg}}$ with highest weight $-\lambda$ (with respect to $B$) and simple quotient $L(-\lambda)$ for each $\lambda\in X(T)$ (not necessarily dominant). Similarly, we use the notation $M_i(-\lambda)$ for the parabolic Verma module in $\mathcal{O}^{\mathfrak{p}_i}_{\rm{alg}}$ with highest weight $-\lambda$ with respect to $B$ ( cf. Section~9.4 of \cite{Hum08}).  We define $\overline{L}_i(\lambda)$ as the irreducible algebraic representation of $L_i(\mathbb{Q}_p)$ with a highest weight $\lambda$ dominant with respect to $\overline{B}\cap L_i$. For example, if $\lambda\in X(T)_+$, then we know that $\lambda$, $s_i\cdot\lambda$ and $s_is_{3-i}\cdot\lambda$ are dominant with respect to $\overline{B}\cap L_{3-i}$ for $i=1,2$. We use the following notation for various parabolic inductions
$$I^{\mathrm{GL}_3}_{B}(\chi):=\left(\mathrm{Ind}_{B(\mathbb{Q}_p)}^{\mathrm{GL}_3(\mathbb{Q}_p)}\chi\right)^{\rm{an}}, ~I^{\mathrm{GL}_3}_{P_i}(\pi_i):=\left(\mathrm{Ind}_{P_i(\mathbb{Q}_p)}^{\mathrm{GL}_3(\mathbb{Q}_p)}\pi_i\right)^{\rm{an}}$$
if $\chi$ is an arbitrary locally analytic character of $T(\mathbb{Q}_p)$ and $\pi_i$ is an arbitrary locally analytic representation of $L_i(\mathbb{Q}_p)$ for each $i=1,2$. Moreover, we use the notation
$$i^{\mathrm{GL}_3}_{B}(\chi):=\left(\mathrm{Ind}_{B(\mathbb{Q}_p)}^{\mathrm{GL}_3(\mathbb{Q}_p)}\chi^{\infty}\right)^{\infty}\otimes_E \overline{L}(\lambda), ~i^{\mathrm{GL}_3}_{P_i}(\pi_i):=\left(\mathrm{Ind}_{P_i(\mathbb{Q}_p)}^{\mathrm{GL}_3(\mathbb{Q}_p)}\pi_i^{\infty}\right)^{\infty}\otimes_E \overline{L}(\lambda)$$
for $i=1,2$ if $\chi=\delta_{T,\lambda}\otimes_E\chi^{\infty}$ and $\pi_i=\overline{L}_i(\lambda)\otimes_E\pi_i^{\infty}$ are locally algebraic where $\chi^{\infty}$ (resp. $\pi_i^{\infty}$) is a smooth representation of $T(\mathbb{Q}_p)$ (resp. of $L_i(\mathbb{Q}_p)$). We will also use similar notation for parabolic induction to Levi subgroups such as $I_{B\cap L_i}^{L_i}$ and $i_{B\cap L_i}^{L_i}$ for $i=1,2$. Then we define the locally analytic (generalized) Steinberg representation as well as the smooth (generalized) Steinberg representation for $\mathrm{GL}_3(\mathbb{Q}_p)$ by
$$\mathrm{St}^{\rm{an}}_3(\lambda):=I^{\mathrm{GL}_3}_{B}(\delta_{T,\lambda})/\left(I^{\mathrm{GL}_3}_{P_1}(\overline{L}_1(\lambda))+I^{\mathrm{GL}_3}_{P_2}(\overline{L}_2(\lambda))\right), ~ \mathrm{St}^{\infty}_3:=i^{\mathrm{GL}_3}_{B}(1)/\left(i^{\mathrm{GL}_3}_{P_1}(1_{L_1})+i^{\mathrm{GL}_3}_{P_2}(1_{L_2})\right)$$
and
$$v^{\rm{an}}_{P_i}(\lambda):=I^{\mathrm{GL}_3}_{P_i}(\overline{L}_i(\lambda))/\overline{L}(\lambda), ~ v^{\infty}_{P_i}:=i^{\mathrm{GL}_3}_{P_i}(1_{L_i})/1_3$$
where $1_3$ (resp. $1_{L_i}$) is the trivial representation of $\mathrm{GL}_3(\mathbb{Q}_p)$ (resp. of $L_i(\mathbb{Q}_p)$ for each $i=1,2$).
We define the following smooth representations of $L_1(\mathbb{Q}_p)$:
$$\begin{array}{ccc}
\pi_{1,1}^{\infty}&:=&\mathrm{St}_2^{\infty}\otimes_E1\\
\pi_{1,2}^{\infty}&:=&i_{B_2}^{\mathrm{GL}_2}\left(1\otimes_E|\cdot|^{-1}\right)\otimes_E|\cdot|\\
\pi_{1,3}^{\infty}&:=&\left(\mathrm{St}_2^{\infty}\otimes_E(|\cdot|^{-1}\circ\mathrm{det}_2)\right)\otimes_E|\cdot|^2\\
\end{array}
$$
and the following smooth representations of $L_2(\mathbb{Q}_p)$:
$$
\begin{array}{ccc}
\pi_{2,1}^{\infty}&:=&1\otimes_E\mathrm{St}_2^{\infty}\\
\pi_{2,2}^{\infty}&:=&|\cdot|^{-1}\otimes_E i_{B_2}^{\mathrm{GL}_2}\left(|\cdot|\otimes_E1\right)\\
\pi_{2,3}^{\infty}&:=&|\cdot|^{-2}\otimes_E\left(\mathrm{St}_2^{\infty}\otimes_E(|\cdot|\circ\mathrm{det}_2)\right)\\
\end{array}
$$
Consequently, we can define the following locally analytic representations for $i=1,2$:
\begin{equation}\label{3irr rep I}
\begin{array}{cccccc}
C^1_{s_i,1}&:=&\mathcal{F}^{\mathrm{GL}_3}_{P_{3-i}}\left(L(-s_i\cdot\lambda),~1_{L_{3-i}}\right)&C^2_{s_i,1}&:=&\mathcal{F}^{\mathrm{GL}_3}_{P_{3-i}}\left(L(-s_i\cdot\lambda),~\pi_{i,1}^{\infty}\right)\\
C^1_{s_is_{3-i},1}&:=&\mathcal{F}^{\mathrm{GL}_3}_{P_{3-i}}\left(L(-s_is_{3-i}\cdot\lambda),~1_{L_{3-i}}\right)&C^2_{s_is_{3-i},1}&:=&\mathcal{F}^{\mathrm{GL}_3}_{P_{3-i}}\left(L(-s_is_{3-i}\cdot\lambda),~\pi_{i,1}^{\infty}\right)\\
C_{s_i,s_i}&:=&\mathcal{F}^{\mathrm{GL}_3}_{P_{3-i}}\left(L(-s_i\cdot\lambda),~\pi_{i,2}^{\infty}\right)&C_{s_is_{3-i},s_i}&:=&\mathcal{F}^{\mathrm{GL}_3}_{P_{3-i}}\left(L(-s_is_{3-i}\cdot\lambda),~\pi_{i,2}^{\infty}\right)\\
C^1_{s_i,s_is_{3-i}}&:=&\mathcal{F}^{\mathrm{GL}_3}_{P_{3-i}}\left(L(-s_i\cdot\lambda),~\mathfrak{d}_{P_{3-i}}^{\infty}\right)&C^2_{s_1,1}&:=&\mathcal{F}^{\mathrm{GL}_3}_{P_{3-i}}\left(L(-s_i\cdot\lambda),~\pi_{i,3}^{\infty}\right)\\
C^1_{s_is_{3-i},s_is_{3-i}}&:=&\mathcal{F}^{\mathrm{GL}_3}_{P_{3-i}}\left(L(-s_is_{3-i}\cdot\lambda),~\mathfrak{d}_{P_{3-i}}^{\infty}\right)&C^2_{s_is_{3-i},1}&:=&\mathcal{F}^{\mathrm{GL}_3}_{P_{3-i}}\left(L(-s_is_{3-i}\cdot\lambda),~\pi_{i,3}^{\infty}\right)
\end{array}
\end{equation}
where
$$\mathfrak{d}_{P_1}^{\infty}:=|\cdot|^{-1}\circ\mathrm{det}_2\otimes_E|\cdot|^2\mbox{ and }\mathfrak{d}_{P_2}^{\infty}:=|\cdot|^{-2}\otimes_E|\cdot|\circ\mathrm{det}_2.$$
We also define
\begin{equation}\label{3irr rep II}
C_{s_1s_2s_1,w}:=\mathcal{F}^{\mathrm{GL}_3}_B\left(L(-s_1s_2s_1\cdot\lambda),~\chi_w^{\infty}\right)
\end{equation}
for each $w\in W$ where
$$
\begin{xy}
(0,0)*+{\chi_1^{\infty}}; (6,0)*+{:=}; (12,0)*+{1_T}; (50,0)*+{\chi_{s_1}^{\infty}}; (56,0)*+{:=}; (73,0)*+{|\cdot|^{-1}\otimes_E|\cdot|\otimes_E1}; (105,0)*+{\chi_{s_2}^{\infty}}; (112,0)*+{:=}; (130,0)*+{1\otimes_E|\cdot|^{-1}\otimes_E|\cdot|}; (0,-6)*+{\chi_{s_1s_2}^{\infty}}; (6,-6)*+{:=}; (25,-6)*+{|\cdot|^{-2}\otimes_E|\cdot|\otimes_E|\cdot|}; (50,-6)*+{\chi_{s_2s_1}^{\infty}}; (56,-6)*+{:=}; (77,-6)*+{|\cdot|^{-1}\otimes_E|\cdot|^{-1}\otimes_E|\cdot|^2}; (105,-6)*+{\chi_{s_1s_2s_1}^{\infty}}; (112,-6)*+{:=}; (130,-6)*+{|\cdot|^{-2}\otimes_E1\otimes_E|\cdot|^2};
\end{xy}
$$
We notice that the representations considered in (\ref{3irr rep I}) and (\ref{3irr rep II}) are all irreducible objects inside $\mathrm{Rep}^{\mathcal{OS}}_{\mathrm{GL}_3(\mathbb{Q}_p), E}$ according to the main theorem of \cite{OS15}. We use the notation $\Omega$ for the set whose elements are listed as the following:
$$
\begin{array}{cccc}
\overline{L}(\lambda)&\overline{L}(\lambda)\otimes_Ev_{P_1}^{\infty}&\overline{L}(\lambda)\otimes_Ev_{P_2}^{\infty}&\overline{L}(\lambda)\otimes_E\mathrm{St}_3^{\infty}\\
C^1_{s_1,1}&C^2_{s_1,1}&C^1_{s_2,1}&C^2_{s_2,1}\\
C^1_{s_1s_2,1}&C^2_{s_1s_2,1}&C^1_{s_2s_1,1}&C^2_{s_2s_1,1}\\
C^1_{s_1,s_1s_2}&C^2_{s_1,s_1s_2}&C^1_{s_2,s_2s_1}&C^2_{s_2,s_2s_1}\\
C^1_{s_1s_2,s_1s_2}&C^2_{s_1s_2,s_1s_2}&C^1_{s_2s_1,s_2s_1}&C^2_{s_2s_1,s_2s_1}\\
C_{s_1,s_1}&C_{s_1s_2,s_1}&C_{s_2,s_2}&C_{s_2s_1,s_2}\\
C_{s_1s_2s_1,w}&w\in W&&
\end{array}
$$
\begin{rema}
It is actually possible to show that $\Omega$ is the set of (isomorphism classes of) irreducible objects of the block inside $\mathrm{Rep}^{\mathcal{OS}}_{\mathrm{GL}_3(\mathbb{Q}_p), E}$ containing the object $\overline{L}(\lambda)$.
\end{rema}
\begin{lemm}\label{3lemm: structure of St}
The representation $v^{\rm{an}}_{P_i}(\lambda)$ fits into a non-split extension
\begin{equation}\label{3generalized St picture}
\overline{L}(\lambda)\otimes_Ev_{P_i}^{\infty}\hookrightarrow v^{\rm{an}}_{P_i}(\lambda)\twoheadrightarrow C^1_{s_{3-i},1}
\end{equation}
for $i=1,2$. On the other hand, the representation $\mathrm{St}^{\rm{an}}_3(\lambda)$ has the following form:
\begin{equation}\label{3St picture}
\begin{xy}
(0,0)*+{\overline{L}(\lambda)\otimes_E\mathrm{St}_3^{\infty}}="a"; (20,6)*+{C^2_{s_1,1}}="b"; (20,-6)*+{C^2_{s_2,1}}="c"; (40,6)*+{C^1_{s_2s_1,1}}="d"; (40,-6)*+{C^1_{s_1s_2,1}}="e"; (60,6)*+{C^2_{s_2s_1,1}}="f"; (60,-6)*+{C^2_{s_1s_2,1}}="g"; (80,0)*+{C_{s_1s_2s_1,1}}="h";
{\ar@{-}"a";"b"}; {\ar@{-}"a";"c"}; {\ar@{-}"b";"d"}; {\ar@{-}"c";"e"}; {\ar@{-}"a";"b"}; {\ar@{-}"d";"f"}; {\ar@{-}"e";"g"}; {\ar@{-}"b";"g"}; {\ar@{-}"c";"f"}; {\ar@{--}"f";"h"}; {\ar@{--}"g";"h"};
\end{xy}.
\end{equation}
\end{lemm}
\begin{proof}
The non-split short exact sequence follows directly from (3.62) of \cite{BD18}. It follows easily from the definition of $\mathrm{St}^{\rm{an}}_3(\lambda)$ that
$$\mathrm{JH}_{\mathrm{GL}_3(\mathbb{Q}_p)}\left(\mathrm{St}^{\rm{an}}_3(\lambda)\right)=\{\overline{L}(\lambda)\otimes_E\mathrm{St}_3^{\infty},~C^2_{s_1,1},~C^2_{s_2,1},~C^1_{s_2s_1,1},~C^1_{s_1s_2,1},~C^2_{s_2s_1,1},~C^2_{s_1s_2,1},~C_{s_1s_2s_1,1}\}$$
and each Jordan--H\"older factor occurs with multiplicity one. It follows from Section~5.2 of \cite{Bre17} that
$$H_0\left(N_i,~\mathcal{F}^{\mathrm{GL}_3}_{P_i}\left(L(-s_{3-i}s_i\cdot\lambda),~i_{B\cap L_i}^{L_i}(1_T)\right)\right)=\overline{L}_i(-s_{3-i}s_i\cdot\lambda)\otimes_Ei_{B\cap L_i}^{L_i}(1_T)$$
which together with
$$\mathrm{JH}_{\mathrm{GL}_3(\mathbb{Q}_p)}\left(\mathcal{F}^{\mathrm{GL}_3}_{P_i}\left(L(-s_{3-i}s_i\cdot\lambda),~i_{B\cap L_i}^{L_i}(1_T)\right)\right)=\{C^1_{s_{3-i}s_i,1},~C^2_{s_{3-i}s_i,1}\}$$
imply that $\mathcal{F}^{\mathrm{GL}_3}_{P_i}\left(L(-s_{3-i}s_i\cdot\lambda),~i_{B\cap L_i}^{L_i}(1_T)\right)$ fits into a non-split extension
\begin{equation}\label{3non split ext in St}
C^1_{s_{3-i}s_i,1}\hookrightarrow\mathcal{F}^{\mathrm{GL}_3}_{P_i}\left(L(-s_{3-i}s_i\cdot\lambda),~i_{B\cap L_i}^{L_i}(1_T)\right)\twoheadrightarrow C^2_{s_{3-i}s_i,1}
\end{equation}
for $i=1,2$. We also observe from Section~5.2 and 5.3 of \cite{Bre17} that
$$H_2\left(N_{3-i},~\mathcal{F}^{\mathrm{GL}_3}_{P_i}\left(M_i(-s_{3-i}\cdot\lambda),~\pi_{i,1}^{\infty}\right)\right)\not\cong H_2(N_{3-i},~C^2_{s_{3-i},1})\oplus H_2(N_{3-i},~C^2_{s_{3-i}s_i,1})$$
which together with
$$\mathrm{JH}_{\mathrm{GL}_3(\mathbb{Q}_p)}\left(\mathcal{F}^{\mathrm{GL}_3}_{P_i}\left(M_i(-s_{3-i}\cdot\lambda),~\pi_{i,1}^{\infty}\right)\right)=\{C^2_{s_{3-i},1},~C^2_{s_{3-i}s_i,1}\}$$
imply that $\mathcal{F}^{\mathrm{GL}_3}_{P_i}\left(M_i(-s_{3-i}\cdot\lambda),~\pi_{i,1}^{\infty}\right)$ fits into a non-split extension
\begin{equation}\label{3non split ext in St prime}
C^2_{s_{3-i},1}\hookrightarrow\mathcal{F}^{\mathrm{GL}_3}_{P_i}\left(M_i(-s_{3-i}\cdot\lambda),~\pi_{i,1}^{\infty}\right)\twoheadrightarrow C^2_{s_{3-i}s_i,1}
\end{equation}
for $i=1,2$. We notice that both $\mathcal{F}^{\mathrm{GL}_3}_{P_i}\left(L(-s_{3-i}s_i\cdot\lambda),~i_{B\cap L_i}^{L_i}(1_T)\right)$ and $\mathcal{F}^{\mathrm{GL}_3}_{P_i}\left(M_i(-s_{3-i}\cdot\lambda),~\pi_{i,1}^{\infty}\right)$ are subquotients of $\mathrm{St}^{\rm{an}}_3(\lambda)$ by various properties of the functors $\mathcal{F}^{\mathrm{GL}_3}_{P_i}$ ( cf. main theorem of \cite{OS15}) and the definition of $\mathrm{St}^{\rm{an}}_3(\lambda)$. We finish the proof by combining (\ref{3non split ext in St}) and (\ref{3non split ext in St prime}) with the results before Remark~3.38 of \cite{BD18}.
\end{proof}
\begin{rema}
It is actually possible to show that all the possibly non-split extensions indicated in (\ref{3St picture}) are non-split, although they are essentially unrelated to the $p$-adic dilogarithm function.
\end{rema}

\subsection{$p$\text{-}adic logarithm and dilogarithm}\label{3subsection: log dilog}
In this section, we recall $p$\text{-}adic logarithm and dilogarithm function as well as their representation theoretic interpretations.

We recall the $p$\text{-}adic logarithm function $\mathrm{log}_0: \mathbb{Q}_p^{\times}\rightarrow \mathbb{Q}_p$ defined by power series on a open subgroup of $\mathbb{Z}_p^{\times}$ and then extended to $\mathbb{Q}_p^{\times}$ by the condition $\mathrm{log}_0(p)=0$. We also recall the $p$\text{-}adic valuation function $\mathrm{val}_p: \mathbb{Q}_p^{\times}\rightarrow \mathbb{Z}$ satisfying $|\cdot|=p^{-\mathrm{val}_p(\cdot)}$ (and in particular $\mathrm{val}_p(p)=1$). We notice that
$$\{\mathrm{log}_0,~\mathrm{val}_p\}$$
forms a basis of the two dimensional $E$\text{-}vector space
$$\mathrm{Hom}_{\rm{cont}}\left(\mathbb{Q}_p^{\times},~E\right).$$
We define $\mathrm{log}_{\mathscr{L}}:=\mathrm{log}_0-\mathscr{L}\mathrm{val}_p$ for each $\mathscr{L}\in E$ and consider the following two dimensional locally analytic representation of $\mathbb{Q}_p^{\times}$
$$V_{\mathscr{L}}: \mathbb{Q}_p^{\times}\rightarrow B_2(E), ~a\mapsto
\left(\begin{array}{cc}
1&\mathrm{log}_{\mathscr{L}}(a)\\
0&1\\
\end{array}\right) $$
and therefore
\begin{equation}\label{3extension of two trivial reps}
\mathrm{soc}_{\mathbb{Q}_p^{\times}}(V_{\mathscr{L}})=\mathrm{cosoc}_{\mathbb{Q}_p^{\times}}(V_{\mathscr{L}})=1
\end{equation}
where $1$ is the notation for the trivial character of $\mathbb{Q}_p^{\times}$. We notice that
$$\mathrm{Ext}^1_{\mathbb{Q}_p^{\times}}(1,1)\cong \mathrm{Hom}_{\rm{cont}}\left(\mathbb{Q}_p^{\times},~E\right),$$
by a standard fact in (continuous) group cohomology and therefore the set $\{V_{\mathscr{L}}\mid \mathscr{L}\in E\}$ exhausts (up to isomorphism) all different two dimensional locally analytic non-smooth $E$-representations of $\mathbb{Q}_p^{\times}$ satisfying (\ref{3extension of two trivial reps}). We observe that $V_{\mathscr{L}}$ can be viewed as a representation of $T_2(\mathbb{Q}_p)\cong\mathbb{Q}_p^{\times}\times\mathbb{Q}_p^{\times}$ by composing with the map
\begin{equation}\label{3easy map}
T_2(\mathbb{Q}_p)\rightarrow \mathbb{Q}_p^{\times}: ~\left(\begin{array}{cc}
a&0\\
0&b\\
\end{array}\right) \mapsto a^{-1}b.
\end{equation}
As a result, we can consider the parabolic induction
$$I_{B_2}^{\mathrm{GL}_2}\left(V_{\mathscr{L}}\otimes_E\delta_{T_2,\nu}\right)$$
which naturally fits into an exact sequence
\begin{equation}\label{3non split self extension}
I_{B_2}^{\mathrm{GL}_2}(\delta_{T_2,\nu})\hookrightarrow I_{B_2}^{\mathrm{GL}_2}\left(V_{\mathscr{L}}\otimes_E\delta_{T_2,\nu}\right)\twoheadrightarrow I_{B_2}^{\mathrm{GL}_2}(\delta_{T_2,\nu}).
\end{equation}
Then we define $\Sigma_{\mathrm{GL}_2}(\nu,\mathscr{L})$ as the subrepresentation of $I_{B_2}^{\mathrm{GL}_2}\left(V_{\mathscr{L}}\otimes_E\delta_{T_2,\nu}\right)/\overline{L}_{\mathrm{GL}_2}(\nu)$ with cosocle $\overline{L}_{\mathrm{GL}_2}(\nu)$. It follows from (the proof of) Theorem~3.14 of \cite{BD18} that $\Sigma_{\mathrm{GL}_2}(\nu,\mathscr{L})$ has the form
\begin{equation}\label{3uniserial GL2}
\begin{xy}
(0,0)*+{\mathrm{St}^{\rm{an}}_2(\nu)}="a"; (20,0)*+{\overline{L}_{\mathrm{GL}_2}(\nu)}="b";
{\ar@{-}"a";"b"};
\end{xy}
\end{equation}
and the set $\{\Sigma_{\mathrm{GL}_2}(\nu,\mathscr{L})\mid \mathscr{L}\in E\}$ exhausts (up to isomorphism) all different locally analytic $E$-representations of $\mathrm{GL}_2(\mathbb{Q}_p)$ of the form (\ref{3uniserial GL2}) that do not contain
$$\begin{xy}
(0,0)*+{\overline{L}_{\mathrm{GL}_2}(\nu)\otimes_E\mathrm{St}^{\infty}_2}="a"; (26,0)*+{\overline{L}_{\mathrm{GL}_2}(\nu)}="b";
{\ar@{-}"a";"b"};
\end{xy}$$
as a subrepresentation. We have the embeddings
$$\iota_i: \mathrm{GL}_2\hookrightarrow L_i$$
for $i=1,2$ by identifying $\mathrm{GL}_2$ with a Levi block of $L_i$, which induce the embeddings
$$\iota_{T,i}: T_2\hookrightarrow T$$
by restricting $\iota_i$ to $T_2\subsetneq\mathrm{GL}_2$. We use the notation $\iota_{T,i}(V_{\mathscr{L}})$ for the locally analytic representation of $T(\mathbb{Q}_p)\cong (\mathbb{Q}_p^{\times})^3$ which is $V_{\mathscr{L}}$ after restricting to $T_2$ via $\iota_{T,i}$ and is trivial after restricting to the other copy of $\mathbb{Q}_p^{\times}$. By a direct analogue of $\Sigma_{\mathrm{GL}_2}(\nu,\mathscr{L})$, we can construct $\Sigma_{L_i}(\lambda,\mathscr{L})$ as the subrepresentation of $I_{B\cap L_i}^{L_i}\left(\iota_{T,i}(V_{\mathscr{L}})\otimes_E\delta_{T,\lambda}\right)/\overline{L}_i(\lambda)$ with cosocle $\overline{L}_i(\lambda)$. In fact, if we have $\lambda|_{T_2, \iota_{T,i}}=\nu$, then we obviously know that $\Sigma_{L_i}(\lambda,\mathscr{L})|_{\mathrm{GL}_2, \iota_i}\cong \Sigma_{\mathrm{GL}_2}(\nu,\mathscr{L})$ where the notation $(\cdot)|_{\ast, \star}$ means the restriction of $\cdot$ to $\ast$ via the embedding $\star$. We observe that the parabolic induction $I_{P_i}^{\mathrm{GL}_3}\left(\Sigma_{L_i}(\lambda,\mathscr{L})\right)$ fits into the exact sequence
$$
\begin{xy}
(0,0)*+{v_{P_{3-i}}^{\rm{an}}(\lambda)}="a"; (20,0)*+{\mathrm{St}_3^{\rm{an}}(\lambda)}="b";
{\ar@{-}"a";"b"};
\end{xy}\hookrightarrow I_{P_i}^{\mathrm{GL}_3}\left(\Sigma_{L_i}(\lambda,\mathscr{L})\right)\twoheadrightarrow
\begin{xy}
(0,0)*+{\overline{L}(\lambda)}="a"; (16,0)*+{v_{P_i}^{\rm{an}}(\lambda)}="b";
{\ar@{-}"a";"b"};
\end{xy}.$$
According to Proposition~5.6 of \cite{Schr11} for example, we know that
$$\mathrm{Ext}^1_{\mathrm{GL}_3(\mathbb{Q}_p),\lambda}\left(\overline{L}(\lambda), ~\mathrm{St}_3^{\rm{an}}(\lambda)\right)=0$$
and thus we can define $\Sigma_i(\lambda, \mathscr{L})$ as the unique quotient of $I_{P_i}^{\mathrm{GL}_3}\left(\Sigma_{L_i}(\lambda,\mathscr{L})\right)$ that fits into the exact sequence
$$\mathrm{St}_3^{\rm{an}}(\lambda)\hookrightarrow \Sigma_i(\lambda, \mathscr{L})\twoheadrightarrow v_{P_i}^{\rm{an}}(\lambda).$$
The constructions of $\Sigma_i(\lambda, \mathscr{L})$ above actually induce canonical isomorphisms
\begin{equation}\label{3canonical isomorphism}
\mathrm{Hom}_{\rm{cont}}\left(\mathbb{Q}_p^{\times},~E\right)\cong\mathrm{Ext}^1_{\mathbb{Q}_p^{\times}}\left(1,~1\right)\xrightarrow{\sim}\mathrm{Ext}^1_{\mathrm{GL}_3(\mathbb{Q}_p),\lambda}\left(v_{P_i}^{\rm{an}}(\lambda), ~\mathrm{St}_3^{\rm{an}}(\lambda)\right)
\end{equation}
for $i=1,2$. We denote the image of $\mathrm{log}_0$ (resp, of $\mathrm{val}_p$) in $$\mathrm{Ext}^1_{\mathrm{GL}_3(\mathbb{Q}_p),\lambda}\left(v_{P_i}^{\rm{an}}(\lambda), ~\mathrm{St}_3^{\rm{an}}(\lambda)\right)$$
by $b_{i,\mathrm{log}_0}$ (resp. by $b_{i,\mathrm{val}_p}$). We use the notation $1_T$ for the trivial character of $T(\mathbb{Q}_p)$. We use the same notation $b_{i,\mathrm{log}_0}$ and $b_{i,\mathrm{val}_p}$ for the image of $\mathrm{log}_0$ and $\mathrm{val}_p$ respectively under the embedding $$\mathrm{Ext}^1_{\mathbb{Q}_p^{\times}}\left(1,~1\right)\hookrightarrow\mathrm{Ext}^1_{T(\mathbb{Q}_p),0}\left(1_T,~1_T\right)$$
induced by the maps
$$T(\mathbb{Q}_p)\xrightarrow{p_i} T_2(\mathbb{Q}_p)\xrightarrow{(\ref{3easy map})}\mathbb{Q}_p^{\times}$$
where $p_i$ is the section of $\iota_{T,i}$ which is compatible with the projection $L_i\twoheadrightarrow\mathrm{GL}_2$. Recall the elements $c_{i,\mathrm{log}}, c_{i,\mathrm{val}}\in \mathrm{Ext}^1_{T(\mathbb{Q}_p),0}(1_T,1_T)$ constructed after (5.24) of \cite{Schr11} and observe that
\begin{equation}\label{3relation of basis}
\left\{\begin{array}{ccc}
c_{1,\mathrm{log}}=b_{1,\mathrm{log}_0}+2b_{2,\mathrm{log}_0},&c_{1,\mathrm{val}}=b_{1,\mathrm{val}_p}+2b_{2,\mathrm{val}_p}&\\
c_{2,\mathrm{log}}=2b_{1,\mathrm{log}_0}+b_{2,\mathrm{log}_0},&c_{2,\mathrm{val}}=2b_{1,\mathrm{val}_p}+b_{2,\mathrm{val}_p}&.\\
\end{array}\right.
\end{equation}
We notice that there exists canonical surjections
\begin{equation}\label{3canonical surjection for L invariant}
\mathrm{Ext}^1_{T(\mathbb{Q}_p),0}\left(1_T,~1_T\right)\twoheadrightarrow \mathrm{Ext}^1_{\mathrm{GL}_3(\mathbb{Q}_p),\lambda}\left(v_{P_i}^{\rm{an}}(\lambda), ~\mathrm{St}_3^{\rm{an}}(\lambda)\right)
\end{equation}
with kernel spanned by $\{c_{i, \mathrm{log}},~c_{i, \mathrm{val}}\}$, according to (5.70) and (5.71) of \cite{Schr11}. Therefore the relation (\ref{3relation of basis}) reduces via the surjection (\ref{3canonical surjection for L invariant}) to
\begin{equation}\label{3relation of basis prime}
c_{3-i,\mathrm{log}}=-3b_{i,\mathrm{log}_0},~c_{3-i,\mathrm{val}}=-3b_{i,\mathrm{val}_p}
\end{equation}
inside the quotient $\mathrm{Ext}^1_{\mathrm{GL}_3(\mathbb{Q}_p),\lambda}\left(v_{P_i}^{\rm{an}}(\lambda), ~\mathrm{St}_3^{\rm{an}}(\lambda)\right)$. We define $\Sigma(\lambda, \mathscr{L}_1, \mathscr{L}_2)$ as the amalgamate sum of $\Sigma_1(\lambda, \mathscr{L}_1)$ and $\Sigma_2(\lambda, \mathscr{L}_2)$ over $\mathrm{St}_3^{\rm{an}}(\lambda)$, for each $\mathscr{L}_1, \mathscr{L}_2\in E$. Consequently, $\Sigma(\lambda, \mathscr{L}_1, \mathscr{L}_2)$ has the following form
$$\begin{xy}
(0,0)*+{\mathrm{St}_3^{\rm{an}}(\lambda)}="a"; (20,6)*+{v_{P_1}^{\rm{an}}(\lambda)}="b"; (20,-6)*+{v_{P_2}^{\rm{an}}(\lambda)}="c";
{\ar@{-}"a";"b"}; {\ar@{-}"a";"c"};
\end{xy}$$
and we have
\begin{equation}\label{3normalization of notation}
\Sigma(\lambda, \mathscr{L}_1, \mathscr{L}_2)\cong \Sigma(\lambda, \mathscr{L}, \mathscr{L}^{\prime})
\end{equation}
if
\begin{equation}\label{3sign of L invariant}
\mathscr{L}_1=-\mathscr{L}^{\prime}, \mathscr{L}_2=-\mathscr{L}\in E,
\end{equation}
where $\Sigma(\lambda, \mathscr{L}, \mathscr{L}^{\prime})$ is the locally analytic representation defined in Definition~5.12 of \cite{Schr11} using the element
$$(c_{2,\mathrm{log}}+\mathscr{L}^{\prime}c_{2,\mathrm{val}},~c_{1,\mathrm{log}}+\mathscr{L}c_{1,\mathrm{val}})$$
in
$$\mathrm{Ext}^1_{\mathrm{GL}_3(\mathbb{Q}_p),\lambda}\left(v_{P_1}^{\rm{an}}(\lambda)\oplus v_{P_2}^{\rm{an}}(\lambda), ~\mathrm{St}_3^{\rm{an}}(\lambda)\right).$$
\begin{rema}\label{3rema: sign of L invariant}
The appearance of a sign in (\ref{3sign of L invariant}) is essentially due to Remark~3.1 of \cite{Ding18}, which implies that our invariants $\mathscr{L}_1$ and $\mathscr{L}_2$ can be identified with Fontaine--Mazur $\mathscr{L}$\text{-}invariants of the corresponding Galois representation via local-global compatibility.
\end{rema}

We have a canonical morphism by (5.26) of \cite{Schr11}
\begin{equation}\label{3canonical morphism of Ext2}
\kappa: \mathrm{Ext}^2_{T(\mathbb{Q}_p),0}(1_T,1_T)\rightarrow\mathrm{Ext}^2_{\mathrm{GL}_3(\mathbb{Q}_p),\lambda}\left(\overline{L}(\lambda), ~\mathrm{St}_3^{\rm{an}}(\lambda)\right).
\end{equation}
Note that we also have
$$\mathrm{Ext}^2_{T(\mathbb{Q}_p),0}(1_T,1_T)\cong \wedge^2\left(\mathrm{Ext}^1_{T(\mathbb{Q}_p),0}(1_T,1_T)\right)$$
by (5.24) of \cite{Schr11} and thus the set
$$\{b_{1,\mathrm{val}_p}\wedge b_{2,\mathrm{val}_p}, b_{1,\mathrm{log}_0}\wedge b_{2,\mathrm{val}_p}, b_{1,\mathrm{val}_p}\wedge b_{2,\mathrm{log}_0}, b_{1,\mathrm{log}_0}\wedge b_{2,\mathrm{log}_0}, b_{1,\mathrm{val}_p}\wedge b_{1,\mathrm{log}_0}, b_{2,\mathrm{val}_p}\wedge b_{2,\mathrm{log}_0}\}$$
forms a basis of $\mathrm{Ext}^2_{T(\mathbb{Q}_p),0}\left(1_T,~1_T\right)$. It follows from (5.27) of \cite{Schr11} and (\ref{3relation of basis}) that the set
$$\{\kappa(b_{1,\mathrm{val}_p}\wedge b_{2,\mathrm{val}_p}), \kappa(b_{1,\mathrm{log}_0}\wedge b_{2,\mathrm{val}_p}), \kappa(b_{1,\mathrm{val}_p}\wedge b_{2,\mathrm{log}_0}), \kappa(b_{1,\mathrm{log}_0}\wedge b_{2,\mathrm{log}_0})\}$$
forms a basis of the image of (\ref{3canonical morphism of Ext2}).

We recall the $p$\text{-}adic dilogarithm function $li_2: \mathbb{Q}_p\setminus\{0,1\}\rightarrow \mathbb{Q}_p$ defined by Coleman in \cite{Cole82} and we consider the function
$$D_{\mathscr{L}}(z):=li_2(z)+\frac{1}{2}\mathrm{log}_{\mathscr{L}}(z)\mathrm{log}_{\mathscr{L}}(1-z)$$
as in (5.34) of \cite{Schr11}. We also define
$$d(z):=\mathrm{log}_{\mathscr{L}}(1-z)\mathrm{val}_p(z)-\mathrm{log}_{\mathscr{L}}(z)\mathrm{val}_p(1-z)$$
as in (5.36) of \cite{Schr11} which is also a locally analytic function over $\mathbb{Q}_p\setminus\{0,1\}$ and is independent of the choice of $\mathscr{L}\in E$. Note by our definition that
$$D_{\mathscr{L}}-D_0=\frac{\mathscr{L}}{2}d.$$
It follows from Theorem 7.2 of \cite{Schr11} that $\{D_0,d\}$ can be identified with a basis of $$\mathrm{Ext}^2_{\mathrm{GL}_2(\mathbb{Q}_p),0}\left(1,~\mathrm{St}_2^{\rm{an}}\right)$$
( cf. (5.38) of \cite{Schr11}) which naturally embeds into $\mathrm{Ext}^2_{\mathrm{GL}_2(\mathbb{Q}_p)}\left(1,~\mathrm{St}_2^{\rm{an}}\right)$. Then the map $\iota_i: \mathrm{GL}_2\hookrightarrow L_i$ induces the isomorphisms
\begin{equation}\label{3sequence of isomorphisms}
\mathrm{Ext}^2_{\mathrm{GL}_2(\mathbb{Q}_p)}\left(1_2,~\mathrm{St}_2^{\rm{an}}\right)\xleftarrow{\sim}\mathrm{Ext}^2_{L_i(\mathbb{Q}_p),0}\left(1_{L_i},~\mathrm{St}_2^{\rm{an}}\right)\xleftarrow{\sim}\mathrm{Ext}^2_{\mathrm{GL}_3(\mathbb{Q}_p),0}\left(1_3,~I_{P_i}^{\mathrm{GL}_3}(\mathrm{St}_2^{\rm{an}})\right)
\end{equation}
where $L_i(\mathbb{Q}_p)$ acts on $\mathrm{St}_2^{\rm{an}}$ via the projection $p_i$. We abuse the notation for the composition
\begin{equation}
\iota_i: \mathrm{Ext}^2_{\mathrm{GL}_2(\mathbb{Q}_p)}\left(1_2,~\mathrm{St}_2^{\rm{an}}\right)\xleftarrow{\sim}\mathrm{Ext}^2_{\mathrm{GL}_3(\mathbb{Q}_p),0}\left(1_3,~I_{P_i}^{\mathrm{GL}_3}(\mathrm{St}_2^{\rm{an}})\right)\rightarrow\mathrm{Ext}^2_{\mathrm{GL}_3(\mathbb{Q}_p),0}\left(1_3, ~\mathrm{St}_3^{\rm{an}}\right)
\end{equation}
given by (\ref{3sequence of isomorphisms}) and the surjection
$$I_{P_i}^{\mathrm{GL}_3}(\mathrm{St}_2^{\rm{an}})\twoheadrightarrow\mathrm{St}_3^{\rm{an}}.$$
Finally there is canonical isomorphism
$$\mathrm{Ext}^2_{\mathrm{GL}_3(\mathbb{Q}_p),0}\left(1_3, ~\mathrm{St}_3^{\rm{an}}\right)\cong\mathrm{Ext}^2_{\mathrm{GL}_3(\mathbb{Q}_p),\lambda}\left(\overline{L}(\lambda), ~\mathrm{St}_3^{\rm{an}}(\lambda)\right)$$
by (5.20) of \cite{Schr11}.
\begin{lemm}\label{3lemm: ext2.1}
We have
$$\mathrm{dim}_E\mathrm{Ext}^2_{\mathrm{GL}_3(\mathbb{Q}_p),\lambda}\left(\overline{L}(\lambda), ~\mathrm{St}_3^{\rm{an}}(\lambda)\right)=5$$
and the set
$$\{\kappa(b_{1,\mathrm{val}_p}\wedge b_{2,\mathrm{val}_p}), \kappa(b_{1,\mathrm{log}_0}\wedge b_{2,\mathrm{val}_p}), \kappa(b_{1,\mathrm{val}_p}\wedge b_{2,\mathrm{log}_0}), \kappa(b_{1,\mathrm{log}_0}\wedge b_{2,\mathrm{log}_0}), \iota_i(D_0)\}$$
forms a basis of $\mathrm{Ext}^2_{\mathrm{GL}_3(\mathbb{Q}_p),\lambda}\left(\overline{L}(\lambda), ~\mathrm{St}_3^{\rm{an}}(\lambda)\right)$ for $i=1,2$.
\end{lemm}
\begin{proof}
This follows directly from (5.57) of \cite{Schr11} and (\ref{3relation of basis}).
\end{proof}

\begin{lemm}\label{3lemm: ext2 prime}
There exists $\gamma\in E^{\times}$ such that
$$\iota_1(d)=\iota_2(d)=\gamma\left(\kappa(b_{1,\mathrm{log}_0}\wedge b_{2,\mathrm{val}_p}+b_{1,\mathrm{val}_p}\wedge b_{2,\mathrm{log}_0}\right).$$
\end{lemm}
\begin{proof}
This follows directly from Lemma~5.8 of \cite{Schr11} and (\ref{3relation of basis}) if we take
$$\gamma:=-3\alpha$$
where $\alpha\in E^{\times}$ is the constant in the statement of Lemma~5.8 of \cite{Schr11}.
\end{proof}
\begin{lemm}\label{3lemm: ext2}
We have
$$\mathrm{dim}_E\mathrm{Ext}^1_{\mathrm{GL}_3(\mathbb{Q}_p),\lambda}\left(\overline{L}(\lambda), ~\Sigma(\lambda, \mathscr{L}_1, \mathscr{L}_2)\right)=1\mbox{ and }\mathrm{dim}_E\mathrm{Ext}^2_{\mathrm{GL}_3(\mathbb{Q}_p),\lambda}\left(\overline{L}(\lambda), ~\Sigma(\lambda, \mathscr{L}_1, \mathscr{L}_2)\right)=2.$$
Moreover, the image of
$$\{\kappa(b_{1,\mathrm{val}_p}\wedge b_{2,\mathrm{val}_p}), \iota_i(D_0)\}$$
under
$$\mathrm{Ext}^2_{\mathrm{GL}_3(\mathbb{Q}_p),\lambda}\left(\overline{L}(\lambda), ~\mathrm{St}_3^{\rm{an}}(\lambda)\right)\rightarrow\mathrm{Ext}^2_{\mathrm{GL}_3(\mathbb{Q}_p),\lambda}\left(\overline{L}(\lambda), ~\Sigma(\lambda, \mathscr{L}_1, \mathscr{L}_2)\right)$$
forms a basis of $\mathrm{Ext}^2_{\mathrm{GL}_3(\mathbb{Q}_p),\lambda}\left(\overline{L}(\lambda), ~\Sigma(\lambda, \mathscr{L}_1, \mathscr{L}_2)\right)$ for $i=1$ or $2$.
\end{lemm}
\begin{proof}
This follows directly from Corollary~5.17 of \cite{Schr11} and (\ref{3relation of basis}).
\end{proof}
We recall from (5.55) of \cite{Schr11} that
\begin{equation}\label{3definition of basis}
c_0:=\alpha^{-1}\iota_1(D_0)-\frac{1}{2}\kappa(c_{1,\mathrm{log}}\wedge c_{2,\mathrm{log}})
\end{equation}
where $\alpha$ is defined in Lemma~5.8 of \cite{Schr11}.
\begin{lemm}\label{3lemm: easy normalization of higher invariant}
Assume that $\mathscr{L}_3\in E$ satisfies the equality
\begin{multline}\label{3equality of higher invariant}
E\left(\iota_1(D_0)+\mathscr{L}_3\kappa(b_{1,\mathrm{val}_p}\wedge b_{2,\mathrm{val}_p})\right)=E\left(c_0+\mathscr{L}^{\prime\prime}\kappa(c_{1,\mathrm{val}}\wedge c_{2,\mathrm{val}})\right)\\
\subsetneq \mathrm{Ext}^2_{\mathrm{GL}_3(\mathbb{Q}_p),\lambda}\left(\overline{L}(\lambda), ~\Sigma(\lambda, \mathscr{L}_1, \mathscr{L}_2)\right).
\end{multline}
Then we have
$$\mathscr{L}_3=\gamma(\mathscr{L}^{\prime\prime}-\frac{1}{2}\mathscr{L}_1\mathscr{L}_2)=\gamma(\mathscr{L}^{\prime\prime}-\frac{1}{2}\mathscr{L}\mathscr{L}^{\prime}).$$
\end{lemm}
\begin{proof}
All the equalities in this lemma are understood to be inside
$$\mathrm{Ext}^2_{\mathrm{GL}_3(\mathbb{Q}_p),\lambda}\left(\overline{L}(\lambda), ~\Sigma(\lambda, \mathscr{L}_1, \mathscr{L}_2)\right)$$
without causing ambiguity. It follows from our assumption (\ref{3equality of higher invariant}) that
$$\iota_1(D_0)+\mathscr{L}_3\kappa(b_{1,\mathrm{val}_p}\wedge b_{2,\mathrm{val}_p})=\alpha\left(c_0+\mathscr{L}^{\prime\prime}\kappa(c_{1,\mathrm{val}}\wedge c_{2,\mathrm{val}})\right)$$
which together with (\ref{3definition of basis}) imply that
\begin{equation}\label{3equality in ext2}
\mathscr{L}_3\kappa(b_{1,\mathrm{val}_p}\wedge b_{2,\mathrm{val}_p})=\frac{\alpha}{2}\kappa(c_{1,\mathrm{log}}\wedge c_{2,\mathrm{log}})+\alpha\mathscr{L}^{\prime\prime}\kappa(c_{1,\mathrm{val}}\wedge c_{2,\mathrm{val}}).
\end{equation}
We know that
\begin{equation}\label{3relation from simple invariant}
\kappa(c_{1,\mathrm{log}}\wedge c_{2,\mathrm{log}})=\mathscr{L}\mathscr{L}^{\prime}\kappa(c_{1,\mathrm{val}}\wedge c_{2,\mathrm{val}})
\end{equation}
from the proof of Corollary~5.17 of \cite{Schr11} and that
\begin{equation}\label{3relation of val}
\kappa(c_{1,\mathrm{val}}\wedge c_{2,\mathrm{val}})=-3\kappa(b_{1,\mathrm{val}_p}\wedge b_{2,\mathrm{val}_p})
\end{equation}
from (\ref{3relation of basis}). Therefore we finish the proof by combining (\ref{3equality in ext2}), (\ref{3relation from simple invariant}) and (\ref{3relation of val}) with (\ref{3sign of L invariant}) and the equality $\gamma=-3\alpha$ from Lemma~\ref{3lemm: ext2 prime}.
\end{proof}
\begin{rema}\label{3equality of dilog}
We emphasize that we do not know whether
$$E \iota_1(D_0)=E \iota_2(D_0)$$
in $\mathrm{Ext}^2_{\mathrm{GL}_3(\mathbb{Q}_p),\lambda}\left(\overline{L}(\lambda), ~\mathrm{St}_3^{\rm{an}}(\lambda)\right)$ or not, which is of independent interest.
\end{rema}
\section{A key result for $\mathrm{GL}_2(\mathbb{Q}_p)$}\label{3section: GL2Qp}
In this section, we are going to prove Proposition~\ref{3prop: key result} which will be a crucial ingredient for the proof of Lemma~\ref{3lemm: ext6} and Proposition~\ref{3prop: main dim}.

We use the following shorten notation
$$I(\nu):=I_{B_2}^{\mathrm{GL}_2}(\delta_{T_2,\nu}),~\widetilde{I}(\nu):=I_{B_2}^{\mathrm{GL}_2}(\delta_{T_2,\nu}\otimes_E(|\cdot|^{-1}\otimes_E|\cdot|))$$
for each weight $\nu\in X(T_2)$.
\begin{lemm}\label{3lemm: ext1}
We have
$$\mathrm{dim}_E\mathrm{Ext}^1_{\mathrm{GL}_2(\mathbb{Q}_p)}\left(\widetilde{I}(s\cdot\nu),~ \Sigma_{\mathrm{GL}_2}(\nu,\mathscr{L})\right)=1.$$
\end{lemm}
\begin{proof}
This is essentially contained in the proof of Theorem~3.14 of \cite{BD18}. In fact, we know that
$$
\begin{array}{cccc}
\mathrm{Ext}^1_{\mathrm{GL}_2(\mathbb{Q}_p)}\left(\widetilde{I}(s\cdot\nu),~\begin{xy}
(0,0)*+{\overline{L}_{\mathrm{GL}_2}(\nu)\otimes_E\mathrm{St}_2^{\infty}}="a"; (26,0)*+{I(s\cdot\nu)}="b";
{\ar@{-}"a";"b"};
\end{xy}\right)&=&0&\\
\mathrm{Ext}^2_{\mathrm{GL}_2(\mathbb{Q}_p)}\left(\widetilde{I}(s\cdot\nu),~\begin{xy}
(0,0)*+{\overline{L}_{\mathrm{GL}_2}(\nu)\otimes_E\mathrm{St}_2^{\infty}}="a"; (26,0)*+{I(s\cdot\nu)}="b";
{\ar@{-}"a";"b"};
\end{xy}\right)&=&0&
\end{array}
$$
and
$$\mathrm{dim}_E\mathrm{Ext}^1_{\mathrm{GL}_2(\mathbb{Q}_p)}(\widetilde{I}(s\cdot\nu),~\overline{L}_{\mathrm{GL}_2}(\nu))=1$$
which finish the proof by a simple devissage induced by the short exact sequence
$$\left(\begin{xy}
(0,0)*+{\overline{L}_{\mathrm{GL}_2}(\nu)\otimes_E\mathrm{St}_2^{\infty}}="a"; (26,0)*+{I(s\cdot\nu)}="b";
{\ar@{-}"a";"b"};
\end{xy}\right)\hookrightarrow \Sigma_{\mathrm{GL}_2}(\nu,\mathscr{L})\twoheadrightarrow \overline{L}_{\mathrm{GL}_2}(\nu).$$
\end{proof}

We fix a split $p$-adic reductive group $H$ and have a natural embedding
$$U(\mathfrak{h})\hookrightarrow D(H, E)_{\{1\}}\hookrightarrow D(H, E)$$
where $D(H, E)_{\{1\}}$ is the closed subalgebra of $D(H, E)$ consisting of distributions supported at the identity element ( cf. \cite{Koh07}). The embedding above induces another embedding
\begin{equation}\label{3center of distribution}
Z(U(\mathfrak{h}))\hookrightarrow Z(D(H, E))
\end{equation}
by the main result of \cite{Koh07} where $Z(\cdot)$ is the notation for the center of a non-commutative algebra. We say that $\Pi\in\mathrm{Rep}^{\rm{la}}_{\mathrm{GL}_2(\mathbb{Q}_p), E}$ has an infinitesimal character if $Z(U(\mathfrak{h}))$ acts on $\Pi^{\prime }$ via a character.
\begin{lemm}\label{3lemm: existence of inf char}
If $V, W\in\mathrm{Rep}^{\rm{la}}_{H, E}$ have both the same central character and the same infinitesimal character and satisfy
$$\mathrm{Hom}_{H}\left(V, ~W\right)=0,$$
then any non-split extension of the form $\begin{xy}
(0,0)*+{W}="a"; (10,0)*+{V}="b";
{\ar@{-}"a";"b"};
\end{xy}$ has both the same central character and the same infinitesimal character as the one for $V$ and $W$.
\end{lemm}
\begin{proof}
This is a direct analogue of Lemma~3.1 in \cite{BD18} and follows essentially from the fact that both $D(Z(H), E)$ and $Z(U(\mathfrak{h}))$ are subalgebras of $Z(D(H, E))$ by \cite{Koh07}.
\end{proof}
We fix a Borel subgroup $B_H\subseteq H$ as well as its opposite Borel subgroup $\overline{B_H}$. We consider the split maximal torus $T_H:=B_H\cap \overline{B_H}$ and use the notation $N_H$ (resp. $\overline{N_H}$) for the unipotent radical of $B_H$ (resp. of $\overline{B_H}$).
\begin{lemm}\label{3lemm: harish chandra}
If $V\in\mathrm{Rep}^{\rm{la}}_{H, E}$ has an infinitesimal character, then $U(\mathfrak{t}_{\mathfrak{h}})^{W_H}$ (as a subalgebra of $U(\mathfrak{t}_{\mathfrak{h}})$) acts on $J_{\overline{B_H}}(V)$ via a character where $W_H$ is the Weyl group of $H$.
\end{lemm}
\begin{proof}
We know by our assumption that $Z(U(\mathfrak{h}))$ acts on $V^{\prime}$ (and hence on $V$ as well) via a character. We note from (\ref{3center of distribution}) that $Z(U(\mathfrak{h}))$ commutes with $D(\overline{N_H}, E)\subseteq D(H, E)$ and thus the action of $Z(U(\mathfrak{h}))$ on $V$ commutes with that of $\overline{N_H}$, which implies that $Z(U(\mathfrak{h}))$ acts on $V^{\overline{N_H}^{\circ}}$ via a character for each open compact subgroup $\overline{N_H}^{\circ}\subseteq \overline{N_H}$. We use the notation
$$\theta: Z(U(\mathfrak{h}))\xrightarrow{\sim} U(\mathfrak{t}_{\mathfrak{h}})^{W_H}$$
for the Harish-Chandra isomorphism ( cf. Section~1.7 of \cite{Hum08}) and the notation $j_1$ and $j_2$ for the embeddings
$$j_1: Z(U(\mathfrak{h}))\hookrightarrow U(\mathfrak{h})\mbox{ and }j_2: U(\mathfrak{t}_{\mathfrak{h}})\hookrightarrow U(\mathfrak{h}).$$
We choose an arbitrary Verma module $M_H(\lambda_H)$ with highest weight $\lambda_H$, namely we have
$$M_H(\lambda):=U(\mathfrak{h})\otimes_{U(\overline{\mathfrak{b}_H})}\lambda_H.$$
We use the notation $M_H(\lambda_H)_{\mu}$ for the subspace of $M_H(\lambda)$ with $\mathfrak{t}_{\mathfrak{h}}$-weight $\mu$ and note that
$$\mathrm{dim}_EM_H(\lambda_H)_{\lambda_H}=1.$$
We easily observe that
\begin{equation}\label{3preserve highest weight}
Z(U(\mathfrak{h}))\cdot M_H(\lambda_H)_{\lambda_H}=M_H(\lambda_H)_{\lambda_H}\mbox{ and }U(\mathfrak{t}_{\mathfrak{h}})\cdot M_H(\lambda_H)_{\lambda_H}=M_H(\lambda_H)_{\lambda_H}.
\end{equation}
It is well-known that the the direct sum decomposition
\begin{equation}\label{3decomposition lie}
\mathfrak{h}=\mathfrak{n}_H\oplus\mathfrak{t}_{\mathfrak{h}}\oplus\overline{\mathfrak{n}_H}
\end{equation}
induces a tensor decomposition of $E$-vector space
\begin{equation}\label{3PBW decomposition}
U(\mathfrak{h})=U(\mathfrak{n}_H)\otimes_EU(\mathfrak{t}_{\mathfrak{h}})\otimes_EU(\overline{\mathfrak{n}_H}).
\end{equation}
Hence we can write each element in $U(\mathfrak{h})$ as a polynomial with variables indexed by a standard basis of $\mathfrak{h}$ that is compatible with (\ref{3decomposition lie}). It follows from the definition of $\theta$ as the restriction to $Z(U(\mathfrak{h}))$ of the projection $U(\mathfrak{h})\twoheadrightarrow U(\mathfrak{t}_{\mathfrak{h}})$ (coming from (\ref{3PBW decomposition})) that
$$j_1(z)-j_2\circ\theta(z)\in U(\mathfrak{h})\cdot\overline{\mathfrak{n}_H}+\mathfrak{n}_H\cdot U(\mathfrak{h})$$
for each $z\in Z(U(\mathfrak{h}))$. If a monomial $f$ in the decomposition (\ref{3PBW decomposition}) of $j_1(z)-j_2\circ\theta(z)$ belongs to
$$\mathfrak{n}_H\cdot U(\mathfrak{n}_H)\cdot U(\mathfrak{t}_{\mathfrak{h}}),$$
then we have
$$f\cdot M_H(\lambda_H)_{\lambda_H}\subseteq M_H(\lambda_H)_{\mu}$$
for some $\mu\neq \lambda_H$, which contradicts the fact (\ref{3preserve highest weight}). Hence we conclude that
$$j_1(z)-j_2\circ\theta(z)\in U(\mathfrak{h})\cdot\overline{\mathfrak{n}_H}$$
and in particular
$$j_1(z)=j_2\circ\theta(z)$$
on $V^{\overline{N_H}^{\circ}}$ for each $z\in Z(U(\mathfrak{h}))$. Hence we deduce that $U(\mathfrak{t}_{\mathfrak{h}})^{W_H}$ acts on $V^{\overline{N_H}^{\circ}}$ via a character. We note by the definition of $J_{\overline{B_H}}$ ( cf. \cite{Eme06}) that we have a $T_H^+$-equivariant embedding
\begin{equation}\label{3definition of Jacquet}
J_{\overline{B_H}}(V)\hookrightarrow V^{\overline{N_H}^{\circ}}
\end{equation}
where $T_H^+$ is a certain submonoid of $T_H$ containing an open compact subgroup. As a result, (\ref{3definition of Jacquet}) is also $U(\mathfrak{t}_{\mathfrak{h}})$-equivariant and thus $U(\mathfrak{t}_{\mathfrak{h}})^{W_H}$ acts on $J_{\overline{B_H}}(V)$ via a character which finishes the proof.
\end{proof}
We set $H=\mathrm{GL}_2(\mathbb{Q}_p)$, $B_H=B_2$ and $\overline{B_H}=\overline{B_2}$ in the rest of this section. The idea of the following lemma which is closely related to Lemma~3.20 of \cite{BD18}, owes very much to Y.Ding.
\begin{lemm}\label{3lemm: no inf char}
A locally analytic representation of either the form
\begin{equation}\label{3first form}
\begin{xy}
(0,0)*+{\overline{L}_{\mathrm{GL}_2}(\nu)\otimes_E\mathrm{St}_2^{\infty}}="a"; (26,0)*+{I(s\cdot\nu)}="b"; (45,0)*+{\overline{L}_{\mathrm{GL}_2}(\nu)}="c"; (72,0)*+{\overline{L}_{\mathrm{GL}_2}(\nu)\otimes_E\mathrm{St}_2^{\infty}}="d";
{\ar@{-}"a";"b"}; {\ar@{-}"b";"c"}; {\ar@{-}"c";"d"};
\end{xy}
\end{equation}
or the form
\begin{equation}\label{3second form}
\begin{xy}
(0,0)*+{\overline{L}_{\mathrm{GL}_2}(\nu)}="a"; (20,0)*+{\widetilde{I}(s\cdot\nu)}="b"; (45,0)*+{\overline{L}_{\mathrm{GL}_2}(\nu)\otimes_E\mathrm{St}_2^{\infty}}="c"; (72,0)*+{\overline{L}_{\mathrm{GL}_2}(\nu)}="d";
{\ar@{-}"a";"b"}; {\ar@{-}"b";"c"}; {\ar@{-}"c";"d"};
\end{xy}
\end{equation}
does not have an infinitesimal character.
\end{lemm}
\begin{proof}
Assume that a representation $V$ of the form (\ref{3first form}) has an infinitesimal character. Note that $V$ can be represented by an element in the space $\mathrm{Ext}^1_{\mathrm{GL}_2(\mathbb{Q}_p)}(\overline{L}_{\mathrm{GL}_2}(\nu)\otimes_E\mathrm{St}_2^{\infty}, \Sigma_{\mathrm{GL}_2}(\nu, \mathscr{L}))$ for certain $\mathscr{L}\in E$. We consider the upper-triangular Borel subgroup $\overline{B_2}$ and the diagonal split torus $T_2$. Then by the proof of Lemma~3.20 of \cite{BD18} we know that the Jacquet functor $J_{\overline{B_2}}$ ( cf. \cite{Eme06} for the definition) induces a injection
\begin{multline}\label{3Jacquet injection}
\mathrm{Ext}^1_{\mathrm{GL}_2(\mathbb{Q}_p)}\left(\overline{L}_{\mathrm{GL}_2}(\nu)\otimes_E\mathrm{St}_2^{\infty}, ~\Sigma_{\mathrm{GL}_2}(\nu, \mathscr{L})\right)\\
\hookrightarrow\mathrm{Ext}^1_{T_2(\mathbb{Q}_p)}\left(\delta_{T_2,\nu}\otimes_E(|\cdot|\otimes_E|\cdot|^{-1}), ~\delta_{T_2,\nu}\otimes_E(|\cdot|\otimes_E|\cdot|^{-1})\right).
\end{multline}
By twisting $\delta_{T_2, -\nu}\otimes_E(|\cdot|^{-1}\otimes_E|\cdot|)$ we have an isomorphism
\begin{equation}\label{3twist isomorphism}
\mathrm{Ext}^1_{T_2(\mathbb{Q}_p)}\left(\delta_{T_2,\nu}\otimes_E(|\cdot|\otimes_E|\cdot|^{-1}), ~\delta_{T_2,\nu}\otimes_E(|\cdot|\otimes_E|\cdot|^{-1})\right)\cong \mathrm{Ext}^1_{T_2(\mathbb{Q}_p)}\left(1_{T_2}, ~1_{T_2}\right).
\end{equation}
It follows from Lemma~3.20 of \cite{BD18} (up to changes on notation) that the image of the composition of (\ref{3twist isomorphism}) and (\ref{3Jacquet injection}) is a certain two dimensional subspace $\mathrm{Ext}^1_{T_2(\mathbb{Q}_p)}(1, 1)_{\mathscr{L}}$ of $\mathrm{Ext}^1_{T_2(\mathbb{Q}_p)}(1, 1)$ depending on $\mathscr{L}$. More precisely, if we use the notation $\epsilon_1$, $\epsilon_2$ for the two charaters
$$\epsilon_1: T_2(\mathbb{Q}_p)\rightarrow\mathbb{Q}_p^{\times},~\left(\begin{array}{cc}\
a&0\\
0&b\\
\end{array}\right)\mapsto a\mbox{  and  } \epsilon_2: T_2(\mathbb{Q}_p)\rightarrow\mathbb{Q}_p^{\times},~\left(\begin{array}{cc}\
a&0\\
0&b\\
\end{array}\right)\mapsto b,$$
then the set
$$\{\mathrm{log}_0\circ \epsilon_1, \mathrm{val}_p\circ \epsilon_1, \mathrm{log}_0\circ \epsilon_2, \mathrm{val}_p\circ \epsilon_2\}$$
forms a basis of $\mathrm{Ext}^1_{T_2(\mathbb{Q}_p)}(1, 1)$, and the subspace $\mathrm{Ext}^1_{T_2(\mathbb{Q}_p)}(1, 1)_{\mathscr{L}}$ has a basis
$$\{\mathrm{log}_0\circ \epsilon_1+\mathrm{log}_0\circ \epsilon_2, \mathrm{val}_p\circ \epsilon_1+\mathrm{val}_p\circ \epsilon_2, \mathrm{log}_0\circ \epsilon_1-\mathrm{log}_0\circ \epsilon_2+\mathscr{L}(\mathrm{val}_p\circ \epsilon_1-\mathrm{val}_p\circ \epsilon_2)\}.$$
It follows from Lemma~\ref{3lemm: harish chandra} that $U(\mathfrak{t}_2)^{W_{\mathrm{GL}_2}}$ acts on $J_{\overline{B_2}}(V)$ via a character where $W_{\mathrm{GL}_2}$ is the notation for the Weyl group of $\mathrm{GL}_2(\mathbb{Q}_p)$. Therefore we deduce by a twisting that the the subspace of $\mathrm{Ext}^1_{T_2(\mathbb{Q}_p)}(1, 1)$ corresponding to $J_{\overline{B_2}}(V)$ is killed by $U(\mathfrak{t}_2)^{W_{\mathrm{GL}_2}}$. We notice that the subspace $M$ of $\mathrm{Ext}^1_{T_2(\mathbb{Q}_p)}(1, 1)$ killed by $U(\mathfrak{t}_2)^{W_{\mathrm{GL}_2}}$ is two dimensional with basis $$\{\mathrm{val}_p\circ \epsilon_1, \mathrm{val}_p\circ \epsilon_2\}$$
and we have
$$M\cap\mathrm{Ext}^1_{T_2(\mathbb{Q}_p)}(1, 1)_{\mathscr{L}}=E\left(\mathrm{val}_p\circ \epsilon_1+\mathrm{val}_p\circ \epsilon_2\right).$$
However, the representation given by the line $E(\mathrm{val}_p\circ \epsilon_1+\mathrm{val}_p\circ \epsilon_2)$ has a subrepresentation of the form
$$
\begin{xy}
(0,0)*+{\overline{L}_{\mathrm{GL}_2}(\nu)\otimes_E\mathrm{St}_2^{\infty}}="a"; (34,0)*+{\overline{L}_{\mathrm{GL}_2}(\nu)\otimes_E\mathrm{St}_2^{\infty}}="b";
{\ar@{-}"a";"b"};
\end{xy}
$$
which is a contradiction.

The proof of the second statement is a direct analogue as we observe that $J_{\overline{B_2}}$ also induces the following embedding
\begin{multline*}
\mathrm{Ext}^1_{\mathrm{GL}_2(\mathbb{Q}_p)}\left(\overline{L}_{\mathrm{GL}_2}(\nu), ~\begin{xy}
(0,0)*+{\overline{L}_{\mathrm{GL}_2}(\nu)}="a"; (20,0)*+{\widetilde{I}(s\cdot\nu)}="b"; (45,0)*+{\overline{L}_{\mathrm{GL}_2}(\nu)\otimes_E\mathrm{St}_2^{\infty}}="c"; (72,0)*+{\overline{L}_{\mathrm{GL}_2}(\nu)}="d";
{\ar@{-}"a";"b"}; {\ar@{-}"b";"c"}; {\ar@{-}"c";"d"};
\end{xy}\right)\\
\hookrightarrow\mathrm{Ext}^1_{T_2(\mathbb{Q}_p)}\left(\delta_{T_2,\nu}, ~\delta_{T_2,\nu}\right).
\end{multline*}
\end{proof}

We define $\Sigma_2^+(\nu,\mathscr{L})$ as the unique (up to isomorphism) non-split extension of $\Sigma_{\mathrm{GL}_2}(\nu,\mathscr{L})$ by $\widetilde{I}(s\cdot\nu)$ given by Lemma~\ref{3lemm: ext1}.
\begin{prop}\label{3prop: key result}
We have
$$\mathrm{Ext}^1_{\mathrm{GL}_2(\mathbb{Q}_p)}\left(
\begin{xy}
(0,0)*+{\overline{L}_{\mathrm{GL}_2}(\nu)\otimes_E\mathrm{St}_2^{\infty}}="a"; (26,0)*+{\overline{L}_{\mathrm{GL}_2}(\nu)}="b";
{\ar@{-}"a";"b"};
\end{xy},~ \Sigma_2^+(\nu,\mathscr{L})\right)=0.$$
\end{prop}
\begin{proof}
Assume on the contrary that $V$ is a representation given by a certain non-zero element inside
$$\mathrm{Ext}^1_{\mathrm{GL}_2(\mathbb{Q}_p)}\left(\begin{xy}
(0,0)*+{\overline{L}_{\mathrm{GL}_2}(\nu)\otimes_E\mathrm{St}_2^{\infty}}="a"; (26,0)*+{\overline{L}_{\mathrm{GL}_2}(\nu)}="b";
{\ar@{-}"a";"b"};
\end{xy},~ \Sigma_2^+(\nu,\mathscr{L})\right).$$
We deduce that $V$ has both a central character and an infinitesimal character from Lemma~\ref{3lemm: existence of inf char} and the fact
$$\mathrm{Hom}_{\mathrm{GL}_2(\mathbb{Q}_p)}\left(\begin{xy}
(0,0)*+{\overline{L}_{\mathrm{GL}_2}(\nu)\otimes_E\mathrm{St}_2^{\infty}}="a"; (26,0)*+{\overline{L}_{\mathrm{GL}_2}(\nu)}="b";
{\ar@{-}"a";"b"};
\end{xy},~ \Sigma_2^+(\nu,\mathscr{L})\right)=0.$$
Note that we have
$$\mathrm{Ext}^1_{\mathrm{GL}_2(\mathbb{Q}_p)}(\overline{L}_{\mathrm{GL}_2}(\nu)\otimes_E\mathrm{St}_2^{\infty}, ~I(s\cdot\nu))=\mathrm{Ext}^1_{\mathrm{GL}_2(\mathbb{Q}_p)}(\overline{L}_{\mathrm{GL}_2}(\nu), ~\widetilde{I}(s\cdot\nu))=0,$$
$$\mathrm{dim}_E\mathrm{Ext}^1_{\mathrm{GL}_2(\mathbb{Q}_p)}\left(\overline{L}_{\mathrm{GL}_2}(\nu),~\overline{L}_{\mathrm{GL}_2}(\nu)\otimes_E\mathrm{St}_2^{\infty}\right)=1$$
and
$$\mathrm{dim}_E\mathrm{Ext}^1_{\mathrm{GL}_2(\mathbb{Q}_p)}\left(\overline{L}_{\mathrm{GL}_2}(\nu),~I(s\cdot\nu)\right)=1$$
by a combination of Lemma~3.13 of \cite{BD18} with Lemma~\ref{3lemm: cohomology devissage}, and thus $V$ has a subrepresentation of one of the three following forms
\begin{enumerate}
\item $\begin{xy}
(0,0)*+{\overline{L}_{\mathrm{GL}_2}(\nu)\otimes_E\mathrm{St}_2^{\infty}}="a"; (34,0)*+{\overline{L}_{\mathrm{GL}_2}(\nu)\otimes_E\mathrm{St}_2^{\infty}}="b";
{\ar@{-}"a";"b"};
\end{xy}$;
\item $\begin{xy}
(0,0)*+{\overline{L}_{\mathrm{GL}_2}(\nu)\otimes_E\mathrm{St}_2^{\infty}}="a"; (26,0)*+{I(s\cdot\nu)}="b"; (45,0)*+{\overline{L}_{\mathrm{GL}_2}(\nu)}="c"; (72,0)*+{\overline{L}_{\mathrm{GL}_2}(\nu)\otimes_E\mathrm{St}_2^{\infty}}="d";
{\ar@{-}"a";"b"}; {\ar@{-}"b";"c"}; {\ar@{-}"c";"d"};
\end{xy}$;
\item $\begin{xy}
(0,0)*+{\overline{L}_{\mathrm{GL}_2}(\nu)\otimes_E\mathrm{St}_2^{\infty}}="a"; (26,0)*+{I(s\cdot\nu)}="b"; (45,0)*+{\overline{L}_{\mathrm{GL}_2}(\nu)}="c"; (64,0)*+{\widetilde{I}(s\cdot\nu)}="d";  (90,0)*+{\overline{L}_{\mathrm{GL}_2}(\nu)\otimes_E\mathrm{St}_2^{\infty}}="e"; (115,0)*+{\overline{L}_{\mathrm{GL}_2}(\nu)}="f";
{\ar@{-}"a";"b"}; {\ar@{-}"b";"c"}; {\ar@{-}"c";"d"}; {\ar@{-}"d";"e"}; {\ar@{-}"e";"f"};
\end{xy}.$
\end{enumerate}
In the first case, we know from Proposition~4.7 of \cite{Schr11} and the main result of \cite{Or05} that
$$\mathrm{Ext}^1_{\mathrm{GL}_2(\mathbb{Q}_p),\nu}\left(\overline{L}_{\mathrm{GL}_2}(\nu)\otimes_E\mathrm{St}_2^{\infty}, ~\overline{L}_{\mathrm{GL}_2}(\nu)\otimes_E\mathrm{St}_2^{\infty}\right)=0$$
and therefore this case is impossible due to the existence of central character for $V$ (and hence for its subrepresentations). In the second case, we deduce from Lemma~\ref{3lemm: no inf char} a contradiction as $V$ has an infinitesimal character. In the third case, we thus know that $V$ has a quotient representation of the form
$$\begin{xy}
(0,0)*+{\overline{L}_{\mathrm{GL}_2}(\nu)}="a"; (20,0)*+{\widetilde{I}(s\cdot\nu)}="b"; (45,0)*+{\overline{L}_{\mathrm{GL}_2}(\nu)\otimes_E\mathrm{St}_2^{\infty}}="c"; (72,0)*+{\overline{L}_{\mathrm{GL}_2}(\nu)}="d";
{\ar@{-}"a";"b"}; {\ar@{-}"b";"c"}; {\ar@{-}"c";"d"};
\end{xy}$$
which can not have an infinitesimal character due to Lemma~\ref{3lemm: no inf char}, a contradiction again. Hence we finish the proof.
\end{proof}
\begin{rema}\label{3rema: second version}
Note that the argument in Proposition~\ref{3prop: key result} actually implies that
$$\mathrm{Ext}^1_{\mathrm{GL}_2(\mathbb{Q}_p)}\left(\begin{xy}
(0,0)*+{\overline{L}_{\mathrm{GL}_2}(\nu)\otimes_E\mathrm{St}_2^{\infty}}="a"; (26,0)*+{\overline{L}_{\mathrm{GL}_2}(\nu)}="b";
{\ar@{-}"a";"b"};
\end{xy},~ \begin{xy}
(0,0)*+{I(s\cdot\nu)}="b"; (20,0)*+{\overline{L}_{\mathrm{GL}_2}(\nu)}="c"; (40,0)*+{\widetilde{I}(s\cdot\nu)}="d";
{\ar@{-}"b";"c"}; {\ar@{-}"c";"d"};
\end{xy}\right)=0$$
and we can show by the same method that
$$\mathrm{Ext}^1_{\mathrm{GL}_2(\mathbb{Q}_p)}\left(\begin{xy}
(26,0)*+{\overline{L}_{\mathrm{GL}_2}(\nu)\otimes_E\mathrm{St}_2^{\infty}}="b"; (0,0)*+{\overline{L}_{\mathrm{GL}_2}(\nu)}="a";
{\ar@{-}"a";"b"};
\end{xy},~ \begin{xy}
(0,0)*+{\widetilde{I}(s\cdot\nu)}="b"; (26,0)*+{\overline{L}_{\mathrm{GL}_2}(\nu)\otimes_E\mathrm{St}_2^{\infty}}="c"; (50,0)*+{I(s\cdot\nu)}="d";
{\ar@{-}"b";"c"}; {\ar@{-}"c";"d"};
\end{xy}\right)=0.$$
\end{rema}

\section{Computations of $\mathrm{Ext}$ I}\label{3section: computation}
In this section, we are going to compute various $\mathrm{Ext}$\text{-}groups based on known results on group cohomology in Section~5.2 and 5.3 of \cite{Bre17}.
\begin{prop}\label{3prop: locally algebraic extension}
The following spaces are one dimensional
$$
\begin{xy}
(0,0)*+{\mathrm{Ext}^1_{\mathrm{GL}_3(\mathbb{Q}_p),\lambda}\left(\overline{L}(\lambda), ~\overline{L}(\lambda)\otimes_Ev_{P_i}^{\infty}\right)}; (70,0)*+{\mathrm{Ext}^1_{\mathrm{GL}_3(\mathbb{Q}_p),\lambda}\left(\overline{L}(\lambda)\otimes_Ev_{P_i}^{\infty}, ~\overline{L}(\lambda)\right)};
(6.5,-6)*+{\mathrm{Ext}^1_{\mathrm{GL}_3(\mathbb{Q}_p),\lambda}\left(\overline{L}(\lambda)\otimes_E\mathrm{St}_3^{\infty}, ~\overline{L}(\lambda)\otimes_Ev_{P_i}^{\infty}\right)};
(76.5,-6)*+{\mathrm{Ext}^1_{\mathrm{GL}_3(\mathbb{Q}_p),\lambda}\left(\overline{L}(\lambda)\otimes_Ev_{P_i}^{\infty}, ~\overline{L}(\lambda)\otimes_E\mathrm{St}_3^{\infty}\right)};
(0.8,-12)*+{\mathrm{Ext}^2_{\mathrm{GL}_3(\mathbb{Q}_p),\lambda}\left(\overline{L}(\lambda)\otimes_E\mathrm{St}_3^{\infty}, ~\overline{L}(\lambda)\right)};
(70.8,-12)*+{\mathrm{Ext}^2_{\mathrm{GL}_3(\mathbb{Q}_p),\lambda}\left(\overline{L}(\lambda), ~\overline{L}(\lambda)\otimes_E\mathrm{St}_3^{\infty}\right)};
(6,-18)*+{\mathrm{Ext}^2_{\mathrm{GL}_3(\mathbb{Q}_p),\lambda}\left(\overline{L}(\lambda)\otimes_Ev_{P_1}^{\infty}, ~\overline{L}(\lambda)\otimes_Ev_{P_2}^{\infty}\right)};
(76,-18)*+{\mathrm{Ext}^2_{\mathrm{GL}_3(\mathbb{Q}_p),\lambda}\left(\overline{L}(\lambda)\otimes_Ev_{P_2}^{\infty}, ~\overline{L}(\lambda)\otimes_Ev_{P_1}^{\infty}\right)};
\end{xy}
$$
for $i=1,2$. Moreover, we have
$$\mathrm{Ext}^k_{\mathrm{GL}_3(\mathbb{Q}_p),\lambda}\left(V_1, ~V_2\right)=0$$
in all the other cases where $1\leq k\leq 2$ and $V_1,V_2\in\{\overline{L}(\lambda), \overline{L}(\lambda)\otimes_Ev_{P_1}^{\infty}, \overline{L}(\lambda)\otimes_Ev_{P_2}^{\infty}, \overline{L}(\lambda)\otimes_E\mathrm{St}_3^{\infty}\}$.
\end{prop}
\begin{proof}
This follows from a special case of Proposition~4.7 of \cite{Schr11} together with the main result of \cite{Or05}.
\end{proof}

\begin{lemm}\label{3lemm: vanishing locally algebraic}
We have
$$
\begin{xy}
(0,0)*+{\mathrm{Ext}^k_{\mathrm{GL}_3(\mathbb{Q}_p),\lambda}\left(
\begin{xy}
(0,0)*+{\overline{L}(\lambda)\otimes_Ev_{P_i}^{\infty}}="a"; (20,0)*+{\overline{L}(\lambda)}="b";
{\ar@{-}"a";"b"};
\end{xy},~\overline{L}(\lambda)\otimes_E\mathrm{St}_3^{\infty}\right)}; (50,0)*+{=0};
(0.35,-6)*+{\mathrm{Ext}^k_{\mathrm{GL}_3(\mathbb{Q}_p),\lambda}\left(\begin{xy}
(0,0)*+{\overline{L}(\lambda)\otimes_Ev_{P_i}^{\infty}}="a"; (27,0)*+{\overline{L}(\lambda)\otimes_E\mathrm{St}_3^{\infty}}="b";
{\ar@{-}"a";"b"};
\end{xy},~\overline{L}(\lambda)\right)}; (50,-6)*+{=0};
(0.95,-12)*+{\mathrm{Ext}^k_{\mathrm{GL}_3(\mathbb{Q}_p),\lambda}\left(\begin{xy}
(20,0)*+{\overline{L}(\lambda)\otimes_Ev_{P_i}^{\infty}}="b"; (0,0)*+{\overline{L}(\lambda)}="a";
{\ar@{-}"a";"b"};
\end{xy},~\overline{L}(\lambda)\otimes_Ev_{P_{3-i}}^{\infty}\right)}; (50,-12)*+{=0};
\end{xy}
$$
for $i=1,2$ and $k=1,2$.
\end{lemm}
\begin{proof}
It is sufficient to prove that
\begin{equation}\label{3first algebraic vanishing}
\mathrm{Ext}^1_{\mathrm{GL}_3(\mathbb{Q}_p),\lambda}\left(\begin{xy}
(0,0)*+{\overline{L}(\lambda)\otimes_Ev_{P_i}^{\infty}}="a"; (20,0)*+{\overline{L}(\lambda)}="b";
{\ar@{-}"a";"b"};
\end{xy},~\overline{L}(\lambda)\otimes_E\mathrm{St}_3^{\infty}\right)=0
\end{equation}
and
\begin{equation}\label{3second algebraic vanishing}
\mathrm{Ext}^2_{\mathrm{GL}_3(\mathbb{Q}_p),\lambda}\left(\begin{xy}
(0,0)*+{\overline{L}(\lambda)\otimes_Ev_{P_i}^{\infty}}="a"; (20,0)*+{\overline{L}(\lambda)}="b";
{\ar@{-}"a";"b"};
\end{xy},~\overline{L}(\lambda)\otimes_E\mathrm{St}_3^{\infty}\right)=0
\end{equation}
as the other cases are similar. We observe that (\ref{3first algebraic vanishing}) is equivalent to the non-existence of a uniserial representation of the form
$$
\begin{xy}
(0,0)*+{\overline{L}(\lambda)\otimes_E\mathrm{St}_3^{\infty}}="a"; (27,0)*+{\overline{L}(\lambda)\otimes_Ev_{P_i}^{\infty}}="b"; (47,0)*+{\overline{L}(\lambda)}="c";
{\ar@{-}"a";"b"}; {\ar@{-}"b";"c"};
\end{xy}
$$
which is again equivalent to the vanishing
\begin{equation}\label{3third algebraic vanishing}
\mathrm{Ext}^1_{\mathrm{GL}_3(\mathbb{Q}_p),\lambda}\left(\overline{L}(\lambda),~\begin{xy}
(0,0)*+{\overline{L}(\lambda)\otimes_E\mathrm{St}_3^{\infty}}="a"; (27,0)*+{\overline{L}(\lambda)\otimes_Ev_{P_i}^{\infty}}="b";
{\ar@{-}"a";"b"};
\end{xy}\right)=0
\end{equation}
according to the fact
$$\mathrm{Ext}^1_{\mathrm{GL}_3(\mathbb{Q}_p),\lambda}\left(\overline{L}(\lambda),~\overline{L}(\lambda)\otimes_E\mathrm{St}_3^{\infty}\right)=0$$
due to Proposition~\ref{3prop: locally algebraic extension}. The short exact sequence
$$\left(\begin{xy}
(0,0)*+{\overline{L}(\lambda)\otimes_E\mathrm{St}_3^{\infty}}="a"; (27,0)*+{\overline{L}(\lambda)\otimes_Ev_{P_i}^{\infty}}="b";
{\ar@{-}"a";"b"};
\end{xy}\right)\hookrightarrow\mathcal{F}_{P_i}^{\mathrm{GL}_3}\left(M_i(-\lambda),~\pi_{1,3}^{\infty}\right)\twoheadrightarrow C^2_{s_{3-i},s_{3-i}s_i}$$
induces an injection
$$\mathrm{Ext}^1_{\mathrm{GL}_3(\mathbb{Q}_p),\lambda}\left(\overline{L}(\lambda),~\begin{xy}
(0,0)*+{\overline{L}(\lambda)\otimes_E\mathrm{St}_3^{\infty}}="a"; (27,0)*+{\overline{L}(\lambda)\otimes_Ev_{P_i}^{\infty}}="b";
{\ar@{-}"a";"b"};
\end{xy}\right)\hookrightarrow\mathrm{Ext}^1_{\mathrm{GL}_3(\mathbb{Q}_p),\lambda}\left(\overline{L}(\lambda),~\mathcal{F}_{P_i}^{\mathrm{GL}_3}\left(M_i(-\lambda),~\pi_{i,3}^{\infty}\right)\right).$$
Therefore (\ref{3third algebraic vanishing}) follows from Lemma~\ref{3lemm: cohomology devissage} and the facts that
$$\mathrm{Ext}^1_{L_i(\mathbb{Q}_p),\lambda}\left(H_0(N_i, \overline{L}(\lambda)),~\overline{L}_i(\lambda)\otimes_E\pi_{i,3}^{\infty}\right)=\mathrm{Hom}_{L_i(\mathbb{Q}_p),\lambda}\left(H_1(N_i, \overline{L}(\lambda)),~\overline{L}_i(\lambda)\otimes_E\pi_{i,3}^{\infty}\right)=0.$$
On the other hand, the short exact sequence
$$\overline{L}(\lambda)\otimes_Ev_{P_i}^{\infty}\hookrightarrow\left(\begin{xy}
(0,0)*+{\overline{L}(\lambda)\otimes_Ev_{P_i}^{\infty}}="a"; (20,0)*+{\overline{L}(\lambda)}="b";
{\ar@{-}"a";"b"};
\end{xy}\right)\twoheadrightarrow\overline{L}(\lambda)$$
induces a long exact sequence
\begin{multline*}
\mathrm{Ext}^1_{\mathrm{GL}_3(\mathbb{Q}_p),\lambda}\left(\overline{L}(\lambda),~\overline{L}(\lambda)\otimes_E\mathrm{St}_3^{\infty}\right)\hookrightarrow \mathrm{Ext}^1_{\mathrm{GL}_3(\mathbb{Q}_p),\lambda}\left(\begin{xy}
(0,0)*+{\overline{L}(\lambda)\otimes_Ev_{P_i}^{\infty}}="a"; (20,0)*+{\overline{L}(\lambda)}="b";
{\ar@{-}"a";"b"};
\end{xy},~\overline{L}(\lambda)\otimes_E\mathrm{St}_3^{\infty}\right)\\ \rightarrow\mathrm{Ext}^1_{\mathrm{GL}_3(\mathbb{Q}_p),\lambda}\left(\overline{L}(\lambda)\otimes_Ev_{P_i}^{\infty},~\overline{L}(\lambda)\otimes_E\mathrm{St}_3^{\infty}\right)
\rightarrow\mathrm{Ext}^2_{\mathrm{GL}_3(\mathbb{Q}_p),\lambda}\left(\overline{L}(\lambda),~\overline{L}(\lambda)\otimes_E\mathrm{St}_3^{\infty}\right)\\ \rightarrow\mathrm{Ext}^2_{\mathrm{GL}_3(\mathbb{Q}_p),\lambda}\left(\begin{xy}
(0,0)*+{\overline{L}(\lambda)\otimes_Ev_{P_i}^{\infty}}="a"; (20,0)*+{\overline{L}(\lambda)}="b";
{\ar@{-}"a";"b"};
\end{xy},~\overline{L}(\lambda)\otimes_E\mathrm{St}_3^{\infty}\right) \rightarrow\mathrm{Ext}^2_{\mathrm{GL}_3(\mathbb{Q}_p),\lambda}\left(\overline{L}(\lambda)\otimes_Ev_{P_i}^{\infty},~\overline{L}(\lambda)\otimes_E\mathrm{St}_3^{\infty}\right)
\end{multline*}
and thus we can deduce (\ref{3second algebraic vanishing}) from Proposition~\ref{3prop: locally algebraic extension} and (\ref{3first algebraic vanishing}).
\end{proof}

We define $W_0$ as the unique locally algebraic representation of length three satisfying
$$\mathrm{soc}_{\mathrm{GL}_3(\mathbb{Q}_p)}(W_0)=\overline{L}(\lambda)\otimes_E\left(v_{P_1}^{\infty}\oplus v_{P_2}^{\infty}\right)\mbox{ and }\mathrm{cosoc}_{\mathrm{GL}_3(\mathbb{Q}_p)}(W_0)=\overline{L}(\lambda).$$
We also define the (unique up to isomorphism) locally algebraic representation of the form
$$W_i:=\begin{xy}
(0,0)*+{\overline{L}(\lambda)\otimes_Ev_{P_i}^{\infty}}="a"; (20,0)*+{\overline{L}(\lambda)}="b";
{\ar@{-}"a";"b"};
\end{xy}$$
for each $i=1,2$

\begin{lemm}\label{3lemm: vanishing locally algebraic main}
We have
$$\mathrm{dim}_E\mathrm{Ext}^1_{\mathrm{GL}_3(\mathbb{Q}_p),\lambda}\left(W_0,~\overline{L}(\lambda)\otimes_E\mathrm{St}_3^{\infty}\right)=1$$
and
$$\mathrm{Ext}^2_{\mathrm{GL}_3(\mathbb{Q}_p),\lambda}\left(W_0,~\overline{L}(\lambda)\otimes_E\mathrm{St}_3^{\infty}\right)=0.$$
\end{lemm}
\begin{proof}
The short exact sequence
$$\overline{L}(\lambda)\otimes_Ev_{P_1}^{\infty}\hookrightarrow W_0\twoheadrightarrow W_2$$
induces a long exact sequence
\begin{multline*}
\mathrm{Ext}^1_{\mathrm{GL}_3(\mathbb{Q}_p),\lambda}\left(\overline{L}(\lambda)\otimes_Ev_{P_1}^{\infty},~\overline{L}(\lambda)\otimes_E\mathrm{St}_3^{\infty}\right)\hookrightarrow \mathrm{Ext}^1_{\mathrm{GL}_3(\mathbb{Q}_p),\lambda}\left(W_0,~\overline{L}(\lambda)\otimes_E\mathrm{St}_3^{\infty}\right)\\ \rightarrow\mathrm{Ext}^1_{\mathrm{GL}_3(\mathbb{Q}_p),\lambda}\left(W_2,~\overline{L}(\lambda)\otimes_E\mathrm{St}_3^{\infty}\right)
\rightarrow\mathrm{Ext}^2_{\mathrm{GL}_3(\mathbb{Q}_p),\lambda}\left(\overline{L}(\lambda)\otimes_Ev_{P_1}^{\infty},~\overline{L}(\lambda)\otimes_E\mathrm{St}_3^{\infty}\right)\\ \rightarrow\mathrm{Ext}^2_{\mathrm{GL}_3(\mathbb{Q}_p),\lambda}\left(W_0,~\overline{L}(\lambda)\otimes_E\mathrm{St}_3^{\infty}\right) \rightarrow\mathrm{Ext}^2_{\mathrm{GL}_3(\mathbb{Q}_p),\lambda}\left(W_2,~\overline{L}(\lambda)\otimes_E\mathrm{St}_3^{\infty}\right)
\end{multline*}
which finishes the proof by Proposition~\ref{3prop: locally algebraic extension}, (\ref{3first algebraic vanishing}) and (\ref{3second algebraic vanishing}).
\end{proof}

We define the following subsets of $\Omega$:
$$
\begin{xy}
(-6,0)*+{\Omega_1\left(\overline{L}(\lambda)\right)}; (20,0)*+{:=}; (57.5,0)*+{\{\overline{L}(\lambda)\otimes_Ev_{P_1}^{\infty},~\overline{L}(\lambda)\otimes_Ev_{P_2}^{\infty},~C^1_{s_1,1},~C^1_{s_2,1}\}};
(0,-6)*+{\Omega_1\left(\overline{L}(\lambda)\otimes_Ev_{P_1}^{\infty}\right)}; (20,-6)*+{:=}; (60,-6)*+{\{\overline{L}(\lambda),~\overline{L}(\lambda)\otimes_E\mathrm{St}_3^{\infty},~C^2_{s_1,1},~C_{s_2,s_2},~C^1_{s_1,s_1s_2}\}};
(0,-12)*+{\Omega_1\left(\overline{L}(\lambda)\otimes_Ev_{P_2}^{\infty}\right)}; (20,-12)*+{:=}; (60,-12)*+{\{\overline{L}(\lambda),~\overline{L}(\lambda)\otimes_E\mathrm{St}_3^{\infty},~C^2_{s_2,1},~C_{s_1,s_1},~C^1_{s_2,s_2s_1}\}};
(0.5,-18)*+{\Omega_1\left(\overline{L}(\lambda)\otimes_E\mathrm{St}_3^{\infty}\right)}; (20,-18)*+{:=}; (61.5,-18)*+{\{\overline{L}(\lambda)\otimes_Ev_{P_1}^{\infty},~\overline{L}(\lambda)\otimes_Ev_{P_2}^{\infty},~C^2_{s_1,s_1s_2},~C^2_{s_2,s_2s_1}\}};
(-6,-24)*+{\Omega_2\left(\overline{L}(\lambda)\right)}; (20,-24)*+{:=}; (60.5,-24)*+{\{\overline{L}(\lambda)\otimes_E\mathrm{St}_3^{\infty},~C^2_{s_1,1},~C^2_{s_2,1},~C^1_{s_1s_2,1},~C^1_{s_2s_1,1}\}};
(0,-30)*+{\Omega_2\left(\overline{L}(\lambda)\otimes_Ev_{P_1}^{\infty}\right)}; (20,-30)*+{:=}; (62.5,-30)*+{\{\overline{L}(\lambda)\otimes_Ev_{P_2}^{\infty},~C^1_{s_1,1},~C^2_{s_1,s_1s_2},~C^2_{s_1s_2,1},~C_{s_2s_1,s_2}\}};
(0,-36)*+{\Omega_2\left(\overline{L}(\lambda)\otimes_Ev_{P_2}^{\infty}\right)}; (20,-36)*+{:=}; (62.5,-36)*+{\{\overline{L}(\lambda)\otimes_Ev_{P_1}^{\infty},~C^1_{s_2,1},~C^2_{s_2,s_2s_1},~C^2_{s_2s_1,1},~C_{s_1s_2,s_1}\}};
(0.5,-42)*+{\Omega_2\left(\overline{L}(\lambda)\otimes_E\mathrm{St}_3^{\infty}\right)}; (20,-42)*+{:=}; (62,-42)*+{\{\overline{L}(\lambda),~C^1_{s_1,s_1s_2},~C^1_{s_2,s_2s_1},~C^2_{s_1s_2,s_1s_2},~C^2_{s_2s_1,s_2s_1}\}};
\end{xy}
$$

\begin{prop}\label{3prop: list of N coh}
We have all explicit formula for
$$H_k\left(N_i,~\mathcal{F}_{P_j}^{\mathrm{GL}_3}(M,~\pi_j^{\infty})\right)$$
for each smooth admissible representation $\pi_j^{\infty}$ of $L_j(\mathbb{Q}_p)$, each
$$M\in \{L(-\lambda),~M_j(-\lambda), L(-s_{3-j}\cdot\lambda), ~M_j(-s_{3-j}\cdot\lambda),~L(-s_{3-j}s_j\cdot\lambda)\}$$
and each $0\leq k\leq 2$, $i,j=1,2$.
\end{prop}
\begin{proof}
This follows directly from Section~5.2 and 5.3 of \cite{Bre17}.
\end{proof}
\begin{lemm}\label{3lemm: nonvanishing ext1}
For
$$V_0\in\{\overline{L}(\lambda),~\overline{L}(\lambda)\otimes_Ev_{P_1}^{\infty},~\overline{L}(\lambda)\otimes_Ev_{P_2}^{\infty},~\overline{L}(\lambda)\otimes_E\mathrm{St}_3^{\infty}\},$$
we have
$$\mathrm{dim}_E\mathrm{Ext}^1_{\mathrm{GL}_3(\mathbb{Q}_p),\lambda}\left(V_0,~V\right)=1$$
if $V\in\Omega_1(V_0)$ and
$$\mathrm{Ext}^1_{\mathrm{GL}_3(\mathbb{Q}_p),\lambda}\left(V_0,~V\right)=0$$
if $V\in\Omega\setminus\Omega_1(V_0)$.
\end{lemm}
\begin{proof}
We only prove the statements for $V_0=\overline{L}(\lambda)$ as other cases are similar. If
$$V\in\{\overline{L}(\lambda),~\overline{L}(\lambda)\otimes_Ev_{P_1}^{\infty},~\overline{L}(\lambda)\otimes_Ev_{P_2}^{\infty},~\overline{L}(\lambda)\otimes_E\mathrm{St}_3^{\infty}\}$$
then the conclusion follows from Proposition~\ref{3prop: locally algebraic extension}. If
$$V=\mathcal{F}_{P_j}^{\mathrm{GL}_3}(L(-s_{3-j}s_j\cdot\lambda),~\pi_j^{\infty})$$
for a smooth irreducible representation $\pi_j^{\infty}$ and $j=1$ or $2$, then it follows from Lemma~\ref{3lemm: cohomology devissage} that
\begin{multline}\label{3explicit spectral sequence 1}
\mathrm{Ext}^1_{L_j(\mathbb{Q}_p),\lambda}\left(H_0(N_j,~\overline{L}(\lambda)),~\overline{L}_j(s_{3-j}s_j\cdot\lambda)\otimes_E\pi_j^{\infty}\right)\hookrightarrow\mathrm{Ext}^1_{\mathrm{GL}_3(\mathbb{Q}_p),\lambda}\left(\overline{L}(\lambda),~V\right)\\
\rightarrow \mathrm{Hom}_{L_j(\mathbb{Q}_p),\lambda}\left(H_1(N_j,~\overline{L}(\lambda)),~\overline{L}_j(s_{3-j}s_j\cdot\lambda)\otimes_E\pi_j^{\infty}\right)\\
\rightarrow\mathrm{Ext}^2_{L_j(\mathbb{Q}_p),\lambda}\left(H_0(N_j,~\overline{L}(\lambda)),~\overline{L}_j(s_{3-j}s_j\cdot\lambda)\otimes_E\pi_j^{\infty}\right).
\end{multline}
It follows from Proposition~\ref{3prop: list of N coh} and (\ref{3explicit spectral sequence 1}) that
\begin{multline}\label{3explicit spectral sequence 2}
\mathrm{Ext}^1_{L_j(\mathbb{Q}_p),\lambda}\left(\overline{L}_j(\lambda),~\overline{L}_j(s_{3-j}s_j\cdot\lambda)\otimes_E\pi_j^{\infty}\right)\hookrightarrow\mathrm{Ext}^1_{\mathrm{GL}_3(\mathbb{Q}_p),\lambda}\left(\overline{L}(\lambda),~V\right)\\
\rightarrow \mathrm{Hom}_{L_j(\mathbb{Q}_p),\lambda}\left(\overline{L}_j(s_{3-j}\cdot\lambda),~\overline{L}_j(s_{3-j}s_j\cdot\lambda)\otimes_E\pi_j^{\infty}\right).
\end{multline}
We notice that $Z(L_j(\mathbb{Q}_p))$ acts via different characters on $\overline{L}_j(\lambda)$, $\overline{L}_j(s_{3-j}\cdot\lambda)$ and $\overline{L}_j(s_{3-j}s_j\cdot\lambda)\otimes_E\pi_j^{\infty}$, and thus we have the equalities
$$
\begin{array}{cccc}
\mathrm{Ext}^1_{L_j(\mathbb{Q}_p),\lambda}\left(\overline{L}_j(\lambda),~\overline{L}_j(s_{3-j}s_j\cdot\lambda)\otimes_E\pi_j^{\infty}\right)&=&0&\\
\mathrm{Hom}_{L_j(\mathbb{Q}_p),\lambda}\left(\overline{L}_j(s_{3-j}\cdot\lambda),~\overline{L}_j(s_{3-j}s_j\cdot\lambda)\otimes_E\pi_j^{\infty}\right)&=&0&
\end{array}
$$
which imply that
\begin{equation}\label{3vanishing of ext 1.1}
\mathrm{Ext}^1_{\mathrm{GL}_3(\mathbb{Q}_p),\lambda}\left(\overline{L}(\lambda),~\mathcal{F}_{P_j}^{\mathrm{GL}_3}(L(-s_{3-j}s_j\cdot\lambda),~\pi_j^{\infty})\right)=0
\end{equation}
for each $\pi_j^{\infty}$ and $j=1,2$. If
$$V=\mathcal{F}_{P_j}^{\mathrm{GL}_3}(L(-s_{3-j}\cdot\lambda),~\pi_j^{\infty})$$
for a smooth irreducible representation $\pi_j^{\infty}$ and $j=1$ or $2$, then the short exact sequence
$$\mathcal{F}_{P_j}^{\mathrm{GL}_3}(L(-s_{3-j}\cdot\lambda),~\pi_j^{\infty})\hookrightarrow\mathcal{F}_{P_j}^{\mathrm{GL}_3}(M_j(-s_{3-j}\cdot\lambda),~\pi_j^{\infty})\twoheadrightarrow\mathcal{F}_{P_j}^{\mathrm{GL}_3}(L(-s_{3-j}s_j\cdot\lambda),~\pi_j^{\infty})$$
induces a long exact sequence
\begin{multline*}
\mathrm{Ext}^1_{\mathrm{GL}_3(\mathbb{Q}_p),\lambda}\left(\overline{L}(\lambda),~V\right)\hookrightarrow\mathrm{Ext}^1_{\mathrm{GL}_3(\mathbb{Q}_p),\lambda}\left(\overline{L}(\lambda),~\mathcal{F}_{P_j}^{\mathrm{GL}_3}(M_j(-s_{3-j}\cdot\lambda),~\pi_j^{\infty})\right)\\
\rightarrow\mathrm{Ext}^1_{\mathrm{GL}_3(\mathbb{Q}_p),\lambda}\left(\overline{L}(\lambda),~\mathcal{F}_{P_j}^{\mathrm{GL}_3}(L(-s_{3-j}s_j\cdot\lambda),~\pi_j^{\infty})\right)
\end{multline*}
which implies an isomorphism
\begin{equation}\label{3isomorphism of ext 1.2}
\mathrm{Ext}^1_{\mathrm{GL}_3(\mathbb{Q}_p),\lambda}\left(\overline{L}(\lambda),~V\right)\xrightarrow{\sim}\mathrm{Ext}^1_{\mathrm{GL}_3(\mathbb{Q}_p),\lambda}\left(\overline{L}(\lambda),~\mathcal{F}_{P_j}^{\mathrm{GL}_3}(M_j(-s_{3-j}\cdot\lambda),~\pi_j^{\infty})\right)
\end{equation}
by (\ref{3vanishing of ext 1.1}). It follows from Proposition~\ref{3prop: list of N coh} and Lemma~\ref{3lemm: cohomology devissage} that
\begin{multline}\label{3explicit spectral sequence 3}
\mathrm{Ext}^1_{L_j(\mathbb{Q}_p),\lambda}\left(\overline{L}_j(\lambda),~\overline{L}_j(s_{3-j}\cdot\lambda)\otimes_E\pi_j^{\infty}\right)\hookrightarrow\mathrm{Ext}^1_{\mathrm{GL}_3(\mathbb{Q}_p),\lambda}\left(\overline{L}(\lambda),~V\right)\\
\rightarrow \mathrm{Hom}_{L_j(\mathbb{Q}_p),\lambda}\left(\overline{L}_j(s_{3-j}\cdot\lambda),~\overline{L}_j(s_{3-j}\cdot\lambda)\otimes_E\pi_j^{\infty}\right)
\rightarrow\mathrm{Ext}^2_{L_j(\mathbb{Q}_p),\lambda}\left(\overline{L}_j(\lambda),~\overline{L}_j(s_{3-j}\cdot\lambda)\otimes_E\pi_j^{\infty}\right).
\end{multline}
As $Z(L_j(\mathbb{Q}_p))$ acts via different characters on $\overline{L}_j(\lambda)$ and $\overline{L}_j(s_{3-j}\cdot\lambda)\otimes_E\pi_j^{\infty}$, we have the equalities
$$
\begin{array}{cccc}
\mathrm{Ext}^1_{L_j(\mathbb{Q}_p),\lambda}\left(\overline{L}_j(\lambda),~\overline{L}_j(s_{3-j}s_j\cdot\lambda)\otimes_E\pi_j^{\infty}\right)&=&0&\\
\mathrm{Ext}^2_{L_j(\mathbb{Q}_p),\lambda}\left(\overline{L}_j(\lambda),~\overline{L}_j(s_{3-j}s_j\cdot\lambda)\otimes_E\pi_j^{\infty}\right)&=&0&
\end{array}
$$
which imply that
\begin{equation}\label{3isomorphism of ext 1.3}
\mathrm{Ext}^1_{\mathrm{GL}_3(\mathbb{Q}_p),\lambda}\left(\overline{L}(\lambda),~V\right)
\xrightarrow{\sim} \mathrm{Hom}_{L_j(\mathbb{Q}_p),\lambda}\left(\overline{L}_j(s_{3-j}\cdot\lambda),~\overline{L}_j(s_{3-j}\cdot\lambda)\otimes_E\pi_j^{\infty}\right).
\end{equation}
It is then obvious that
$$\mathrm{Hom}_{L_j(\mathbb{Q}_p),\lambda}\left(\overline{L}_j(s_{3-j}\cdot\lambda),~\overline{L}_j(s_{3-j}\cdot\lambda)\otimes_E\pi_j^{\infty}\right)=0$$
for each smooth irreducible $\pi_j^{\infty}\neq 1_{L_j}$, and therefore
$$\mathrm{dim}_E\mathrm{Ext}^1_{\mathrm{GL}_3(\mathbb{Q}_p),\lambda}\left(\overline{L}(\lambda),~\mathcal{F}_{P_j}^{\mathrm{GL}_3}(L(-s_{3-j}\cdot\lambda),~1_{L_j})\right)=1$$
and
$$\mathrm{Ext}^1_{\mathrm{GL}_3(\mathbb{Q}_p),\lambda}\left(\overline{L}(\lambda),~\mathcal{F}_{P_j}^{\mathrm{GL}_3}(L(-s_{3-j}\cdot\lambda),~1_{L_j})\right)=0$$
for each smooth irreducible $\pi_j^{\infty}\neq 1_{L_j}$. Finally, similar methods together with Proposition~\ref{3prop: list of N coh} also show that
$$\mathrm{Ext}^1_{\mathrm{GL}_3(\mathbb{Q}_p),\lambda}\left(\overline{L}(\lambda),~\mathcal{F}_B^{\mathrm{GL}_3}(L(-s_1s_2s_1\cdot\lambda),~\chi_w^{\infty})\right)=0$$
for each $w\in W$.
\end{proof}

We define
$$\Omega^-:=\Omega\setminus\{\overline{L}(\lambda), ~\overline{L}(\lambda)\otimes_Ev_{P_1}^{\infty}, ~\overline{L}(\lambda)\otimes_Ev_{P_2}^{\infty}, ~\overline{L}(\lambda)\otimes_E\mathrm{St}_3^{\infty}\}.$$
Then we define the following subsets of $\Omega^-$ for $i=1,2$:
$$
\begin{xy}
(0,0)*+{\Omega_1\left(C^1_{s_i,1}\right)}; (20,0)*+{:=}; (55.8,0)*+{\{C^1_{s_is_{3-i},1},~C^2_{s_{3-i}s_i,1},~C^2_{s_i,1},~C^1_{s_i,1}\}};
(0,-6)*+{\Omega_1\left(C^2_{s_i,1}\right)}; (20,-6)*+{:=}; (58,-6)*+{\{C^2_{s_is_{3-i},1},~C_{s_{3-i}s_i,s_{3-i}},~C^1_{s_i,1},~C^2_{s_i,1}\}};
(4,-12)*+{\Omega_1\left(C^1_{s_i,s_is_{3-i}}\right)}; (20,-12)*+{:=}; (68,-12)*+{\{C^1_{s_is_{3-i},s_is_{3-i}},~C_{s_{3-i}s_i,s_{3-i}},~C^2_{s_i,s_is_{3-i}},~C^1_{s_i,s_is_{3-i}}\}};
(4,-18)*+{\Omega_1\left(C^2_{s_i,s_is_{3-i}}\right)}; (20,-18)*+{:=}; (69.4,-18)*+{\{C^2_{s_is_{3-i},s_is_{3-i}},~C^1_{s_{3-i}s_i,s_{3-i}s_i},~C^1_{s_i,s_is_{3-i}},~C^2_{s_i,s_is_{3-i}}\}};
(0.65,-24)*+{\Omega_1\left(C_{s_i,s_i}\right)}; (20,-24)*+{:=}; (63.5,-24)*+{\{C_{s_is_{3-i},s_i},~C^1_{s_{3-i}s_i,1},~C^2_{s_{3-i}s_i,s_{3-i}s_i},~C_{s_i,s_i}\}};
\end{xy}
$$
\begin{lemm}\label{3lemm: nonvanishing ext3}
For
$$V_0\in\{C^1_{s_i,1},~C^2_{s_i,1},~C^1_{s_i,s_is_{3-i}},~C^2_{s_i,s_is_{3-i}},~C_{s_i,s_i}\mid i=1,2\},$$
we have
$$\mathrm{dim}_E\mathrm{Ext}^1_{\mathrm{GL}_3(\mathbb{Q}_p),\lambda}\left(V_0,~V\right)=1$$
if $V\in\Omega_1(V_0)$ and
$$\mathrm{Ext}^1_{\mathrm{GL}_3(\mathbb{Q}_p),\lambda}\left(V_0,~V\right)=0$$
if $V\in\Omega^-\setminus\Omega_1(V_0)$.
\end{lemm}
\begin{proof}
The proof is very similar to that of Lemma~\ref{3lemm: nonvanishing ext3}.
\end{proof}

\begin{lemm}\label{3lemm: nonvanishing ext2}
For
$$V_0\in\{\overline{L}(\lambda),~\overline{L}(\lambda)\otimes_Ev_{P_1}^{\infty},~\overline{L}(\lambda)\otimes_Ev_{P_2}^{\infty},~\overline{L}(\lambda)\otimes_E\mathrm{St}_3^{\infty}\},$$
we have
$$\mathrm{dim}_E\mathrm{Ext}^2_{\mathrm{GL}_3(\mathbb{Q}_p),\lambda}\left(V_0,~V\right)=1$$
if $V\in\Omega_2(V_0)$ and
$$\mathrm{Ext}^2_{\mathrm{GL}_3(\mathbb{Q}_p),\lambda}\left(V_0,~V\right)=0$$
if $V\in\Omega\setminus\Omega_2(V_0)$.
\end{lemm}
\begin{proof}
We only prove the statements for $V_0=\overline{L}(\lambda)$ as other cases are similar. If
$$V\in\{\overline{L}(\lambda),~\overline{L}(\lambda)\otimes_Ev_{P_1}^{\infty},~\overline{L}(\lambda)\otimes_Ev_{P_2}^{\infty},~\overline{L}(\lambda)\otimes_E\mathrm{St}_3^{\infty}\}$$
then the conclusion follows from Proposition~\ref{3prop: locally algebraic extension}. We notice that $Z(L_j(\mathbb{Q}_p))$ acts via different characters on $\overline{L}_j(\lambda)$, $\overline{L}_j(s_{3-j}\cdot\lambda)$ and $\overline{L}_j(s_{3-j}s_j\cdot\lambda)\otimes_E\pi_j^{\infty}$, and thus we have
\begin{equation}\label{3vanishing of ext 2.1}
\begin{xy}
(0,0)*+{\mathrm{Ext}^2_{L_j(\mathbb{Q}_p),\lambda}\left(\overline{L}_j(\lambda),~\overline{L}_j(s_{3-j}s_j\cdot\lambda)\otimes_E\pi_j^{\infty}\right)}; (50,0)*+{=0};
(4.6,-6)*+{\mathrm{Ext}^1_{L_j(\mathbb{Q}_p),\lambda}\left(\overline{L}_j(s_{3-j}\cdot\lambda),~\overline{L}_j(s_{3-j}s_j\cdot\lambda)\otimes_E\pi_j^{\infty}\right)}; (50,-6)*+{=0};
(0,-12)*+{\mathrm{Ext}^3_{L_j(\mathbb{Q}_p),\lambda}\left(\overline{L}_j(\lambda),~\overline{L}_j(s_{3-j}s_j\cdot\lambda)\otimes_E\pi_j^{\infty}\right)}; (50,-12)*+{=0};
\end{xy}
\end{equation}
On the other hand, we notice that
\begin{equation}\label{3vanishing of ext 2.2}
\mathrm{Hom}_{L_j(\mathbb{Q}_p),\lambda}\left(\overline{L}_j(s_{3-j}s_j\cdot\lambda),~\overline{L}_j(s_{3-j}s_j\cdot\lambda)\otimes_E\pi_j^{\infty}\right)=0
\end{equation}
for each smooth irreducible $\pi_j^{\infty}\neq 1_{L_j}$ and
\begin{equation}\label{3nonvanishing of ext 2.3}
\mathrm{dim}_E\mathrm{Hom}_{L_j(\mathbb{Q}_p),\lambda}\left(\overline{L}_j(s_{3-j}s_j\cdot\lambda),~\overline{L}_j(s_{3-j}s_j\cdot\lambda)\right)=1.
\end{equation}
We combine (\ref{3vanishing of ext 2.1}), (\ref{3vanishing of ext 2.2}) and (\ref{3nonvanishing of ext 2.3}) with Lemma~\ref{3lemm: cohomology devissage} and Proposition~\ref{3prop: list of N coh} and deduce that
\begin{equation}\label{3vanishing of ext 2.4}
\mathrm{Ext}^2_{\mathrm{GL}_3(\mathbb{Q}_p),\lambda}\left(\overline{L}(\lambda),~\mathcal{F}_{P_j}^{\mathrm{GL}_3}(L(-s_{3-j}s_j\cdot\lambda),~\pi_j^{\infty})\right)=0
\end{equation}
for each smooth irreducible $\pi_j^{\infty}\neq 1_{L_j}$ and
\begin{equation}\label{3nonvanishing of ext 2.5}
\mathrm{dim}_E\mathrm{Ext}^2_{\mathrm{GL}_3(\mathbb{Q}_p),\lambda}\left(\overline{L}(\lambda),~\mathcal{F}_{P_j}^{\mathrm{GL}_3}(L(-s_{3-j}s_j\cdot\lambda),~1_{L_j})\right)=1
\end{equation}
which finishes the proof if
$$V=\mathcal{F}_{P_j}^{\mathrm{GL}_3}(L(-s_{3-j}s_j\cdot\lambda),~\pi_j^{\infty}).$$
Similarly, we have
\begin{equation}\label{3vanishing of ext 2.6}
\begin{xy}
(0,0)*+{\mathrm{Ext}^2_{L_j(\mathbb{Q}_p),\lambda}\left(\overline{L}_j(\lambda),~\overline{L}_j(s_{3-j}\cdot\lambda)\otimes_E\pi_j^{\infty}\right)}; (50,0)*+{=0};
(7,-6)*+{\mathrm{Hom}_{L_j(\mathbb{Q}_p),\lambda}\left(\overline{L}_j(s_{3-j}s_j\cdot\lambda),~\overline{L}_j(s_{3-j}\cdot\lambda)\otimes_E\pi_j^{\infty}\right)}; (50,-6)*+{=0};
(0,-12)*+{\mathrm{Ext}^3_{L_j(\mathbb{Q}_p),\lambda}\left(\overline{L}_j(\lambda),~\overline{L}_j(s_{3-j}\cdot\lambda)\otimes_E\pi_j^{\infty}\right)}; (50,-12)*+{=0};
\end{xy}
\end{equation}
On the other hand, we have
\begin{equation}\label{3vanishing of ext 2.8}
\mathrm{Ext}^1_{L_j(\mathbb{Q}_p),\lambda}\left(\overline{L}_j(s_{3-i}\cdot\lambda),~\overline{L}_j(s_{3-j}\cdot\lambda)\otimes_E\pi_j^{\infty}\right)=0
\end{equation}
for each smooth irreducible $\pi_j^{\infty}\neq \pi_{j,1}^{\infty}$ and
\begin{equation}\label{3nonvanishing of ext 2.9}
\mathrm{dim}_E\mathrm{Ext}^1_{L_j(\mathbb{Q}_p),\lambda}\left(\overline{L}_j(s_{3-i}\cdot\lambda),~\overline{L}_j(s_{3-j}\cdot\lambda)\otimes_E\pi_{j,1}^{\infty}\right)=1.
\end{equation}
We combine (\ref{3vanishing of ext 2.6}), (\ref{3vanishing of ext 2.8}) and (\ref{3nonvanishing of ext 2.9}) with Lemma~\ref{3lemm: cohomology devissage} and Proposition~\ref{3prop: list of N coh} and deduce that
\begin{equation}\label{3vanishing of ext 2.10}
\mathrm{Ext}^2_{\mathrm{GL}_3(\mathbb{Q}_p),\lambda}\left(\overline{L}(\lambda),~\mathcal{F}_{P_j}^{\mathrm{GL}_3}(M_j(-s_{3-j}\cdot\lambda),~\pi_j^{\infty})\right)=0
\end{equation}
for each smooth irreducible $\pi_j^{\infty}\neq \pi_{j,1}^{\infty}$ and
\begin{equation}\label{3nonvanishing of ext 2.11}
\mathrm{dim}_E\mathrm{Ext}^2_{\mathrm{GL}_3(\mathbb{Q}_p),\lambda}\left(\overline{L}(\lambda),~\mathcal{F}_{P_j}^{\mathrm{GL}_3}(M_j(-s_{3-j}\cdot\lambda),~\pi_{j,1}^{\infty})\right)=1.
\end{equation}
The short exact sequence
$$\mathcal{F}_{P_j}^{\mathrm{GL}_3}(L(-s_{3-j}\cdot\lambda),~\pi_j^{\infty})\hookrightarrow\mathcal{F}_{P_j}^{\mathrm{GL}_3}(M_j(-s_{3-j}\cdot\lambda),~\pi_j^{\infty})\twoheadrightarrow\mathcal{F}_{P_j}^{\mathrm{GL}_3}(L(-s_{3-j}s_j\cdot\lambda),~\pi_j^{\infty})$$
induces a long exact sequence
\begin{multline*}
\mathrm{Ext}^1_{\mathrm{GL}_3(\mathbb{Q}_p),\lambda}\left(\overline{L}(\lambda),~\mathcal{F}_{P_j}^{\mathrm{GL}_3}(L(-s_{3-j}s_j\cdot\lambda),~\pi_j^{\infty})\right)\rightarrow\mathrm{Ext}^2_{\mathrm{GL}_3(\mathbb{Q}_p),\lambda}\left(\overline{L}(\lambda),~\mathcal{F}_{P_j}^{\mathrm{GL}_3}(L(-s_{3-j}\cdot\lambda),~\pi_j^{\infty})\right)\\
\rightarrow\mathrm{Ext}^2_{\mathrm{GL}_3(\mathbb{Q}_p),\lambda}\left(\overline{L}(\lambda),~\mathcal{F}_{P_j}^{\mathrm{GL}_3}(M_j(-s_{3-j}\cdot\lambda),~\pi_j^{\infty})\right)
\rightarrow\mathrm{Ext}^2_{\mathrm{GL}_3(\mathbb{Q}_p),\lambda}\left(\overline{L}(\lambda),~\mathcal{F}_{P_j}^{\mathrm{GL}_3}(L(-s_{3-j}s_j\cdot\lambda),~\pi_j^{\infty})\right)
\end{multline*}
which finishes the proof if
$$V=\mathcal{F}_{P_j}^{\mathrm{GL}_3}(L(-s_{3-j}\cdot\lambda),~\pi_j^{\infty}).$$
Finally, similar methods together with Proposition~\ref{3prop: list of N coh} also show that
$$\mathrm{Ext}^2_{\mathrm{GL}_3(\mathbb{Q}_p),\lambda}\left(\overline{L}(\lambda),~\mathcal{F}_B^{\mathrm{GL}_3}(L(-s_1s_2s_1\cdot\lambda),~\chi_w^{\infty})\right)=0$$
for each $w\in W$.
\end{proof}

\begin{lemm}\label{3lemm: special vanishing 1}
We have
$$
\begin{xy}
(-10,0)*+{\mathrm{Ext}^1_{\mathrm{GL}_3(\mathbb{Q}_p),\lambda}\left(
\begin{xy}
(0,0)*+{\overline{L}(\lambda)\otimes_Ev_{P_i}^{\infty}}="a"; (20,0)*+{\overline{L}(\lambda)}="b";
{\ar@{-}"a";"b"};
\end{xy},~C^2_{s_i,1}\right)}; (50,0)*+{=0};
(0,-6)*+{\mathrm{Ext}^1_{\mathrm{GL}_3(\mathbb{Q}_p),\lambda}\left(\begin{xy}
(0,0)*+{\overline{L}(\lambda)\otimes_Ev_{P_i}^{\infty}}="a"; (27,0)*+{\overline{L}(\lambda)\otimes_E\mathrm{St}_3^{\infty}}="b";
{\ar@{-}"a";"b"};
\end{xy},~C^1_{s_i,s_is_{3-i}}\right)}; (50,-6)*+{=0};
(-10,-12)*+{\mathrm{Ext}^1_{\mathrm{GL}_3(\mathbb{Q}_p),\lambda}\left(\begin{xy}
(20,0)*+{\overline{L}(\lambda)\otimes_Ev_{P_i}^{\infty}}="b"; (0,0)*+{\overline{L}(\lambda)}="a";
{\ar@{-}"a";"b"};
\end{xy},~C^1_{s_i,1}\right)}; (50,-12)*+{=0};
(0,-18)*+{\mathrm{Ext}^1_{\mathrm{GL}_3(\mathbb{Q}_p),\lambda}\left(\begin{xy}
(0,0)*+{\overline{L}(\lambda)\otimes_E\mathrm{St}_3^{\infty}}="a"; (27,0)*+{\overline{L}(\lambda)\otimes_Ev_{P_i}^{\infty}}="b";
{\ar@{-}"a";"b"};
\end{xy},~C^2_{s_i,s_is_{3-i}}\right)}; (50,-18)*+{=0};
\end{xy}
$$
for $i=1,2$.
\end{lemm}
\begin{proof}
We only prove the first vanishing
\begin{equation}\label{3equation vanishing 1.1}
\mathrm{Ext}^1_{\mathrm{GL}_3(\mathbb{Q}_p),\lambda}\left(
W_i,~C^2_{s_i,1}\right)=0
\end{equation}
as the other cases are similar. The embedding
$$C^2_{s_i,1}\hookrightarrow\mathcal{F}_{P_{3-i}}^{\mathrm{GL}_3}(M_{3-i}(-s_i\cdot\lambda),~\pi_{3-i,1}^{\infty})$$
induces an embedding
\begin{equation}\label{3equation vanishing 1.2}
\mathrm{Ext}^1_{\mathrm{GL}_3(\mathbb{Q}_p),\lambda}\left(
W_i,~C^2_{s_i,1}\right)\hookrightarrow \mathrm{Ext}^1_{\mathrm{GL}_3(\mathbb{Q}_p),\lambda}\left(
W_i,~\mathcal{F}_{P_{3-i}}^{\mathrm{GL}_3}(M_{3-i}(-s_i\cdot\lambda),~\pi_{3-i,1}^{\infty})\right).
\end{equation}
It follows from Proposition~\ref{3prop: list of N coh} that
\begin{equation}\label{3equation vanishing 1.3}
\begin{array}{cccc}
H_0(N_{3-i},~W_i)&=&\overline{L}_{3-i}(\lambda)\otimes_E\left(i_{B\cap L_{3-i}}^{L_{3-i}}(\chi_{s_{3-i}}^{\infty})\oplus \mathfrak{d}_{P_{3-i}}^{\infty}\right)&\\
H_1(N_{3-i},~W_i)&=&\overline{L}_{3-i}(s_i\cdot\lambda)\otimes_E\left(i_{B\cap L_{3-i}}^{L_{3-i}}(\chi_{s_{3-i}}^{\infty})\oplus \mathfrak{d}_{P_{3-i}}^{\infty}\right)&
\end{array}.
\end{equation}
We notice that $Z(L_{3-i}(\mathbb{Q}_p))$ acts on $\overline{L}_{3-i}(\lambda)$ and $\overline{L}_{3-i}(s_i\cdot\lambda)$ via different characters and that
$$\mathrm{Hom}_{L_{3-i}(\mathbb{Q}_p),\lambda}\left(i_{B\cap L_{3-i}}^{L_{3-i}}(\chi_{s_{3-i}}^{\infty}),~\overline{L}_{3-i}(s_i\cdot\lambda)\otimes_E\pi_{3-i,1}^{\infty}\right)=0.$$
Therefore we deduce from (\ref{3equation vanishing 1.3}) the equalities
$$
\begin{xy}
(0,0)*+{\mathrm{Ext}^1_{L_{3-i}(\mathbb{Q}_p),\lambda}\left(H_0(N_{3-i},~W_i),~\overline{L}_{3-i}(s_i\cdot\lambda)\otimes_E\pi_{3-i,1}^{\infty}\right)}; (47,0)*+{=0};
(1,-6)*+{\mathrm{Hom}_{L_{3-i}(\mathbb{Q}_p),\lambda}\left(H_1(N_{3-i},~W_i),~\overline{L}_{3-i}(s_i\cdot\lambda)\otimes_E\pi_{3-i,1}^{\infty}\right)}; (47,-6)*+{=0};
\end{xy}
$$
which imply by Lemma~\ref{3lemm: cohomology devissage} that
$$\mathrm{Ext}^1_{\mathrm{GL}_3(\mathbb{Q}_p),\lambda}\left(
W_i,~\mathcal{F}_{P_{3-i}}^{\mathrm{GL}_3}(M_{3-i}(-s_i\cdot\lambda),~\pi_{3-i,1}^{\infty})\right)=0.$$
Hence we finish the proof of (\ref{3equation vanishing 1.1}) by the embedding (\ref{3equation vanishing 1.2}).
\end{proof}

\begin{lemm}\label{3lemm: special vanishing 2}
We have for $i=1,2$:
$$
\begin{xy}
(0,0)*+{\mathrm{Ext}^1_{\mathrm{GL}_3(\mathbb{Q}_p),\lambda}\left(\begin{xy}
(0,0)*+{\overline{L}(\lambda)\otimes_Ev_{P_i}^{\infty}}="a"; (22,0)*+{C_{s_i,s_i}}="b";
{\ar@{-}"a";"b"};
\end{xy},~C^2_{s_i,1}\right)}; (50,0)*+{=0};
(5.25,-6)*+{\mathrm{Ext}^1_{\mathrm{GL}_3(\mathbb{Q}_p),\lambda}\left(\begin{xy}
(0,0)*+{\overline{L}(\lambda)\otimes_Ev_{P_{3-i}}^{\infty}}="a"; (27,0)*+{C^2_{s_i,s_is_{3-i}}}="b";
{\ar@{-}"a";"b"};
\end{xy},~C_{s_i,s_i}\right)}; (50,-6)*+{=0};
(-2.4,-12)*+{\mathrm{Ext}^1_{\mathrm{GL}_3(\mathbb{Q}_p),\lambda}\left(\begin{xy}
(20,0)*+{C^1_{s_i,s_is_{3-i}}}="b"; (0,0)*+{\overline{L}(\lambda)}="a";
{\ar@{-}"a";"b"};
\end{xy},~C^1_{s_i,1}\right)}; (50,-12)*+{=0};
(0,-18)*+{\mathrm{Ext}^1_{\mathrm{GL}_3(\mathbb{Q}_p),\lambda}\left(\begin{xy}
(7,0)*+{\overline{L}(\lambda)\otimes_E\mathrm{St}_3^{\infty}}="a"; (22,0)*+{C^2_{s_i,1}}="b";
{\ar@{-}"a";"b"};
\end{xy},~C^2_{s_i,s_is_{3-i}}\right)}; (50,-18)*+{=0};
\end{xy}
$$
\end{lemm}
\begin{proof}
We only prove that
\begin{equation}\label{3equation vanishing 2.1}
\mathrm{Ext}^1_{\mathrm{GL}_3(\mathbb{Q}_p),\lambda}\left(\begin{xy}
(0,0)*+{\overline{L}(\lambda)\otimes_Ev_{P_i}^{\infty}}="a"; (22,0)*+{C_{s_i,s_i}}="b";
{\ar@{-}"a";"b"};
\end{xy},~C^2_{s_i,1}\right)=0
\end{equation}
as the other cases are similar. The surjection
$$\mathcal{F}_{P_{3-i}}^{\mathrm{GL}_3}(M_{3-i}(-\lambda),~\pi_{3-i,2}^{\infty})\twoheadrightarrow \begin{xy}
(0,0)*+{\overline{L}(\lambda)\otimes_Ev_{P_i}^{\infty}}="a"; (22,0)*+{C_{s_i,s_i}}="b";
{\ar@{-}"a";"b"};
\end{xy}$$
and the embedding
$$C^2_{s_i,1}\hookrightarrow \mathcal{F}_{P_{3-i}}^{\mathrm{GL}_3}(M_{3-i}(-s_i\cdot\lambda),~\pi_{3-i,1}^{\infty})$$
induce an embedding
\begin{multline}\label{3equation vanishing 2.2}
\mathrm{Ext}^1_{\mathrm{GL}_3(\mathbb{Q}_p),\lambda}\left(\begin{xy}
(0,0)*+{\overline{L}(\lambda)\otimes_Ev_{P_i}^{\infty}}="a"; (22,0)*+{C_{s_i,s_i}}="b";
{\ar@{-}"a";"b"};
\end{xy},~C^2_{s_i,1}\right)\\
\hookrightarrow \mathrm{Ext}^1_{\mathrm{GL}_3(\mathbb{Q}_p),\lambda}\left(\mathcal{F}_{P_{3-i}}^{\mathrm{GL}_3}(M_{3-i}(-\lambda),~\pi_{3-i,2}^{\infty}),~\mathcal{F}_{P_{3-i}}^{\mathrm{GL}_3}(M_{3-i}(-s_i\cdot\lambda),~\pi_{3-i,1}^{\infty})\right).
\end{multline}
It follows from Proposition~\ref{3prop: list of N coh} that
$$
H_0(N_{3-i},~\mathcal{F}_{P_{3-i}}^{\mathrm{GL}_3}(M_{3-i}(-\lambda),~\pi_{3-i,2}^{\infty}))=\left(\overline{L}_{3-i}(\lambda)\oplus\overline{L}_{3-i}(s_i\cdot\lambda)\right)\otimes_E\pi_{3-i,2}^{\infty}
$$
and
\begin{multline*}
H_1(N_{3-i},~\mathcal{F}_{P_{3-i}}^{\mathrm{GL}_3}(M_{3-i}(-\lambda),~\pi_{3-i,2}^{\infty}))\\
=\left(\overline{L}_{3-i}(s_i\cdot\lambda)\oplus\overline{L}_{3-i}(s_is_{3-i}\cdot\lambda)\right)\otimes_E\pi_{3-i,2}^{\infty}\oplus I_{B\cap L_{3-i}}^{L_{3-i}}\left(\delta_{s_i\cdot\lambda}\right)\oplus I_{B\cap L_{3-i}}^{L_{3-i}}\left(\delta_{s_i\cdot\lambda}\otimes_E\chi_{s_1s_2s_1}^{\infty}\right).
\end{multline*}
We notice that $Z(L_{3-i}(\mathbb{Q}_p))$ acts on each direct summand of $H_k(N_{3-i},~\mathcal{F}_{P_{3-i}}^{\mathrm{GL}_3}(M_{3-i}(-\lambda),~\pi_{3-i,2}^{\infty}))$ ($k=0,1$) via a different character, and the only direct summand that produces the same character as $\overline{L}_{3-i}(s_i\cdot\lambda)\otimes\pi_{3-i,1}^{\infty}$ is $I_{B\cap L_{3-i}}^{L_{3-i}}\left(\delta_{s_i\cdot\lambda}\right)$. However, we know that
$$\mathrm{cosoc}_{L_{3-i}(\mathbb{Q}_p),\lambda}\left(I_{B\cap L_{3-i}}^{L_{3-i}}\left(\delta_{s_i\cdot\lambda}\right)\right)=I_{B\cap L_{3-i}}^{L_{3-i}}\left(\delta_{s_{3-i}s_i\cdot\lambda}\right)$$
and thus
\begin{equation*}
\mathrm{Hom}_{L_{3-i}(\mathbb{Q}_p),\lambda}\left(I_{B\cap L_{3-i}}^{L_{3-i}}\left(\delta_{s_{3-i}s_i\cdot\lambda}\right),~\overline{L}_{3-i}(s_i\cdot\lambda)\otimes\pi_{3-i,1}^{\infty}\right)=0.
\end{equation*}
As a result, we deduce the equalities
$$
\begin{xy}
(0,0)*+{\mathrm{Ext}^1_{L_{3-i}(\mathbb{Q}_p),\lambda}\left(H_0(N_{3-i},~\mathcal{F}_{P_{3-i}}^{\mathrm{GL}_3}(M_{3-i}(-\lambda),~\pi_{3-i,2}^{\infty})),~\overline{L}_{3-i}(s_i\cdot\lambda)\otimes_E\pi_{3-i,1}^{\infty}\right)}; (66,0)*+{=0};
(0.85,-6)*+{\mathrm{Hom}_{L_{3-i}(\mathbb{Q}_p),\lambda}\left(H_1(N_{3-i},~\mathcal{F}_{P_{3-i}}^{\mathrm{GL}_3}(M_{3-i}(-\lambda),~\pi_{3-i,2}^{\infty})),~\overline{L}_{3-i}(s_i\cdot\lambda)\otimes_E\pi_{3-i,1}^{\infty}\right)}; (66,-6)*+{=0};
\end{xy}
$$
which imply by Lemma~\ref{3lemm: cohomology devissage} that
$$\mathrm{Ext}^1_{\mathrm{GL}_3(\mathbb{Q}_p),\lambda}\left(
\mathcal{F}_{P_{3-i}}^{\mathrm{GL}_3}(M_{3-i}(-\lambda),~\pi_{3-i,2}^{\infty}),~\mathcal{F}_{P_{3-i}}^{\mathrm{GL}_3}(M_{3-i}(-s_i\cdot\lambda),~\pi_{3-i,1}^{\infty})\right)=0.$$
Hence we finish the proof of (\ref{3equation vanishing 2.1}) by the embedding (\ref{3equation vanishing 2.2}).
\end{proof}

\begin{lemm}\label{3lemm: existence of diamond}
There exists a unique representation of the form
$$\begin{xy}
(0,0)*+{C^2_{s_i,1}}="a"; (25,6)*+{C^1_{s_{3-i}s_i,1}}="b"; (25,-6)*+{\overline{L}(\lambda)\otimes_Ev_{P_i}^{\infty}}="c"; (50,0)*+{C_{s_i,s_i}}="d";
{\ar@{-}"a";"b"}; {\ar@{-}"a";"c"}; {\ar@{-}"b";"d"}; {\ar@{-}"c";"d"};
\end{xy}$$
or of the form
$$\begin{xy}
(0,0)*+{C_{s_i,s_i}}="a"; (27,6)*+{C^1_{s_{3-i}s_i,s_{3-i}s_i}}="b"; (27,-6)*+{\overline{L}(\lambda)\otimes_Ev_{P_{3-i}}^{\infty}}="c"; (54,0)*+{C^2_{s_i,s_is_{3-i}}}="d";
{\ar@{-}"a";"b"}; {\ar@{-}"a";"c"}; {\ar@{-}"b";"d"}; {\ar@{-}"c";"d"};
\end{xy}.$$
\end{lemm}
\begin{proof}
We only prove the first statement as the second is similar. It follows from Proposition~4.4.2 of \cite{Bre17} that there exists a unique representation of the form
$$\begin{xy}
(0,0)*+{C^2_{s_i,1}}="a"; (25,6)*+{C^1_{s_{3-i}s_i,1}}="b"; (25,-6)*+{\overline{L}(\lambda)\otimes_Ev_{P_i}^{\infty}}="c"; (50,0)*+{C_{s_i,s_i}}="d";
{\ar@{-}"a";"b"}; {\ar@{-}"a";"c"}; {\ar@{--}"b";"d"}; {\ar@{-}"c";"d"};
\end{xy}$$
but it is not proven there whether its quotient
\begin{equation}\label{3equation split in diamond}
\begin{xy}
(0,0)*+{C^1_{s_{3-i}s_i,1}}="b"; (20,0)*+{C_{s_i,s_i}}="d";
{\ar@{--}"b";"d"};
\end{xy}
\end{equation}
is split or not. However, If (\ref{3equation split in diamond}) is split, then there exists a representation of the form
$$
\begin{xy}
(0,0)*+{C^2_{s_i,1}}="a"; (20,-0)*+{\overline{L}(\lambda)\otimes_Ev_{P_i}^{\infty}}="c"; (40,0)*+{C_{s_i,s_i}}="d";
{\ar@{-}"a";"c"}; {\ar@{-}"c";"d"};
\end{xy}
$$
which contradicts the first vanishing in Lemma~\ref{3lemm: special vanishing 2}, and thus we finish the proof.
\end{proof}
\begin{rema}\label{3rema: existence of diamond}
Our method used in Lemma~\ref{3lemm: special vanishing 2} and in Lemma~\ref{3lemm: existence of diamond} is different from the one due to Y.Ding mentioned in part (ii) of Remark~4.4.3 of \cite{Bre17}. It is not difficult to observe that
\begin{equation}\label{3equation rema 1}
\mathrm{dim}_E\mathrm{Ext}^1_{\mathrm{GL}_3(\mathbb{Q}_p),\lambda}\left(C_{s_i,s_i},~\begin{xy}
(20,0)*+{C^2_{s_i,1}}="b"; (40,6)*+{C^1_{s_{3-i}s_i,1}}="c"; (40,-6)*+{\overline{L}(\lambda)\otimes_Ev_{P_i}^{\infty}}="d";
{\ar@{-}"b";"c"}; {\ar@{-}"b";"d"};
\end{xy}\right)=1
\end{equation}
and
\begin{equation}\label{3equation rema 2}
\mathrm{dim}_E\mathrm{Ext}^1_{\mathrm{GL}_3(\mathbb{Q}_p),\lambda}\left(C^2_{s_i,s_is_{3-i}},~\begin{xy}
(20,0)*+{C_{s_i,s_i}}="b"; (40,6)*+{C^1_{s_{3-i}s_i,s_{3-i}s_i}}="c"; (40,-6)*+{\overline{L}(\lambda)\otimes_Ev_{P_{3-i}}^{\infty}}="d";
{\ar@{-}"b";"c"}; {\ar@{-}"b";"d"};
\end{xy}\right)=1
\end{equation}
for $i=1,2$. Similar methods as those used in Proposition~4.4.2 of \cite{Bre17}, in Lemma~\ref{3lemm: special vanishing 2} and in Lemma~\ref{3lemm: existence of diamond} also imply the existence of a unique representation of the form
$$\begin{xy}
(0,0)*+{C^1_{s_i,1}}="a"; (25,6)*+{C_{s_{3-i}s_i,s_{3-i}}}="b"; (25,-6)*+{\overline{L}(\lambda)}="c"; (50,0)*+{C^1_{s_i,s_is_{3-i}}}="d";
{\ar@{-}"a";"b"}; {\ar@{-}"a";"c"}; {\ar@{-}"b";"d"}; {\ar@{-}"c";"d"};
\end{xy}$$
or of the form
$$\begin{xy}
(0,0)*+{C^2_{s_i,s_is_{3-i}}}="a"; (27,6)*+{C_{s_{3-i}s_i,s_{3-i}}}="b"; (27,-6)*+{\overline{L}(\lambda)\otimes_E\mathrm{St}_3^{\infty}}="c"; (54,0)*+{C^2_{s_i,1}}="d";
{\ar@{-}"a";"b"}; {\ar@{-}"a";"c"}; {\ar@{-}"b";"d"}; {\ar@{-}"c";"d"};
\end{xy}.$$
\end{rema}
\section{Computations of $\mathrm{Ext}$ II}\label{3section: technial min}
In this section, we are going to establish several computational results (most notably Lemma~\ref{3lemm: ext6}) which have crucial applications in Section~\ref{3section: local-global}.
\begin{lemm}\label{3lemm: ext3}
We have
$$\mathrm{dim}_E\mathrm{Ext}^1_{\mathrm{GL}_3(\mathbb{Q}_p),\lambda}\left(C_{s_i,s_i},~\Sigma_i(\lambda, \mathscr{L}_i)\right)=1$$
for $i=1,2$.
\end{lemm}
\begin{proof}
We only prove that
\begin{equation}\label{3equation 3.1}
\mathrm{dim}_E\mathrm{Ext}^1_{\mathrm{GL}_3(\mathbb{Q}_p),\lambda}\left(C_{s_1,s_1},~\Sigma_1(\lambda, \mathscr{L}_1)\right)=1
\end{equation}
as the other equality is similar. We note that $\Sigma_1(\lambda, \mathscr{L}_1)$ admits a subrepresentation of the form
$$
W:=\begin{xy}
(0,0)*+{\overline{L}(\lambda)\otimes_E\mathrm{St}_3^{\infty}}="a"; (20,0)*+{C^2_{s_1,1}}="b"; (40,6)*+{C^1_{s_2s_1,1}}="c"; (40,-6)*+{\overline{L}(\lambda)\otimes_Ev_{P_1}^{\infty}}="d";
{\ar@{-}"a";"b"}; {\ar@{-}"b";"c"}; {\ar@{-}"b";"d"};
\end{xy}
$$
due to Lemma~3.34, Lemma~3.37 and Remark~3.38 of \cite{BD18}. Therefore $\Sigma_1(\lambda, \mathscr{L}_1))$ admits a filtration such that $W$ appears as one term of the filtration and the only reducible graded piece is
$$V_1:=\begin{xy}
(20,0)*+{C^2_{s_1,1}}="b"; (40,6)*+{C^1_{s_2s_1,1}}="c"; (40,-6)*+{\overline{L}(\lambda)\otimes_Ev_{P_1}^{\infty}}="d";
{\ar@{-}"b";"c"}; {\ar@{-}"b";"d"};
\end{xy}.$$
It follows from Lemma~4.4.1 and Proposition~4.2.1 of \cite{Bre17} as well as our Lemma~\ref{3lemm: nonvanishing ext3} that
\begin{equation}\label{3equation 3.2}
\mathrm{Ext}^1_{\mathrm{GL}_3(\mathbb{Q}_p),\lambda}\left(C_{s_1,s_1},~V\right)=0
\end{equation}
for all graded pieces $V$ such that $V\neq V_1$. On the other hand, we have
\begin{equation}\label{3equation 3.4}
\mathrm{dim}_E\mathrm{Ext}^1_{\mathrm{GL}_3(\mathbb{Q}_p),\lambda}\left(C_{s_1,s_1},~V_1\right)=1
\end{equation}
due to (\ref{3equation rema 1}) and
\begin{equation}\label{3equation 3.3}
\mathrm{Ext}^2_{\mathrm{GL}_3(\mathbb{Q}_p),\lambda}\left(C_{s_1,s_1},~\overline{L}(\lambda)\otimes_E\mathrm{St}_3^{\infty}\right)=0
\end{equation}
by Proposition 4.6.1 of \cite{Bre17}. Hence we finish the proof by combining (\ref{3equation 3.2}), (\ref{3equation 3.4}), (\ref{3equation 3.3}) and part (ii) of Proposition~\ref{3prop: formal devissages}.
\end{proof}

\begin{lemm}\label{3lemm: ext4}
We have
$$\mathrm{dim}_E\mathrm{Ext}^1_{\mathrm{GL}_3(\mathbb{Q}_p),\lambda}\left(\overline{L}(\lambda)\otimes_Ev_{P_{3-i}}^{\infty}, ~\Sigma^+_i(\lambda,\mathscr{L}_i)\right)=3$$
for $i=1,2$.
\end{lemm}
\begin{proof}
By symmetry, it suffices to prove that
$$\mathrm{dim}_E\mathrm{Ext}^1_{\mathrm{GL}_3(\mathbb{Q}_p),\lambda}\left(\overline{L}(\lambda)\otimes_Ev_{P_2}^{\infty}, ~\Sigma^+_1(\lambda,\mathscr{L}_1\right)=3.$$
This follows immediately from Lemma~3.42 of \cite{Bre17} as our $\Sigma^+_1(\lambda,\mathscr{L}_1)$ can be identified with the locally analytic representation $\widetilde{\Pi}^1(\lambda, \psi)$ defined before (3.76) of \cite{Bre17} up to changes on notation.
\end{proof}

We define $\Sigma^+_1(\lambda, \mathscr{L}_1)$ (resp. $\Sigma^+_2(\lambda, \mathscr{L}_2)$) as the unique non-split extension given by a non-zero element in $\mathrm{Ext}^1_{\mathrm{GL}_3(\mathbb{Q}_p),\lambda}\left(C_{s_1,s_1}, ~\Sigma_1(\lambda, \mathscr{L}_1)\right)$ (resp. in $\mathrm{Ext}^1_{\mathrm{GL}_3(\mathbb{Q}_p),\lambda}\left(C_{s_2,s_2}, ~\Sigma_2(\lambda, \mathscr{L}_2)\right)$). Hence we may consider the amalgamate sum of $\Sigma^+_1(\lambda, \mathscr{L}_1)$ and $\Sigma^+_2(\lambda, \mathscr{L}_2)$ over $\mathrm{St}_3^{\rm{an}}(\lambda)$ and denote it by $\Sigma^+(\lambda, \mathscr{L}_1, \mathscr{L}_2)$. In particular, $\Sigma^+(\lambda, \mathscr{L}_1, \mathscr{L}_2)$ has the following form
$$\begin{xy}
(0,0)*+{\mathrm{St}_3^{\rm{an}}(\lambda)}="a"; (20,4)*+{v_{P_1}^{\rm{an}}(\lambda)}="b"; (20,-4)*+{v_{P_2}^{\rm{an}}(\lambda)}="c"; (40,4)*+{C_{s_1,s_1}}="d"; (40,-4)*+{C_{s_2,s_2}}="e";
{\ar@{-}"a";"b"}; {\ar@{-}"a";"c"}; {\ar@{-}"b";"d"}; {\ar@{-}"c";"e"};
\end{xy}.$$
\begin{lemm}\label{3lemm: ext5}
We have
$$\mathrm{dim}_E\mathrm{Ext}^1_{\mathrm{GL}_3(\mathbb{Q}_p),\lambda}\left(\overline{L}(\lambda)\otimes_Ev_{P_i}^{\infty}, ~\Sigma^+(\lambda, \mathscr{L}_1, \mathscr{L}_2)\right)=2$$
for $i=1,2$.
\end{lemm}
\begin{proof}
The short exact sequence
$$\Sigma^+_2(\lambda, \mathscr{L}_2)\hookrightarrow\Sigma^+(\lambda, \mathscr{L}_1, \mathscr{L}_2)\twoheadrightarrow \begin{xy}
(0,0)*+{v_{P_1}^{\rm{an}}(\lambda)}="a"; (18,0)*+{C_{s_1,s_1}}="b";
{\ar@{-}"a";"b"};
\end{xy}$$
induces the following long exact sequence
\begin{multline*}
\mathrm{Hom}_{\mathrm{GL}_3(\mathbb{Q}_p),\lambda}\left(\overline{L}(\lambda)\otimes_Ev_{P_1}^{\infty}, ~\begin{xy}
(0,0)*+{v_{P_1}^{\rm{an}}(\lambda)}="a"; (18,0)*+{C_{s_1,s_1}}="b";
{\ar@{-}"a";"b"};
\end{xy}\right)\\
\hookrightarrow \mathrm{Ext}^1_{\mathrm{GL}_3(\mathbb{Q}_p),\lambda}\left(\overline{L}(\lambda)\otimes_Ev_{P_1}^{\infty}, ~\Sigma^+_2(\lambda, \mathscr{L}_2)\right)\\
\rightarrow \mathrm{Ext}^1_{\mathrm{GL}_3(\mathbb{Q}_p),\lambda}\left(\overline{L}(\lambda)\otimes_Ev_{P_1}^{\infty}, ~\Sigma^+(\lambda, \mathscr{L}_1, \mathscr{L}_2)\right)\\
\rightarrow \mathrm{Ext}^1_{\mathrm{GL}_3(\mathbb{Q}_p),\lambda}\left(\overline{L}(\lambda)\otimes_Ev_{P_1}^{\infty}, ~\begin{xy}
(0,0)*+{v_{P_1}^{\rm{an}}(\lambda)}="a"; (18,0)*+{C_{s_1,s_1}}="b";
{\ar@{-}"a";"b"};
\end{xy}\right).
\end{multline*}
As a result, we can deduce
$$\mathrm{dim}_E\mathrm{Ext}^1_{\mathrm{GL}_3(\mathbb{Q}_p),\lambda}\left(\overline{L}(\lambda)\otimes_Ev_{P_1}^{\infty}, ~\Sigma^+(\lambda, \mathscr{L}_1, \mathscr{L}_2)\right)=2$$
from Lemma~\ref{3lemm: ext4} and the facts
$$\mathrm{dim}_E\mathrm{Hom}_{\mathrm{GL}_3(\mathbb{Q}_p),\lambda}\left(\overline{L}(\lambda)\otimes_Ev_{P_1}^{\infty}, ~\begin{xy}
(0,0)*+{v_{P_1}^{\rm{an}}(\lambda)}="a"; (18,0)*+{C_{s_1,s_1}}="b";
{\ar@{-}"a";"b"};
\end{xy}\right)=1$$
and
$$\mathrm{Ext}^1_{\mathrm{GL}_3(\mathbb{Q}_p),\lambda}\left(\overline{L}(\lambda)\otimes_Ev_{P_1}^{\infty}, ~\begin{xy}
(0,0)*+{v_{P_1}^{\rm{an}}(\lambda)}="a"; (18,0)*+{C_{s_1,s_1}}="b";
{\ar@{-}"a";"b"};
\end{xy}\right)=0$$
by Proposition~\ref{3prop: locally algebraic extension} and Lemma~\ref{3lemm: nonvanishing ext1}. The proof for
$$\mathrm{dim}_E\mathrm{Ext}^1_{\mathrm{GL}_3(\mathbb{Q}_p),\lambda}\left(\overline{L}(\lambda)\otimes_Ev_{P_2}^{\infty}, ~\Sigma^+(\lambda, \mathscr{L}_1, \mathscr{L}_2)\right)=2$$
is similar.
\end{proof}

\begin{lemm}\label{3lemm: ext6}
We have
$$\mathrm{Ext}^1_{\mathrm{GL}_3(\mathbb{Q}_p),\lambda}(W_i, ~\Sigma^+_i(\lambda, \mathscr{L}_i))=0$$
and in particular
$$\mathrm{Ext}^1_{\mathrm{GL}_3(\mathbb{Q}_p),\lambda}\left(W_i, ~\Sigma_i(\lambda, \mathscr{L}_i)\right)=0$$
for $i=1,2$.
\end{lemm}
\begin{proof}
We only need to show the vanishing
$$\mathrm{Ext}^1_{\mathrm{GL}_3(\mathbb{Q}_p),\lambda}(W_2,~\Sigma^+_1(\lambda, \mathscr{L}_1))=0$$
as the others are similar or easier. We define $\nu:=\lambda_{T_2, \iota_{T,1}}$ (which is the restriction of $\lambda$ from $T$ to $T_2$ via the embedding $\iota_{T,1}: T_2\hookrightarrow T$) and view $\Sigma^+_{\mathrm{GL}_2}(\nu, \mathscr{L}_1)$ (which is defined before Proposition~\ref{3prop: key result}) as a locally analytic representation of $L_1(\mathbb{Q}_p)$ via the projection $p_1: L_1\twoheadrightarrow\mathrm{GL}_2$ and denote it by $\Sigma^+_{L_1}(\lambda, \mathscr{L}_1)$. We note by definition by of $\Sigma_1(\lambda, \mathscr{L}_1)$ that we have an isomorphism
$$\Sigma_1(\lambda, \mathscr{L}_1)\xrightarrow{\sim} I_{P_1}^{\mathrm{GL}_3}\left(\Sigma_{L_1}(\lambda, \mathscr{L}_1)\right)/\left(
\begin{xy}
(0,0)*+{v_{P_2}^{\rm{an}}(\lambda)}="a"; (16,0)*+{\overline{L}(\lambda)}="b";
{\ar@{-}"a";"b"};
\end{xy}\right).$$
Therefore we can deduce from the short exact sequence
$$\Sigma^+_{\mathrm{GL}_2}(\nu, \mathscr{L}_1)\hookrightarrow \Sigma^+_{\mathrm{GL}_2}(\nu, \mathscr{L}_1)\twoheadrightarrow \widetilde{I}(s\cdot\nu)$$
and the fact (up to viewing $\widetilde{I}(s\cdot\nu)$ as a locally analytic representation of $L_1(\mathbb{Q}_p)$ via the projection $p_1$)
$$C_{s_1,s_1}\cong\mathrm{soc}_{\mathrm{GL}_3(\mathbb{Q}_p)}\left(I_{P_1}^{\mathrm{GL}_3}\left(\widetilde{I}(s\cdot\nu)\right)\right)$$
that we have an injection
$$\Sigma^+_1(\lambda, \mathscr{L}_1)\hookrightarrow I_{P_1}^{\mathrm{GL}_3}\left(\Sigma^+_{L_1}(\lambda, \mathscr{L}_1)\right)/\left(
\begin{xy}
(0,0)*+{v_{P_2}^{\rm{an}}(\lambda)}="a"; (16,0)*+{\overline{L}(\lambda)}="b";
{\ar@{-}"a";"b"};
\end{xy}\right)$$
which induces an injection
\begin{equation}\label{3reduce via injection}
\mathrm{Ext}^1_{\mathrm{GL}_3(\mathbb{Q}_p),\lambda}\left(W_2, ~\Sigma^+_1(\lambda, \mathscr{L}_1)\right)\hookrightarrow \mathrm{Ext}^1_{\mathrm{GL}_3(\mathbb{Q}_p),\lambda}\left(W_2, ~V\right)
\end{equation}
where we use the shorten notation
$$V:=I_{P_1}^{\mathrm{GL}_3}\left(\Sigma^+_{L_1}(\lambda, \mathscr{L}_1)\right)/\left(\begin{xy}
(0,0)*+{v_{P_2}^{\rm{an}}(\lambda)}="a"; (18,0)*+{\overline{L}(\lambda)}="b";
{\ar@{-}"a";"b"};
\end{xy}\right).$$
Note that we have an exact sequence
\begin{multline}\label{3key long exact sequence for vanishing}
\mathrm{Ext}^1_{\mathrm{GL}_3(\mathbb{Q}_p),\lambda}\left(W_2, ~I_{P_1}^{\mathrm{GL}_3}\left(\Sigma^+_{L_1}(\lambda, \mathscr{L}_1)\right)\right)\\ \rightarrow \mathrm{Ext}^1_{\mathrm{GL}_3(\mathbb{Q}_p),\lambda}\left(W_2, ~V\right) \rightarrow\mathrm{Ext}^2_{\mathrm{GL}_3(\mathbb{Q}_p),\lambda}\left(W_2,~\begin{xy}
(0,0)*+{v_{P_2}^{\rm{an}}(\lambda)}="a"; (18,0)*+{\overline{L}(\lambda)}="b";
{\ar@{-}"a";"b"};
\end{xy}\right)
\end{multline}
It follows from Proposition~\ref{3prop: list of N coh} that
$$
\begin{array}{cccc}
H_0(N_1, ~W_2)&=&\overline{L}_1(\lambda)\otimes_Ei_{B\cap L_1}^{L_1}(\chi_{s_1}^{\infty})&\\
H_1(N_1, ~W_2)&=&\overline{L}_1(s_2\cdot \lambda)\otimes_E\otimes_Ei_{B\cap L_1}^{L_1}(\chi_{s_1}^{\infty})&.
\end{array}
$$
Therefore we observe that
$$\mathrm{Hom}_{L_1(\mathbb{Q}_p),\lambda}\left(H_1(N_1, ~W_2),~\Sigma^+_{L_1}(\lambda, \mathscr{L}_1)\right)=0$$
from the action of $Z(L_1(\mathbb{Q}_p))$ and
$$\mathrm{Ext}^1_{L_1(\mathbb{Q}_p),\lambda}\left(H_0(N_1, ~W_2),~\Sigma^+_{L_1}(\lambda, \mathscr{L}_1)\right)=0$$
according to Proposition~\ref{3prop: key result} and the natural identification
$$\mathrm{Ext}^1_{L_1(\mathbb{Q}_p),\lambda}(-,-)\cong \mathrm{Ext}^1_{\mathrm{GL}_2(\mathbb{Q}_p)}(-,-).$$
As a result, we deduce
\begin{equation}\label{3intermediate vanishing}
\mathrm{Ext}^1_{\mathrm{GL}_3(\mathbb{Q}_p),\lambda}\left(W_2, ~I_{P_1}^{\mathrm{GL}_3}\left(\Sigma^+_{L_1}(\lambda, \mathscr{L}_1)\right)\right)=0
\end{equation}
from Lemma~\ref{3lemm: cohomology devissage}. We know that
\begin{equation}\label{3a vanishing ext2}
\mathrm{Ext}^2_{\mathrm{GL}_3(\mathbb{Q}_p),\lambda}\left(W_2,~\begin{xy}
(0,0)*+{v_{P_2}^{\rm{an}}(\lambda)}="a"; (18,0)*+{\overline{L}(\lambda)}="b";
{\ar@{-}"a";"b"};
\end{xy}\right)=0
\end{equation}
due to Proposition~\ref{3prop: locally algebraic extension}, Lemma~\ref{3lemm: nonvanishing ext2} and a simple devissage, and thus we finish the proof by (\ref{3reduce via injection}), (\ref{3key long exact sequence for vanishing}), (\ref{3intermediate vanishing}) and (\ref{3a vanishing ext2}).
\end{proof}

\begin{lemm}\label{3lemm: ext6.1}
We have
\begin{equation}\label{3first key dim}
\mathrm{dim}_E\mathrm{Ext}^2_{\mathrm{GL}_3(\mathbb{Q}_p),\lambda}\left(\overline{L}(\lambda), ~\Sigma^+_i(\lambda, \mathscr{L}_i)\right)=3
\end{equation}
for each $i=1,2$,
\begin{equation}\label{3second key dim}
\mathrm{dim}_E\mathrm{Ext}^2_{\mathrm{GL}_3(\mathbb{Q}_p),\lambda}\left(\overline{L}(\lambda), ~\Sigma^+(\lambda, \mathscr{L}_1, \mathscr{L}_2)\right)=2
\end{equation}
and
\begin{equation}\label{3third key dim}
\mathrm{dim}_E\mathrm{Ext}^1_{\mathrm{GL}_3(\mathbb{Q}_p),\lambda}\left(\overline{L}(\lambda), ~\Sigma^+(\lambda, \mathscr{L}_1, \mathscr{L}_2)\right)=1.
\end{equation}
\end{lemm}
\begin{proof}
The equalities (\ref{3second key dim}) and (\ref{3third key dim}) follow directly from Lemma~\ref{3lemm: ext2} and the fact that
\begin{equation}\label{3vanishing of simple special ext}
\mathrm{Ext}^1_{\mathrm{GL}_3(\mathbb{Q}_p),\lambda}\left(\overline{L}(\lambda),~C_{s_i,s_i}\right)=\mathrm{Ext}^2_{\mathrm{GL}_3(\mathbb{Q}_p),\lambda}\left(\overline{L}(\lambda),~C_{s_i,s_i}\right)=0
\end{equation}
by Lemma~\ref{3lemm: nonvanishing ext1} and Lemma~\ref{3lemm: nonvanishing ext2} using a long exact sequence induced from the short exact sequence
$$\Sigma_i(\lambda, \mathscr{L}_i)\hookrightarrow\Sigma^+_i(\lambda, \mathscr{L}_i)\twoheadrightarrow C_{s_i,s_i}.$$
Due to a similar argument using (\ref{3vanishing of simple special ext}), we only need to show that
\begin{equation}\label{3equation ext 6.1.1}
\mathrm{dim}_E\mathrm{Ext}^2_{\mathrm{GL}_3(\mathbb{Q}_p),\lambda}\left(\overline{L}(\lambda), ~\Sigma_i(\lambda, \mathscr{L}_i)\right)=3
\end{equation}
to finish the proof of (\ref{3first key dim}). The short exact sequence
$$\mathrm{St}_3^{\rm{an}}(\lambda)\hookrightarrow\Sigma_i(\lambda, \mathscr{L}_i)\twoheadrightarrow v_{P_i}^{\rm{an}}(\lambda)$$ induces a long exact sequence
\begin{multline}\label{3equation ext 6.1.2}
\mathrm{Ext}^1_{\mathrm{GL}_3(\mathbb{Q}_p),\lambda}\left(\overline{L}(\lambda), ~\Sigma_i(\lambda, \mathscr{L}_i)\right)\rightarrow \mathrm{Ext}^1_{\mathrm{GL}_3(\mathbb{Q}_p),\lambda}\left(\overline{L}(\lambda), ~v_{P_i}^{\rm{an}}(\lambda)\right)\\
\rightarrow \mathrm{Ext}^2_{\mathrm{GL}_3(\mathbb{Q}_p),\lambda}\left(\overline{L}(\lambda), ~\mathrm{St}_3^{\rm{an}}(\lambda)\right)\rightarrow
\mathrm{Ext}^2_{\mathrm{GL}_3(\mathbb{Q}_p),\lambda}\left(\overline{L}(\lambda), ~\Sigma_i(\lambda, \mathscr{L}_i)\right)\rightarrow
\mathrm{Ext}^2_{\mathrm{GL}_3(\mathbb{Q}_p),\lambda}\left(\overline{L}(\lambda), ~v_{P_i}^{\rm{an}}(\lambda)\right).
\end{multline}
We know that
$$\mathrm{dim}_E\mathrm{Ext}^2_{\mathrm{GL}_3(\mathbb{Q}_p),\lambda}\left(\overline{L}(\lambda), ~\mathrm{St}_3^{\rm{an}}(\lambda)\right)=5$$
by Lemma~\ref{3lemm: ext2.1}. It follows from Proposition~\ref{3prop: locally algebraic extension}, Lemma~\ref{3lemm: nonvanishing ext1}, Lemma~\ref{3lemm: nonvanishing ext2} and a simple devissage that
\begin{equation}\label{3equation ext 6.1.2 prime}
\mathrm{dim}_E\mathrm{Ext}^1_{\mathrm{GL}_3(\mathbb{Q}_p),\lambda}\left(\overline{L}(\lambda), ~v_{P_i}^{\rm{an}}(\lambda)\right)=2
\end{equation}
and
\begin{equation}\label{3equation ext 6.1.2 primeprime}
\mathrm{Ext}^2_{\mathrm{GL}_3(\mathbb{Q}_p),\lambda}\left(\overline{L}(\lambda), ~v_{P_i}^{\rm{an}}(\lambda)\right)=0.
\end{equation}
Hence it remains to show that
\begin{equation}\label{3equation ext 6.1.3}
\mathrm{Ext}^1_{\mathrm{GL}_3(\mathbb{Q}_p),\lambda}\left(\overline{L}(\lambda), ~\Sigma_i(\lambda, \mathscr{L}_i)\right)=0
\end{equation}
to deduce (\ref{3equation ext 6.1.1}) from (\ref{3equation ext 6.1.2}). The short exact sequence
$$\begin{xy}
(0,0)*+{v_{P_{3-i}}^{\rm{an}}(\lambda)}="a"; (16,0)*+{\overline{L}(\lambda)}="b";
{\ar@{-}"a";"b"};
\end{xy}\hookrightarrow I_{P_i}^{\mathrm{GL}_3}\left(\Sigma_{L_i}(\lambda, \mathscr{L}_i)\right)\twoheadrightarrow\Sigma_i(\lambda, \mathscr{L}_i)$$
induces
\begin{multline*}
\mathrm{Ext}^1_{\mathrm{GL}_3(\mathbb{Q}_p),\lambda}\left(\overline{L}(\lambda),~\begin{xy}
(0,0)*+{v_{P_{3-i}}^{\rm{an}}(\lambda)}="a"; (16,0)*+{\overline{L}(\lambda)}="b";
{\ar@{-}"a";"b"};
\end{xy}\right)\\
\hookrightarrow\mathrm{Ext}^1_{\mathrm{GL}_3(\mathbb{Q}_p),\lambda}\left(\overline{L}(\lambda),~I_{P_i}^{\mathrm{GL}_3}\left(\Sigma_{L_i}(\lambda, \mathscr{L}_i)\right)\right)\twoheadrightarrow \mathrm{Ext}^1_{\mathrm{GL}_3(\mathbb{Q}_p),\lambda}\left(\overline{L}(\lambda),~\Sigma_i(\lambda, \mathscr{L}_i)\right)
\end{multline*}
by the vanishing
$$\mathrm{Ext}^2_{\mathrm{GL}_3(\mathbb{Q}_p),\lambda}\left(\overline{L}(\lambda), \begin{xy}
(0,0)*+{v_{P_{3-i}}^{\rm{an}}(\lambda)}="a"; (16,0)*+{\overline{L}(\lambda)}="b";
{\ar@{-}"a";"b"};
\end{xy}\right)=0$$
using Proposition~\ref{3prop: locally algebraic extension} and Lemma~\ref{3lemm: nonvanishing ext2}. Therefore we only need to show that
\begin{equation}\label{3equation ext 6.1.4}
\mathrm{dim}_E\mathrm{Ext}^1_{\mathrm{GL}_3(\mathbb{Q}_p),\lambda}\left(\overline{L}(\lambda),~\begin{xy}
(0,0)*+{v_{P_{3-i}}^{\rm{an}}(\lambda)}="a"; (16,0)*+{\overline{L}(\lambda)}="b";
{\ar@{-}"a";"b"};
\end{xy}\right)=1
\end{equation}
and
\begin{equation}\label{3equation ext 6.1.5}
\mathrm{dim}_E\mathrm{Ext}^1_{\mathrm{GL}_3(\mathbb{Q}_p),\lambda}\left(\overline{L}(\lambda),~I_{P_i}^{\mathrm{GL}_3}\left(\Sigma_{L_i}(\lambda, \mathscr{L}_i)\right)\right)=1.
\end{equation}
The equality (\ref{3equation ext 6.1.5}) follows from Lemma~\ref{3lemm: cohomology devissage} and the facts
$$
\mathrm{dim}_E\mathrm{Ext}^1_{L_i(\mathbb{Q}_p),\lambda}\left(H_0(N_i,~\overline{L}(\lambda)),~\Sigma_{L_i}(\lambda, \mathscr{L}_i)\right)=1,~\mathrm{Hom}_{L_i(\mathbb{Q}_p),\lambda}\left(H_1(N_i,~\overline{L}(\lambda)),~\Sigma_{L_i}(\lambda, \mathscr{L}_i)\right)=0
$$
where the first equality essentially follows from Lemma~3.14 of \cite{BD18} and the second equality follows from checking the action of $Z(L_i(\mathbb{Q}_p))$. On the other hand, (\ref{3equation ext 6.1.4}) follows from (\ref{3equation ext 6.1.2 prime}) and Proposition~\ref{3prop: locally algebraic extension} by an easy devissage. Hence we finish the proof.
\end{proof}

\begin{prop}\label{3prop: isomorphism from cup product}
The short exact sequence
$$\overline{L}(\lambda)\otimes_Ev_{P_i}^{\infty}\hookrightarrow W_i\twoheadrightarrow \overline{L}(\lambda)$$
induces the following isomorphisms
\begin{equation}\label{3first key isomorphism}
\mathrm{Ext}^1_{\mathrm{GL}_3(\mathbb{Q}_p),\lambda}\left(\overline{L}(\lambda)\otimes_Ev_{P_{3-i}}^{\infty}, ~\Sigma^+_i(\lambda, \mathscr{L}_i)\right)\xrightarrow{\sim}\mathrm{Ext}^2_{\mathrm{GL}_3(\mathbb{Q}_p),\lambda}\left(\overline{L}(\lambda), ~\Sigma^+_i(\lambda, \mathscr{L}_i)\right)
\end{equation}
and
\begin{equation}\label{3second key isomorphism}
\mathrm{Ext}^1_{\mathrm{GL}_3(\mathbb{Q}_p),\lambda}\left(\overline{L}(\lambda)\otimes_Ev_{P_{3-i}}^{\infty}, ~\Sigma^+(\lambda, \mathscr{L}_1, \mathscr{L}_2)\right)\xrightarrow{\sim}\mathrm{Ext}^2_{\mathrm{GL}_3(\mathbb{Q}_p),\lambda}\left(\overline{L}(\lambda), ~\Sigma^+(\lambda, \mathscr{L}_1, \mathscr{L}_2)\right)
\end{equation}
for $i=1,2$.
\end{prop}
\begin{proof}
The vanishing from Lemma~\ref{3lemm: ext6} implies that
$$\mathrm{Ext}^1_{\mathrm{GL}_3(\mathbb{Q}_p),\lambda}\left(\overline{L}(\lambda)\otimes_Ev_{P_{3-i}}^{\infty}, ~\Sigma^+_i(\lambda, \mathscr{L}_i)\right)\rightarrow\mathrm{Ext}^2_{\mathrm{GL}_3(\mathbb{Q}_p),\lambda}\left(\overline{L}(\lambda), ~\Sigma^+_i(\lambda, \mathscr{L}_i)\right)$$
is an injection and hence an isomorphism as both spaces have dimension three according to Lemma~\ref{3lemm: ext4} and Lemma~\ref{3lemm: ext6.1}. The proof of (\ref{3second key isomorphism}) is similar. We emphasize that both (\ref{3first key isomorphism}) and (\ref{3second key isomorphism}) can be interpreted as the isomorphism given by the cup product with the one dimensional space $$\mathrm{Ext}^1_{\mathrm{GL}_3(\mathbb{Q}_p),\lambda}\left(\overline{L}(\lambda),~\overline{L}(\lambda)\otimes_Ev_{P_{3-i}}^{\infty}\right).$$
\end{proof}

We define
$$\Sigma^{\flat}(\lambda, \mathscr{L}_1, \mathscr{L}_2):=\Sigma(\lambda, \mathscr{L}_1, \mathscr{L}_2)/\overline{L}(\lambda)\otimes_E\mathrm{St}_3^{\infty}\mbox{ and }\Sigma^{\flat}_i(\lambda, \mathscr{L}_i):=\Sigma_i(\lambda, \mathscr{L}_i)/\overline{L}(\lambda)\otimes_E\mathrm{St}_3^{\infty}$$
for $i=1,2$.
\begin{lemm}\label{3lemm: ext7}
We have
$$\mathrm{dim}_E\mathrm{Ext}^1_{\mathrm{GL}_3(\mathbb{Q}_p),\lambda}\left(\overline{L}(\lambda), ~\Sigma^{\flat}(\lambda, \mathscr{L}_1, \mathscr{L}_2)\right)=1.$$
\end{lemm}
\begin{proof}
We define $\Sigma^{\flat,-}(\lambda, \mathscr{L}_1, \mathscr{L}_2)$ as the subrepresentation of $\Sigma^{\flat}(\lambda, \mathscr{L}_1, \mathscr{L}_2)$ that fits into the following short exact sequence
\begin{equation}\label{3equation 7.1}
\Sigma^{\flat,-}(\lambda, \mathscr{L}_1, \mathscr{L}_2)\hookrightarrow\Sigma^{\flat}(\lambda, \mathscr{L}_1, \mathscr{L}_2)\twoheadrightarrow C^1_{s_2,1}\oplus C^1_{s_1,1},
\end{equation}
( cf. (\ref{3irr rep I}) for the definition of $C^1_{s_2,1}$, $C^1_{s_1,1}$, $C^2_{s_2,1}$ and $C^2_{s_1,1}$ ) and then define $\Sigma^{\flat,--}(\lambda, \mathscr{L}_1, \mathscr{L}_2)$ as the subrepresentation of $\Sigma^{\flat,-}(\lambda, \mathscr{L}_1, \mathscr{L}_2)$ that fits into
\begin{equation}\label{3equation 7.1 prime}
\Sigma^{\flat,--}(\lambda, \mathscr{L}_1, \mathscr{L}_2)\hookrightarrow\Sigma^{\flat,-}(\lambda, \mathscr{L}_1, \mathscr{L}_2)\twoheadrightarrow \left(\begin{xy}
(0,0)*+{C^2_{s_1,1}}="a"; (20,0)*+{\overline{L}(\lambda)\otimes_Ev_{P_1}^{\infty}}="b";
{\ar@{-}"a";"b"};
\end{xy}\right)\oplus \left(\begin{xy}
(0,0)*+{C^2_{s_2,1}}="a"; (20,0)*+{\overline{L}(\lambda)\otimes_Ev_{P_2}^{\infty}}="b";
{\ar@{-}"a";"b"};
\end{xy}\right).
\end{equation}
It follows from Lemma~\ref{3lemm: nonvanishing ext1} that
$$\mathrm{Ext}^1_{\mathrm{GL}_3(\mathbb{Q}_p),\lambda}(\overline{L}(\lambda), ~V)=0$$
for each $V\in\mathrm{JH}_{\mathrm{GL}_3(\mathbb{Q}_p)}\left(\Sigma^{\flat,--}(\lambda, \mathscr{L}_1, \mathscr{L}_2)\right)$ and therefore
\begin{equation}\label{3equation 7.1.1}
\mathrm{Ext}^1_{\mathrm{GL}_3(\mathbb{Q}_p),\lambda}\left(\overline{L}(\lambda), ~\Sigma^{\flat,--}(\lambda, \mathscr{L}_1, \mathscr{L}_2)\right)=0
\end{equation}
by part (i) of Proposition~\ref{3prop: formal devissages}. On the other hand, we know from Lemma~\ref{3lemm: nonvanishing ext1} and Lemma~\ref{3lemm: special vanishing 1} that there is no uniserial representation of the form
$$\begin{xy}
(0,0)*+{C^2_{s_i,1}}="a"; (20,0)*+{\overline{L}(\lambda)\otimes_Ev_{P_2}^{\infty}}="b"; (40,0)*+{\overline{L}(\lambda)}="c";
{\ar@{-}"a";"b"}; {\ar@{-}"b";"c"};
\end{xy}$$
which implies that
\begin{equation}\label{3equation 7.1.2}
\mathrm{Ext}^1_{\mathrm{GL}_3(\mathbb{Q}_p),\lambda}\left(\overline{L}(\lambda), ~\begin{xy}
(0,0)*+{C^2_{s_i,1}}="a"; (20,0)*+{\overline{L}(\lambda)\otimes_Ev_{P_i}^{\infty}}="b";
{\ar@{-}"a";"b"};
\end{xy}\right)=0
\end{equation}
for $i=1,2$. Hence we deduce from (\ref{3equation 7.1 prime}), (\ref{3equation 7.1.1}), (\ref{3equation 7.1.2}) and Proposition~\ref{3prop: formal devissages} that
\begin{equation}
\mathrm{Ext}^1_{\mathrm{GL}_3(\mathbb{Q}_p),\lambda}\left(\overline{L}(\lambda), ~\Sigma^{\flat,-}(\lambda, \mathscr{L}_1, \mathscr{L}_2)\right)=0.
\end{equation}
Therefore (\ref{3equation 7.1}) induces an injection
\begin{equation}\label{3equation 7.2}
\mathrm{Ext}^1_{\mathrm{GL}_3(\mathbb{Q}_p),\lambda}\left(\overline{L}(\lambda), ~\Sigma^{\flat}(\lambda, \mathscr{L}_1, \mathscr{L}_2)\right)\hookrightarrow \mathrm{Ext}^1_{\mathrm{GL}_3(\mathbb{Q}_p),\lambda}\left(\overline{L}(\lambda), ~C^1_{s_2,1}\oplus C^1_{s_1,1}\right).
\end{equation}
Assume first that (\ref{3equation 7.2}) is a surjection, then we pick a representation $W$ represented by a non-zero element in $\mathrm{Ext}^1_{\mathrm{GL}_3(\mathbb{Q}_p),\lambda}\left(\overline{L}(\lambda), ~\Sigma^{\flat}(\lambda, \mathscr{L}_1, \mathscr{L}_2)\right)$ lying in the preimage of $\mathrm{Ext}^1_{\mathrm{GL}_3(\mathbb{Q}_p),\lambda}\left(\overline{L}(\lambda), ~C^1_{s_2,1}\right)$ under (\ref{3equation 7.2}). We note that there is a short exact sequence
$$\Sigma^{\flat}_1(\lambda, \mathscr{L}_1)\hookrightarrow\Sigma^{\flat}(\lambda, \mathscr{L}_1, \mathscr{L}_2)\twoheadrightarrow v_{P_2}^{\rm{an}}(\lambda).$$
We observe that $\overline{L}(\lambda)$ lies above neither $C^1_{s_1,1}$ nor $\overline{L}(\lambda)\otimes_Ev_{P_2}^{\infty}$ inside $W$ by our definition and (\ref{3equation 7.1.2}), and thus $W$ is mapped to zero under the map
$$f:\mathrm{Ext}^1_{\mathrm{GL}_3(\mathbb{Q}_p),\lambda}\left(\overline{L}(\lambda), ~\Sigma^{\flat}(\lambda, \mathscr{L}_1, \mathscr{L}_2)\right)\rightarrow\mathrm{Ext}^1_{\mathrm{GL}_3(\mathbb{Q}_p),\lambda}\left(\overline{L}(\lambda), ~v_{P_2}^{\rm{an}}(\lambda)\right)$$
which means that $W$ comes from an element in
$$\mathrm{Ker}(f)=\mathrm{Ext}^1_{\mathrm{GL}_3(\mathbb{Q}_p),\lambda}\left(\overline{L}(\lambda), ~\Sigma^{\flat}_1(\lambda, \mathscr{L}_1)\right)$$
and in particular
\begin{equation}\label{3equation 7.6}
\mathrm{Ext}^1_{\mathrm{GL}_3(\mathbb{Q}_p),\lambda}\left(\overline{L}(\lambda), ~\Sigma^{\flat}_1(\lambda, \mathscr{L}_1)\right)\neq 0
\end{equation}
The short exact sequence
$$\overline{L}(\lambda)\otimes_Ev_{P_2}^{\infty}\hookrightarrow W_2\twoheadrightarrow\overline{L}(\lambda)$$
induces an injection
\begin{equation}\label{3equation 7.3}
\mathrm{Ext}^1_{\mathrm{GL}_3(\mathbb{Q}_p),\lambda}\left(\overline{L}(\lambda), ~\Sigma^{\flat}_1(\lambda, \mathscr{L}_1)\right)\hookrightarrow\mathrm{Ext}^1_{\mathrm{GL}_3(\mathbb{Q}_p),\lambda}\left(W_2, ~\Sigma^{\flat}_1(\lambda, \mathscr{L}_1)\right).
\end{equation}
On the other hand, the short exact sequence
\begin{equation}\label{3equation 7.3.1}
\overline{L}(\lambda)\otimes_E\mathrm{St}_3^{\infty}\hookrightarrow\Sigma_1(\lambda, \mathscr{L}_1)\twoheadrightarrow\Sigma^{\flat}_1(\lambda, \mathscr{L}_1)
\end{equation}
induces a long exact sequence
\begin{multline*}
\mathrm{Ext}^1_{\mathrm{GL}_3(\mathbb{Q}_p),\lambda}\left(W_2, ~\overline{L}(\lambda)\otimes_E\mathrm{St}_3^{\infty}\right)\rightarrow\mathrm{Ext}^1_{\mathrm{GL}_3(\mathbb{Q}_p),\lambda}\left(W_2, ~\Sigma_1(\lambda, \mathscr{L}_1)\right)\\ \rightarrow\mathrm{Ext}^1_{\mathrm{GL}_3(\mathbb{Q}_p),\lambda}\left(W_2, ~\Sigma^{\flat}_1(\lambda, \mathscr{L}_1)\right)\rightarrow\mathrm{Ext}^2_{\mathrm{GL}_3(\mathbb{Q}_p),\lambda}\left(W_2, ~\overline{L}(\lambda)\otimes_E\mathrm{St}_3^{\infty}\right)
\end{multline*}
which implies
\begin{equation}\label{3equation 7.4}
\mathrm{Ext}^1_{\mathrm{GL}_3(\mathbb{Q}_p),\lambda}\left(W_2, ~\Sigma_1(\lambda, \mathscr{L}_1)\right)\xrightarrow{\sim}\mathrm{Ext}^1_{\mathrm{GL}_3(\mathbb{Q}_p),\lambda}\left(W_2, ~\Sigma^{\flat}_1(\lambda, \mathscr{L}_1)\right)
\end{equation}
as we have
$$\mathrm{Ext}^1_{\mathrm{GL}_3(\mathbb{Q}_p),\lambda}\left(W_2, ~\overline{L}(\lambda)\otimes_E\mathrm{St}_3^{\infty}\right)=\mathrm{Ext}^2_{\mathrm{GL}_3(\mathbb{Q}_p),\lambda}\left(W_2, ~\overline{L}(\lambda)\otimes_E\mathrm{St}_3^{\infty}\right)=0$$
from Lemma~\ref{3lemm: vanishing locally algebraic}. We combine Lemma~\ref{3lemm: ext6}, (\ref{3equation 7.3}) and (\ref{3equation 7.4}) and deduce that
$$\mathrm{Ext}^1_{\mathrm{GL}_3(\mathbb{Q}_p),\lambda}\left(\overline{L}(\lambda), ~\Sigma^{\flat}_1(\lambda, \mathscr{L}_1)\right)=0$$
which contradicts (\ref{3equation 7.6}). In all, we have thus shown that
\begin{equation}\label{3equation 7.5}
\mathrm{dim}_E\mathrm{Ext}^1_{\mathrm{GL}_3(\mathbb{Q}_p),\lambda}\left(\overline{L}(\lambda), ~\Sigma^{\flat}(\lambda, \mathscr{L}_1, \mathscr{L}_2)\right)<\mathrm{dim}_E\mathrm{Ext}^1_{\mathrm{GL}_3(\mathbb{Q}_p),\lambda}\left(\overline{L}(\lambda), ~C^1_{s_2,1}\oplus C^1_{s_1,1}\right)=2
\end{equation}
by combining Lemma~\ref{3lemm: nonvanishing ext1}. Finally, the vanishing
$$\mathrm{Ext}^1_{\mathrm{GL}_3(\mathbb{Q}_p),\lambda}\left(\overline{L}(\lambda), ~\overline{L}(\lambda)\otimes_E\mathrm{St}_3^{\infty}\right)=0$$
from Proposition~\ref{3prop: locally algebraic extension} implies an injection
$$\mathrm{Ext}^1_{\mathrm{GL}_3(\mathbb{Q}_p),\lambda}\left(\overline{L}(\lambda), ~\Sigma(\lambda, \mathscr{L}_1, \mathscr{L}_2)\right)\hookrightarrow\mathrm{Ext}^1_{\mathrm{GL}_3(\mathbb{Q}_p),\lambda}\left(\overline{L}(\lambda), ~\Sigma^{\flat}(\lambda, \mathscr{L}_1, \mathscr{L}_2)\right)$$
which finishes the proof by combining Lemma~\ref{3lemm: ext2} and (\ref{3equation 7.5}).
\end{proof}

\begin{lemm}\label{3lemm: ext8}
We have
$$\mathrm{dim}_E\mathrm{Ext}^1_{\mathrm{GL}_3(\mathbb{Q}_p),\lambda}\left(W_0, ~\Sigma(\lambda, \mathscr{L}_1, \mathscr{L}_2)\right)=2.$$
\end{lemm}
\begin{proof}
The short exact sequence
$$\Sigma_i^{\flat}(\lambda, \mathscr{L}_i)\hookrightarrow\Sigma^{\flat}(\lambda, \mathscr{L}_1, \mathscr{L}_2)\twoheadrightarrow v_{P_{3-i}}^{\rm{an}}(\lambda)$$
induces a long exact sequence
\begin{multline}
\mathrm{Hom}_{\mathrm{GL}_3(\mathbb{Q}_p),\lambda}\left(\overline{L}(\lambda)\otimes_E v_{P_{3-i}}^{\infty}, ~v_{P_{3-i}}^{\rm{an}}(\lambda)\right)\hookrightarrow\mathrm{Ext}^1_{\mathrm{GL}_3(\mathbb{Q}_p),\lambda}\left(\overline{L}(\lambda)\otimes_E v_{P_{3-i}}^{\infty}, ~\Sigma_i^{\flat}(\lambda, \mathscr{L}_i)\right)\\
\rightarrow \mathrm{Ext}^1_{\mathrm{GL}_3(\mathbb{Q}_p),\lambda}\left(\overline{L}(\lambda)\otimes_E v_{P_{3-i}}^{\infty}, ~\Sigma^{\flat}(\lambda, \mathscr{L}_1, \mathscr{L}_2)\right)\rightarrow \mathrm{Ext}^1_{\mathrm{GL}_3(\mathbb{Q}_p),\lambda}\left(\overline{L}(\lambda)\otimes_E v_{P_{3-i}}^{\infty}, ~v_{P_{3-i}}^{\rm{an}}(\lambda)\right).
\end{multline}
It is easy to observe that
$$
\mathrm{dim}_E\mathrm{Hom}_{\mathrm{GL}_3(\mathbb{Q}_p),\lambda}\left(\overline{L}(\lambda)\otimes_E v_{P_{3-i}}^{\infty}, ~v_{P_{3-i}}^{\rm{an}}(\lambda)\right)=1
$$
and
$$
\mathrm{Ext}^1_{\mathrm{GL}_3(\mathbb{Q}_p),\lambda}\left(\overline{L}(\lambda)\otimes_E v_{P_{3-i}}^{\infty}, ~v_{P_{3-i}}^{\rm{an}}(\lambda)\right)=0$$
from Proposition~\ref{3prop: locally algebraic extension} and Lemma~\ref{3lemm: nonvanishing ext1}. We can actually observe from Lemma~\ref{3lemm: nonvanishing ext1} that the only $V\in\mathrm{JH}_{\mathrm{GL}_3(\mathbb{Q}_p)}(\Sigma_i^{\flat}(\lambda, \mathscr{L}_i))$ such that
$$\mathrm{Ext}^1_{\mathrm{GL}_3(\mathbb{Q}_p),\lambda}\left(\overline{L}(\lambda)\otimes_E v_{P_{3-i}}^{\infty}, ~V\right)\neq 0$$
is $V=C^2_{s_{3-i}, 1}$ and
$$\mathrm{dim}_E\mathrm{Ext}^1_{\mathrm{GL}_3(\mathbb{Q}_p),\lambda}\left(\overline{L}(\lambda)\otimes_E v_{P_{3-i}}^{\infty}, ~C^2_{s_{3-i}, 1}\right)=1.$$
Hence we deduce that
$$\mathrm{Ext}^1_{\mathrm{GL}_3(\mathbb{Q}_p),\lambda}\left(\overline{L}(\lambda)\otimes_E v_{P_{3-i}}^{\infty}, ~\Sigma_i^{\flat}(\lambda, \mathscr{L}_i)\right)\leq 1$$
and therefore
\begin{equation}\label{3equation 8.1}
\mathrm{Ext}^1_{\mathrm{GL}_3(\mathbb{Q}_p),\lambda}\left(\overline{L}(\lambda)\otimes_E v_{P_{3-i}}^{\infty}, ~\Sigma^{\flat}(\lambda, \mathscr{L}_1, \mathscr{L}_2)\right)=0
\end{equation}
for $i=1,2$. The short exact sequence
$$\overline{L}(\lambda)\otimes_E \left(v_{P_1}^{\infty}\oplus v_{P_2}^{\infty}\right)\hookrightarrow W_0\twoheadrightarrow \overline{L}(\lambda)$$
induces
\begin{multline*}
\mathrm{Ext}^1_{\mathrm{GL}_3(\mathbb{Q}_p),\lambda}\left(\overline{L}(\lambda), ~\Sigma^{\flat}(\lambda, \mathscr{L}_1, \mathscr{L}_2)\right)\hookrightarrow \mathrm{Ext}^1_{\mathrm{GL}_3(\mathbb{Q}_p),\lambda}\left(W_0, ~\Sigma^{\flat}(\lambda, \mathscr{L}_1, \mathscr{L}_2)\right)\\ \rightarrow \mathrm{Ext}^1_{\mathrm{GL}_3(\mathbb{Q}_p),\lambda}\left(\overline{L}(\lambda)\otimes_E \left(v_{P_{1}}^{\infty}\oplus v_{P_2}^{\infty}\right), ~\Sigma^{\flat}(\lambda, \mathscr{L}_1, \mathscr{L}_2)\right)
\end{multline*}
which implies
\begin{equation}\label{3equation 8.2}
\mathrm{Ext}^1_{\mathrm{GL}_3(\mathbb{Q}_p),\lambda}\left(\overline{L}(\lambda), ~\Sigma^{\flat}(\lambda, \mathscr{L}_1, \mathscr{L}_2)\right)\xrightarrow{\sim} \mathrm{Ext}^1_{\mathrm{GL}_3(\mathbb{Q}_p),\lambda}\left(W_0, ~\Sigma^{\flat}(\lambda, \mathscr{L}_1, \mathscr{L}_2)\right)
\end{equation}
by (\ref{3equation 8.1}). Finally, the short exact sequence (\ref{3equation 7.3.1}) induces
\begin{multline*}
\mathrm{Ext}^1_{\mathrm{GL}_3(\mathbb{Q}_p),\lambda}\left(W_0, ~\overline{L}(\lambda)\otimes_E\mathrm{St}_3^{\infty}\right)\hookrightarrow\mathrm{Ext}^1_{\mathrm{GL}_3(\mathbb{Q}_p),\lambda}\left(W_0, ~\Sigma(\lambda, \mathscr{L}_1, \mathscr{L}_2)\right)\\
\rightarrow\mathrm{Ext}^1_{\mathrm{GL}_3(\mathbb{Q}_p),\lambda}\left(W_0, ~\Sigma^{\flat}(\lambda, \mathscr{L}_1, \mathscr{L}_2)\right)\rightarrow\mathrm{Ext}^2_{\mathrm{GL}_3(\mathbb{Q}_p),\lambda}\left(W_0, ~\overline{L}(\lambda)\otimes_E\mathrm{St}_3^{\infty}\right)
\end{multline*}
which finishes the proof by
$$\mathrm{dim}_E\mathrm{Ext}^1_{\mathrm{GL}_3(\mathbb{Q}_p),\lambda}\left(W_0, ~\overline{L}(\lambda)\otimes_E\mathrm{St}_3^{\infty}\right)=1\mbox{ and }\mathrm{Ext}^2_{\mathrm{GL}_3(\mathbb{Q}_p),\lambda}\left(W_0, ~\overline{L}(\lambda)\otimes_E\mathrm{St}_3^{\infty}\right)=0$$
from Lemma~\ref{3lemm: vanishing locally algebraic main}, and by Lemma~\ref{3lemm: ext7} as well as (\ref{3equation 8.2}).
\end{proof}

\begin{lemm}\label{3lemm: ext9}
We have the inequality
$$\mathrm{dim}_E\mathrm{Ext}^1_{\mathrm{GL}_3(\mathbb{Q}_p),\lambda}\left(W_0, ~\begin{xy}
(0,0)*+{v_{P_i}^{\rm{an}}(\lambda)}="a"; (18,0)*+{C_{s_i,s_i}}="b";
{\ar@{-}"a";"b"};
\end{xy}\right)\leq 2$$
for $i=1,2$.
\end{lemm}
\begin{proof}
We know that
$$\mathrm{Ext}^1_{\mathrm{GL}_3(\mathbb{Q}_p),\lambda}\left(\overline{L}(\lambda)\otimes_Ev_{P_j}^{\infty}, ~C^1_{s_i,1}\right)=\mathrm{Ext}^1_{\mathrm{GL}_3(\mathbb{Q}_p),\lambda}\left(\overline{L}(\lambda)\otimes_Ev_{P_j}^{\infty}, ~\overline{L}(\lambda)\otimes_Ev_{P_i}^{\infty}\right)=0$$
for $i,j=1,2$ from Proposition~\ref{3prop: locally algebraic extension} and Lemma~\ref{3lemm: nonvanishing ext1}, and thus
$$\mathrm{Ext}^1_{\mathrm{GL}_3(\mathbb{Q}_p),\lambda}\left(\overline{L}(\lambda)\otimes_Ev_{P_j}^{\infty}, ~v_{P_i}^{\rm{an}}(\lambda)\right)=0$$
for $i,j=1,2$ which together with (\ref{3equation ext 6.1.2 prime}) imply that
\begin{multline}\label{3equation 9.1}
\mathrm{dim}_E\mathrm{Ext}^1_{\mathrm{GL}_3(\mathbb{Q}_p),\lambda}\left(W_0, ~v_{P_i}^{\rm{an}}(\lambda)\right)\leq \mathrm{dim}_E\mathrm{Ext}^1_{\mathrm{GL}_3(\mathbb{Q}_p),\lambda}\left(W_i, ~v_{P_i}^{\rm{an}}(\lambda)\right)\\
\leq \mathrm{dim}_E\mathrm{Ext}^1_{\mathrm{GL}_3(\mathbb{Q}_p),\lambda}\left(\overline{L}(\lambda), ~v_{P_i}^{\rm{an}}(\lambda)\right)-\mathrm{dim}_E\mathrm{Hom}_{\mathrm{GL}_3(\mathbb{Q}_p),\lambda}\left(\overline{L}(\lambda)\otimes_Ev_{P_i}^{\infty}, ~v_{P_i}^{\rm{an}}(\lambda)\right)\\
=2-1=1.
\end{multline}
On the other hand, note that
$$\mathrm{Ext}^1_{\mathrm{GL}_3(\mathbb{Q}_p),\lambda}\left(\overline{L}(\lambda), ~C_{s_i,s_i}\right)=\mathrm{Ext}^1_{\mathrm{GL}_3(\mathbb{Q}_p),\lambda}\left(\overline{L}(\lambda)\otimes_Ev_{P_i}^{\infty}, ~C_{s_i,s_i}\right)=0$$
by Lemma~\ref{3lemm: nonvanishing ext1} and thus we have
\begin{equation}\label{3equation 9.2}
\mathrm{dim}_E\mathrm{Ext}^1_{\mathrm{GL}_3(\mathbb{Q}_p),\lambda}\left(W_0, ~C_{s_i,s_i}\right)\leq \mathrm{dim}_E\mathrm{Ext}^1_{\mathrm{GL}_3(\mathbb{Q}_p),\lambda}\left(\overline{L}(\lambda)\otimes_Ev_{P_{3-i}}^{\infty}, ~C_{s_i,s_i}\right)=1
\end{equation}
where the last equality follows again from Lemma~\ref{3lemm: nonvanishing ext1}. We finish the proof by combining (\ref{3equation 9.1}) and (\ref{3equation 9.2}) with the inequality
\begin{multline*}
\mathrm{dim}_E\mathrm{Ext}^1_{\mathrm{GL}_3(\mathbb{Q}_p),\lambda}\left(W_0, ~\begin{xy}
(0,0)*+{v_{P_i}^{\rm{an}}(\lambda)}="a"; (18,0)*+{C_{s_i,s_i}}="b";
{\ar@{-}"a";"b"};
\end{xy}\right)\\
\leq \mathrm{dim}_E\mathrm{Ext}^1_{\mathrm{GL}_3(\mathbb{Q}_p),\lambda}\left(W_0, ~v_{P_i}^{\rm{an}}(\lambda)\right)+\mathrm{dim}_E\mathrm{Ext}^1_{\mathrm{GL}_3(\mathbb{Q}_p),\lambda}\left(W_0, ~C_{s_i,s_i}\right).
\end{multline*}
\end{proof}
\section{Key exact sequences}\label{3section: exact sequence min}
\begin{lemm}\label{3lemm: upper bound}
We have the inequality
$$\mathrm{dim}_E\mathrm{Ext}^1_{\mathrm{GL}_3(\mathbb{Q}_p),\lambda}\left(W_0, ~\Sigma^+(\lambda, \mathscr{L}_1, \mathscr{L}_2)\right)\leq 3.$$
\end{lemm}
\begin{proof}
The short exact sequence
$$\Sigma(\lambda, \mathscr{L}_1, \mathscr{L}_2)\hookrightarrow\Sigma^+(\lambda, \mathscr{L}_1, \mathscr{L}_2)\twoheadrightarrow C_{s_1,s_1}\oplus C_{s_2,s_2}$$
induces the exact sequence
\begin{multline}\label{3first key exact sequence}
\mathrm{Ext}^1_{\mathrm{GL}_3(\mathbb{Q}_p),\lambda}\left(W_0, ~\Sigma(\lambda, \mathscr{L}_1, \mathscr{L}_2)\right)\hookrightarrow \mathrm{Ext}^1_{\mathrm{GL}_3(\mathbb{Q}_p),\lambda}\left(W_0, ~\Sigma^+(\lambda, \mathscr{L}_1, \mathscr{L}_2)\right)\\ \rightarrow\mathrm{Ext}^1_{\mathrm{GL}_3(\mathbb{Q}_p),\lambda}\left(W_0, ~C_{s_1,s_1}\oplus C_{s_2,s_2}\right).
\end{multline}
We know that
\begin{multline*}
\mathrm{dim}_E\mathrm{Ext}^1_{\mathrm{GL}_3(\mathbb{Q}_p),\lambda}\left(W_0, ~C_{s_1,s_1}\oplus C_{s_2,s_2}\right)\\=\mathrm{dim}_E\mathrm{Ext}^1_{\mathrm{GL}_3(\mathbb{Q}_p),\lambda}\left(W_0, ~C_{s_1,s_1}\right)+\mathrm{dim}_E\mathrm{Ext}^1_{\mathrm{GL}_3(\mathbb{Q}_p),\lambda}\left(W_0, ~C_{s_2,s_2}\right)=1+1=2
\end{multline*}
by Lemma~\ref{3lemm: nonvanishing ext1} and Lemma~\ref{3lemm: nonvanishing ext2}. We also know that $$\mathrm{dim}_E\mathrm{Ext}^1_{\mathrm{GL}_3(\mathbb{Q}_p),\lambda}\left(W_0, ~\Sigma(\lambda, \mathscr{L}_1, \mathscr{L}_2)\right)=2$$
by Lemma~\ref{3lemm: ext8}, and thus we obtain the following inequality:
\begin{multline}
\mathrm{dim}_E\mathrm{Ext}^1_{\mathrm{GL}_3(\mathbb{Q}_p),\lambda}\left(W_0, ~\Sigma^+(\lambda, \mathscr{L}_1, \mathscr{L}_2)\right)\\ \leq \mathrm{dim}_E\mathrm{Ext}^1_{\mathrm{GL}_3(\mathbb{Q}_p),\lambda}\left(W_0, ~\Sigma(\lambda, \mathscr{L}_1, \mathscr{L}_2)\right)+\mathrm{dim}_E\mathrm{Ext}^1_{\mathrm{GL}_3(\mathbb{Q}_p),\lambda}\left(W_0, ~C_{s_1,s_1}\oplus C_{s_2,s_2}\right)
=2+2=4.
\end{multline}
Assume first that
\begin{equation}\label{3assumption upper bound}
\mathrm{dim}_E\mathrm{Ext}^1_{\mathrm{GL}_3(\mathbb{Q}_p),\lambda}\left(W_0, ~\Sigma^+(\lambda, \mathscr{L}_1, \mathscr{L}_2)\right)=4.
\end{equation}
The short exact sequence
$$\Sigma^+_1(\lambda,\mathscr{L}_1)\hookrightarrow \Sigma^+(\lambda, \mathscr{L}_1, \mathscr{L}_2)\twoheadrightarrow \left(\begin{xy}
(0,0)*+{v_{P_2}^{\rm{an}}(\lambda)}="a"; (18,0)*+{C_{s_2,s_2}}="b";
{\ar@{-}"a";"b"};
\end{xy}\right)$$
induces a long exact sequence
\begin{multline}
\mathrm{Ext}^1_{\mathrm{GL}_3(\mathbb{Q}_p),\lambda}\left(W_0,~\Sigma^+_1(\lambda,\mathscr{L}_1)\right)\hookrightarrow\mathrm{Ext}^1_{\mathrm{GL}_3(\mathbb{Q}_p),\lambda}\left(W_0,~\Sigma^+(\lambda, \mathscr{L}_1, \mathscr{L}_2)\right)\\ \rightarrow\mathrm{Ext}^1_{\mathrm{GL}_3(\mathbb{Q}_p),\lambda}\left(W_0,~\begin{xy}
(0,0)*+{v_{P_2}^{\rm{an}}(\lambda)}="a"; (18,0)*+{C_{s_2,s_2}}="b";
{\ar@{-}"a";"b"};
\end{xy}\right)
\end{multline}
which implies
\begin{equation}\label{3lower bound}
\mathrm{dim}_E\mathrm{Ext}^1_{\mathrm{GL}_3(\mathbb{Q}_p),\lambda}\left(W_0,~\Sigma^+_1(\lambda,\mathscr{L}_1)\right)\geq 2
\end{equation}
by (\ref{3assumption upper bound}) and Lemma~\ref{3lemm: ext9}. We observe that $\Sigma^+_1(\lambda,\mathscr{L}_1)$ admits a filtration whose only reducible graded piece is
$$\begin{xy}
(0,0)*+{C^2_{s_1,1}}="a"; (20,0)*+{\overline{L}(\lambda)\otimes_E v_{P_1}^{\infty}}="b";
{\ar@{-}"a";"b"};
\end{xy}$$
and thus it follows from Lemma~\ref{3lemm: nonvanishing ext1} and
$$\mathrm{Ext}^1_{\mathrm{GL}_3(\mathbb{Q}_p),\lambda}\left(\overline{L}(\lambda)\otimes_E v_{P_1}^{\infty},~C^2_{s_1,1}-\overline{L}(\lambda)\otimes_E v_{P_1}^{\infty}\right)=0$$
(coming from Proposition~\ref{3prop: locally algebraic extension}, Lemma~\ref{3lemm: nonvanishing ext1} together with a simple devissage) that
$$\mathrm{Ext}^1_{\mathrm{GL}_3(\mathbb{Q}_p),\lambda}\left(\overline{L}(\lambda)\otimes_E v_{P_1}^{\infty},~V\right)=0$$
for all graded pieces of such a filtration except the subrepresentation $\overline{L}(\lambda)\otimes_E \mathrm{St}_3^{\infty}$. Hence we deduce by part (ii) of Proposition~\ref{3prop: formal devissages} an isomorphism of one dimensional spaces
\begin{equation}\label{3easy computation}
\mathrm{Ext}^1_{\mathrm{GL}_3(\mathbb{Q}_p),\lambda}\left(\overline{L}(\lambda)\otimes_E v_{P_1}^{\infty}, ~\overline{L}(\lambda)\otimes_E \mathrm{St}_3^{\infty}\right)\xrightarrow{\sim}\mathrm{Ext}^1_{\mathrm{GL}_3(\mathbb{Q}_p),\lambda}\left(\overline{L}(\lambda)\otimes_E v_{P_1}^{\infty}, ~\Sigma^+_1(\lambda,\mathscr{L}_1)\right).
\end{equation}
Then the short exact sequence
$$\overline{L}(\lambda)\otimes_Ev_{P_1}^{\infty}\hookrightarrow W_0\twoheadrightarrow W_2$$
induces a long exact sequence
\begin{multline*}
\mathrm{Ext}^1_{\mathrm{GL}_3(\mathbb{Q}_p),\lambda}\left(W_2, ~\Sigma^+_1(\lambda,\mathscr{L}_1)\right)\hookrightarrow \mathrm{Ext}^1_{\mathrm{GL}_3(\mathbb{Q}_p),\lambda}\left(W_0, ~\Sigma^+_1(\lambda,\mathscr{L}_1)\right)\\
\rightarrow \mathrm{Ext}^1_{\mathrm{GL}_3(\mathbb{Q}_p),\lambda}\left(\overline{L}(\lambda)\otimes_E v_{P_1}^{\infty}, ~\Sigma^+_1(\lambda,\mathscr{L}_1)\right)
\end{multline*}
which together with (\ref{3lower bound}) and (\ref{3easy computation}) implies that
$$\mathrm{dim}_E\mathrm{Ext}^1_{\mathrm{GL}_3(\mathbb{Q}_p),\lambda}\left(W_2,~\Sigma^+_1(\lambda,\mathscr{L}_1)\right)\geq 1$$
which contradicts Lemma~\ref{3lemm: ext6}. Hence we finish the proof.
\end{proof}

\begin{prop}\label{3prop: main dim}
We have
$$\mathrm{dim}_E\mathrm{Ext}^1_{\mathrm{GL}_3(\mathbb{Q}_p),\lambda}\left(W_0, ~\Sigma^+(\lambda, \mathscr{L}_1, \mathscr{L}_2)\right)=3.$$
\end{prop}
\begin{proof}
The short exact sequence
$$\overline{L}(\lambda)\otimes_E\left(v_{P_2}^{\infty}\oplus v_{P_1}^{\infty}\right)\hookrightarrow W_0\twoheadrightarrow \overline{L}(\lambda)$$  induces a long exact sequence
\begin{multline}\label{3main long exact sequence}
\mathrm{Ext}^1_{\mathrm{GL}_3(\mathbb{Q}_p),\lambda}\left(\overline{L}(\lambda), ~\Sigma^+(\lambda, \mathscr{L}_1, \mathscr{L}_2)\right)\hookrightarrow\mathrm{Ext}^1_{\mathrm{GL}_3(\mathbb{Q}_p),\lambda}\left(W_0, ~\Sigma^+(\lambda, \mathscr{L}_1, \mathscr{L}_2)\right)\\
\rightarrow \mathrm{Ext}^1_{\mathrm{GL}_3(\mathbb{Q}_p),\lambda}\left(\overline{L}(\lambda)\otimes_E \left(v_{P_2}^{\infty}\oplus v_{P_1}^{\infty}\right), ~\Sigma^+(\lambda, \mathscr{L}_1, \mathscr{L}_2)\right)\rightarrow \mathrm{Ext}^2_{\mathrm{GL}_3(\mathbb{Q}_p),\lambda}\left(\overline{L}(\lambda), ~\Sigma^+(\lambda, \mathscr{L}_1, \mathscr{L}_2)\right)
\end{multline}
and thus we have
\begin{multline}
\mathrm{dim}_E\mathrm{Ext}^1_{\mathrm{GL}_3(\mathbb{Q}_p),\lambda}(W_0, ~\Sigma^+(\lambda, \mathscr{L}_1, \mathscr{L}_2))\\
\geq \mathrm{dim}_E\mathrm{Ext}^1_{\mathrm{GL}_3(\mathbb{Q}_p),\lambda}(\overline{L}(\lambda), ~\Sigma^+(\lambda, \mathscr{L}_1, \mathscr{L}_2))+\mathrm{dim}_E\mathrm{Ext}^1_{\mathrm{GL}_3(\mathbb{Q}_p),\lambda}(\overline{L}(\lambda)\otimes_E\left(v_{P_2}^{\infty}\oplus v_{P_1}^{\infty}\right), ~\Sigma^+(\lambda, \mathscr{L}_1, \mathscr{L}_2))\\-\mathrm{dim}_E\mathrm{Ext}^2_{\mathrm{GL}_3(\mathbb{Q}_p),\lambda}(\overline{L}(\lambda), ~\Sigma^+(\lambda, \mathscr{L}_1, \mathscr{L}_2))=1+4-2=3
\end{multline}
due to Lemma~\ref{3lemm: ext5} and Lemma~\ref{3lemm: ext6.1}, which finishes the proof by combining with Lemma~\ref{3lemm: upper bound}.
\end{proof}

We define $\Sigma^{\sharp}(\lambda, \mathscr{L}_1, \mathscr{L}_2)$ as the unique non-split extension of $\Sigma(\lambda, \mathscr{L}_1, \mathscr{L}_2)$ by $\overline{L}(\lambda)$ ( cf. Lemma~\ref{3lemm: ext2}) and then set $\Sigma^{\sharp,+}(\lambda, \mathscr{L}_1, \mathscr{L}_2)$ to be the amalgamate sum of $\Sigma^{\sharp}(\lambda, \mathscr{L}_1, \mathscr{L}_2)$ and $\Sigma^+(\lambda, \mathscr{L}_1, \mathscr{L}_2)$ over $\Sigma(\lambda, \mathscr{L}_1, \mathscr{L}_2)$. Hence $\Sigma^{\sharp}(\lambda, \mathscr{L}_1, \mathscr{L}_2)$ has the form
$$
\begin{xy}
(0,0)*+{\mathrm{St}_3^{\rm{an}}(\lambda)}="a"; (20,4)*+{v_{P_1}^{\rm{an}}(\lambda)}="b"; (20,-4)*+{v_{P_2}^{\rm{an}}(\lambda)}="c"; (40,0)*+{\overline{L}(\lambda)}="d";
{\ar@{-}"a";"b"}; {\ar@{-}"a";"c"}; {\ar@{-}"b";"d"}; {\ar@{-}"c";"d"};
\end{xy}
$$
and $\Sigma^{\sharp,+}(\lambda, \mathscr{L}_1, \mathscr{L}_2)$ has the form
$$
\begin{xy}
(0,0)*+{\mathrm{St}_3^{\rm{an}}(\lambda)}="a"; (20,4)*+{v_{P_1}^{\rm{an}}(\lambda)}="b"; (20,-4)*+{v_{P_2}^{\rm{an}}(\lambda)}="c"; (40,8)*+{C_{s_1,s_1}}="d"; (40,-8)*+{C_{s_2,s_2}}="e"; (40,0)*+{\overline{L}(\lambda)}="f";
{\ar@{-}"a";"b"}; {\ar@{-}"a";"c"}; {\ar@{-}"b";"d"}; {\ar@{-}"c";"e"}; {\ar@{-}"b";"f"}; {\ar@{-}"c";"f"};
\end{xy}.
$$
It follows from Lemma~\ref{3lemm: ext2}, Proposition~\ref{3prop: locally algebraic extension}, (\ref{3vanishing of simple special ext}) and an easy devissage that
\begin{equation}\label{3vanishing of ext}
\mathrm{Ext}^1_{\mathrm{GL}_3(\mathbb{Q}_p),\lambda}\left(\overline{L}(\lambda),~\Sigma^{\sharp}(\lambda, \mathscr{L}_1, \mathscr{L}_2)\right)=\mathrm{Ext}^1_{\mathrm{GL}_3(\mathbb{Q}_p),\lambda}\left(\overline{L}(\lambda),~\Sigma^{\sharp,+}(\lambda, \mathscr{L}_1, \mathscr{L}_2)\right)=0.
\end{equation}
Then we set
$$\Sigma^{\ast,\flat}(\lambda, \mathscr{L}_1, \mathscr{L}_2):=\Sigma^{\ast}(\lambda, \mathscr{L}_1, \mathscr{L}_2)/\overline{L}(\lambda)\otimes_E\mathrm{St}_3^{\infty}$$
for $\ast=\{+\}, \{\sharp\}$ and $\{\sharp, +\}$.
It follows from Lemma~\ref{3lemm: ext7}, (\ref{3vanishing of simple special ext}) and an easy devissage that
\begin{equation}\label{3vanishing of ext prime}
\mathrm{Ext}^1_{\mathrm{GL}_3(\mathbb{Q}_p),\lambda}\left(\overline{L}(\lambda),~\Sigma^{\sharp, \flat}(\lambda, \mathscr{L}_1, \mathscr{L}_2)\right)=\mathrm{Ext}^1_{\mathrm{GL}_3(\mathbb{Q}_p),\lambda}\left(\overline{L}(\lambda),~\Sigma^{\sharp,+, \flat}(\lambda, \mathscr{L}_1, \mathscr{L}_2)\right)=0.
\end{equation}
\begin{lemm}\label{3lemm: ext10}
We have
$$\mathrm{Ext}^1_{\mathrm{GL}_3(\mathbb{Q}_p),\lambda}\left(\overline{L}(\lambda), ~\Sigma^{\sharp}(\lambda, \mathscr{L}_1,\mathscr{L}_2)\right)=\mathrm{Ext}^1_{\mathrm{GL}_3(\mathbb{Q}_p),\lambda}\left(\overline{L}(\lambda), ~\Sigma^{\sharp,+}(\lambda, \mathscr{L}_1,\mathscr{L}_2)\right)=0$$
and
$$\mathrm{dim}_E\mathrm{Ext}^2_{\mathrm{GL}_3(\mathbb{Q}_p),\lambda}\left(\overline{L}(\lambda), ~\Sigma^{\sharp}(\lambda, \mathscr{L}_1, \mathscr{L}_2)\right)=\mathrm{dim}_E\mathrm{Ext}^2_{\mathrm{GL}_3(\mathbb{Q}_p),\lambda}\left(\overline{L}(\lambda), ~\Sigma^{\sharp,+}(\lambda, \mathscr{L}_1, \mathscr{L}_2)\right)=2$$
\end{lemm}
\begin{proof}
It follows from (\ref{3vanishing of simple special ext}) that we only need to show that
$$\mathrm{Ext}^1_{\mathrm{GL}_3(\mathbb{Q}_p),\lambda}\left(\overline{L}(\lambda), ~\Sigma^{\sharp}(\lambda, \mathscr{L}_1,\mathscr{L}_2)\right)=0\mbox{ and }\mathrm{dim}_E\mathrm{Ext}^2_{\mathrm{GL}_3(\mathbb{Q}_p),\lambda}\left(\overline{L}(\lambda), ~\Sigma^{\sharp}(\lambda, \mathscr{L}_1, \mathscr{L}_2)\right)=2.$$
These results follow from combining the long exact sequence
\begin{multline*}
\mathrm{Hom}_{\mathrm{GL}_3(\mathbb{Q}_p),\lambda}\left(\overline{L}(\lambda), ~\overline{L}(\lambda)\right)\hookrightarrow\mathrm{Ext}^1_{\mathrm{GL}_3(\mathbb{Q}_p),\lambda}\left(\overline{L}(\lambda), ~\Sigma(\lambda, \mathscr{L}_1,\mathscr{L}_2)\right)\\ \rightarrow\mathrm{Ext}^1_{\mathrm{GL}_3(\mathbb{Q}_p),\lambda}\left(\overline{L}(\lambda), ~\Sigma^{\sharp}(\lambda, \mathscr{L}_1,\mathscr{L}_2)\right) \rightarrow\mathrm{Ext}^1_{\mathrm{GL}_3(\mathbb{Q}_p),\lambda}\left(\overline{L}(\lambda), ~\overline{L}(\lambda)\right)\\ \rightarrow\mathrm{Ext}^2_{\mathrm{GL}_3(\mathbb{Q}_p),\lambda}\left(\overline{L}(\lambda), ~\Sigma(\lambda, \mathscr{L}_1,\mathscr{L}_2)\right)\rightarrow\mathrm{Ext}^2_{\mathrm{GL}_3(\mathbb{Q}_p),\lambda}\left(\overline{L}(\lambda), ~\Sigma^{\sharp}(\lambda, \mathscr{L}_1,\mathscr{L}_2)\right) \\ \rightarrow\mathrm{Ext}^2_{\mathrm{GL}_3(\mathbb{Q}_p),\lambda}\left(\overline{L}(\lambda), ~\overline{L}(\lambda)\right),
\end{multline*}
with Lemma~\ref{3lemm: ext2} and the equalities
$$
\begin{array}{cccc}
\mathrm{dim}_E&\mathrm{Hom}_{\mathrm{GL}_3(\mathbb{Q}_p),\lambda}(\overline{L}(\lambda), ~\overline{L}(\lambda))&=&1\\
&\mathrm{Ext}^1_{\mathrm{GL}_3(\mathbb{Q}_p),\lambda}(\overline{L}(\lambda), ~\overline{L}(\lambda))&=&0\\
&\mathrm{Ext}^2_{\mathrm{GL}_3(\mathbb{Q}_p),\lambda}(\overline{L}(\lambda), ~\overline{L}(\lambda))&=&0
\end{array}
$$
due to Proposition~\ref{3prop: locally algebraic extension}.
\end{proof}

\begin{lemm}\label{3lemm: ext11}
We have
$$\mathrm{Ext}^1_{\mathrm{GL}_3(\mathbb{Q}_p),\lambda}(\overline{L}(\lambda), ~\Sigma^{\sharp,\flat}(\lambda, \mathscr{L}_1, \mathscr{L}_2))=\mathrm{Ext}^1_{\mathrm{GL}_3(\mathbb{Q}_p),\lambda}(\overline{L}(\lambda), ~\Sigma^{\sharp,+,\flat}(\lambda, \mathscr{L}_1, \mathscr{L}_2))=0$$
and
$$\mathrm{dim}_E\mathrm{Ext}^2_{\mathrm{GL}_3(\mathbb{Q}_p),\lambda}(\overline{L}(\lambda), ~\Sigma^{\sharp,\flat}(\lambda, \mathscr{L}_1, \mathscr{L}_2))= \mathrm{dim}_E\mathrm{Ext}^2_{\mathrm{GL}_3(\mathbb{Q}_p),\lambda}(\overline{L}(\lambda), ~\Sigma^{\sharp,+,\flat}(\lambda, \mathscr{L}_1, \mathscr{L}_2))\geq 1.$$
\end{lemm}
\begin{proof}
It follows from (\ref{3vanishing of simple special ext}) that we only need to show that
$$\mathrm{Ext}^1_{\mathrm{GL}_3(\mathbb{Q}_p),\lambda}(\overline{L}(\lambda), ~\Sigma^{\sharp,\flat}(\lambda, \mathscr{L}_1, \mathscr{L}_2))=0\mbox{ and }\mathrm{dim}_E\mathrm{Ext}^2_{\mathrm{GL}_3(\mathbb{Q}_p),\lambda}(\overline{L}(\lambda), ~\Sigma^{\sharp,\flat}(\lambda, \mathscr{L}_1, \mathscr{L}_2))\geq 1,$$
which follow from combining (\ref{3vanishing of ext prime}), Lemma~\ref{3lemm: ext10} and the long exact sequence
\begin{multline}\label{3equation 11.1}
\mathrm{Ext}^1_{\mathrm{GL}_3(\mathbb{Q}_p),\lambda}(\overline{L}(\lambda), ~\overline{L}(\lambda)\otimes_E\mathrm{St}_3^{\infty})\rightarrow\mathrm{Ext}^1_{\mathrm{GL}_3(\mathbb{Q}_p),\lambda}(\overline{L}(\lambda), ~\Sigma^{\sharp}(\lambda, \mathscr{L}_1, \mathscr{L}_2))\\ \rightarrow\mathrm{Ext}^1_{\mathrm{GL}_3(\mathbb{Q}_p),\lambda}(\overline{L}(\lambda), ~\Sigma^{\sharp,\flat}(\lambda, \mathscr{L}_1, \mathscr{L}_2)) \rightarrow\mathrm{Ext}^2_{\mathrm{GL}_3(\mathbb{Q}_p),\lambda}(\overline{L}(\lambda), ~\overline{L}(\lambda)\otimes_E\mathrm{St}_3^{\infty})\\ \rightarrow \mathrm{Ext}^2_{\mathrm{GL}_3(\mathbb{Q}_p),\lambda}(\overline{L}(\lambda), ~\Sigma^{\sharp}(\lambda, \mathscr{L}_1, \mathscr{L}_2))\rightarrow\mathrm{Ext}^2_{\mathrm{GL}_3(\mathbb{Q}_p),\lambda}(\overline{L}(\lambda), ~\Sigma^{\sharp,\flat}(\lambda, \mathscr{L}_1, \mathscr{L}_2))
\end{multline}
with the equalities
$$
\begin{array}{cccc}
&\mathrm{Ext}^1_{\mathrm{GL}_3(\mathbb{Q}_p),\lambda}(\overline{L}(\lambda), ~\overline{L}(\lambda)\otimes_E\mathrm{St}_3^{\infty})&=&0\\
\mathrm{dim}_E&\mathrm{Ext}^2_{\mathrm{GL}_3(\mathbb{Q}_p),\lambda}(\overline{L}(\lambda), ~\overline{L}(\lambda)\otimes_E\mathrm{St}_3^{\infty})&=&1\\
\end{array}
$$
due to Proposition~\ref{3prop: locally algebraic extension}.
\end{proof}

We use the shorten notation $\underline{\mathscr{L}}:=(\mathscr{L}_1, \mathscr{L}_2, \mathscr{L}_1^{\prime}, \mathscr{L}_2^{\prime})$ for a tuple of four elements in $E$. We recall from Proposition~\ref{3prop: isomorphism from cup product} an isomorphism of two dimensional spaces
\begin{equation}\label{3isomorphism for normalization}
\mathrm{Ext}^1_{\mathrm{GL}_3(\mathbb{Q}_p),\lambda}\left(\overline{L}(\lambda)\otimes_Ev_{P_i}^{\infty},~\Sigma^+(\lambda, \mathscr{L}_1, \mathscr{L}_2)\right)\xrightarrow{\sim}\mathrm{Ext}^2_{\mathrm{GL}_3(\mathbb{Q}_p),\lambda}\left(\overline{L}(\lambda),~\Sigma^+(\lambda, \mathscr{L}_1, \mathscr{L}_2)\right).
\end{equation}
We emphasize that the isomorphism (\ref{3isomorphism for normalization}) can be naturally interpreted as the cup product map
\begin{multline}\label{3cup product interpretation}
\mathrm{Ext}^1_{\mathrm{GL}_3(\mathbb{Q}_p),\lambda}\left(\overline{L}(\lambda)\otimes_Ev_{P_i}^{\infty},~\Sigma^+(\lambda, \mathscr{L}_1, \mathscr{L}_2)\right)~\cup~ \mathrm{Ext}^1_{\mathrm{GL}_3(\mathbb{Q}_p),\lambda}\left(\overline{L}(\lambda),~\overline{L}(\lambda)\otimes_Ev_{P_i}^{\infty}\right)\\
\rightarrow\mathrm{Ext}^2_{\mathrm{GL}_3(\mathbb{Q}_p),\lambda}\left(\overline{L}(\lambda),~\Sigma^+(\lambda, \mathscr{L}_1, \mathscr{L}_2)\right)
\end{multline}
where $\mathrm{Ext}^1_{\mathrm{GL}_3(\mathbb{Q}_p),\lambda}\left(\overline{L}(\lambda),~\overline{L}(\lambda)\otimes_Ev_{P_i}^{\infty}\right)$ is one dimensional by Proposition~\ref{3prop: locally algebraic extension}. We recall from the proof of Lemma~\ref{3lemm: ext6.1} that there is a canonical isomorphism
$$\mathrm{Ext}^2_{\mathrm{GL}_3(\mathbb{Q}_p),\lambda}\left(\overline{L}(\lambda),~\Sigma(\lambda, \mathscr{L}_1, \mathscr{L}_2)\right)\xrightarrow{\sim}\mathrm{Ext}^2_{\mathrm{GL}_3(\mathbb{Q}_p),\lambda}\left(\overline{L}(\lambda),~\Sigma^+(\lambda, \mathscr{L}_1, \mathscr{L}_2)\right)$$
which together with Lemma~\ref{3lemm: ext2} implies that $\mathrm{Ext}^2_{\mathrm{GL}_3(\mathbb{Q}_p),\lambda}\left(\overline{L}(\lambda),~\Sigma^+(\lambda, \mathscr{L}_1, \mathscr{L}_2)\right)$ admits a basis of the form
$$\{\kappa(b_{1,\mathrm{val}_p}\wedge b_{2,\mathrm{val}_p}), \iota_1(D_0)\},$$
and therefore the element
$$\iota_1(D_0)+\mathscr{L}\kappa(b_{1,\mathrm{val}_p}\wedge b_{2,\mathrm{val}_p})$$
generates a line in $\mathrm{Ext}^2_{\mathrm{GL}_3(\mathbb{Q}_p),\lambda}\left(\overline{L}(\lambda),~\Sigma^+(\lambda, \mathscr{L}_1, \mathscr{L}_2)\right)$ for each $\mathscr{L}\in E$. We define $\Sigma^+_i(\lambda, \mathscr{L}_1, \mathscr{L}_2, \mathscr{L}_i^{\prime})$ as the representation represent by the preimage of $$\iota_1(D_0)+\mathscr{L}_i^{\prime}\kappa(b_{1,\mathrm{val}_p}\wedge b_{2,\mathrm{val}_p})$$
in
$$\mathrm{Ext}^1_{\mathrm{GL}_3(\mathbb{Q}_p),\lambda}\left(\overline{L}(\lambda)\otimes_Ev_{P_i}^{\infty},~\Sigma^+(\lambda, \mathscr{L}_1, \mathscr{L}_2)\right)$$
via (\ref{3isomorphism for normalization}) for $i=1,2$. Then we define $\Sigma^+(\lambda,\underline{\mathscr{L}})$ as the amalgamate sum of $\Sigma^+_1(\lambda, \mathscr{L}_1, \mathscr{L}_2, \mathscr{L}_1^{\prime})$ and $\Sigma^+_2(\lambda, \mathscr{L}_1, \mathscr{L}_2, \mathscr{L}_2^{\prime})$ over $\Sigma^+(\lambda, \mathscr{L}_1, \mathscr{L}_2)$, and therefore $\Sigma^+(\lambda,\underline{\mathscr{L}})$ has the form
$$\begin{xy}
(0,0)*+{\mathrm{St}_3^{\rm{an}}(\lambda)}="a"; (20,4)*+{v_{P_1}^{\rm{an}}(\lambda)}="b"; (20,-4)*+{v_{P_2}^{\rm{an}}(\lambda)}="c"; (40,4)*+{C_{s_1,s_1}}="d"; (40,-4)*+{C_{s_2,s_2}}="e"; (60,4)*+{\overline{L}(\lambda)\otimes_Ev_{P_2}^{\infty}}="g"; (60,-4)*+{\overline{L}(\lambda)\otimes_Ev_{P_1}^{\infty}}="h";
{\ar@{-}"a";"b"}; {\ar@{-}"a";"c"}; {\ar@{-}"b";"d"}; {\ar@{-}"c";"e"}; {\ar@{-}"d";"g"}; {\ar@{-}"e";"h"};
\end{xy}.$$
We define $\Sigma^{\sharp,+}(\lambda,\underline{\mathscr{L}})$ as the amalgamate sum of $\Sigma^{+}(\lambda,\underline{\mathscr{L}})$ and $\Sigma^{\sharp}(\lambda,\mathscr{L}_1, \mathscr{L}_2)$ over $\Sigma(\lambda,\mathscr{L}_1, \mathscr{L}_2)$, and thus $\Sigma^{\sharp,+}(\lambda,\underline{\mathscr{L}})$ has the form
$$\begin{xy}
(0,0)*+{\mathrm{St}_3^{\rm{an}}(\lambda)}="a"; (20,4)*+{v_{P_1}^{\rm{an}}(\lambda)}="b"; (20,-4)*+{v_{P_2}^{\rm{an}}(\lambda)}="c"; (40,8)*+{C_{s_1,s_1}}="d"; (40,-8)*+{C_{s_2,s_2}}="e"; (40,0)*+{\overline{L}(\lambda)}="f"; (60,8)*+{\overline{L}(\lambda)\otimes_Ev_{P_2}^{\infty}}="g"; (60,-8)*+{\overline{L}(\lambda)\otimes_Ev_{P_1}^{\infty}}="h";
{\ar@{-}"a";"b"}; {\ar@{-}"a";"c"}; {\ar@{-}"b";"d"}; {\ar@{-}"c";"e"}; {\ar@{-}"b";"f"}; {\ar@{-}"c";"f"}; {\ar@{-}"d";"g"}; {\ar@{-}"e";"h"};
\end{xy}.$$
We also need the quotients
$$\Sigma^{+,\flat}(\lambda,\underline{\mathscr{L}}):=\Sigma^+(\lambda,\underline{\mathscr{L}})/\overline{L}(\lambda)\otimes_E\mathrm{St}_3^{\infty},~\Sigma^{\sharp,+,\flat}(\lambda,\underline{\mathscr{L}}):=\Sigma^{\sharp,+}(\lambda,\underline{\mathscr{L}})/\overline{L}(\lambda)\otimes_E\mathrm{St}_3^{\infty}.$$
\begin{lemm}\label{3lemm: ext13}
We have the inequality
$$\mathrm{dim}_E\mathrm{Ext}^1_{\mathrm{GL}_3(\mathbb{Q}_p),\lambda}\left(\overline{L}(\lambda), ~\Sigma^{\sharp,+,\flat}(\lambda,\underline{\mathscr{L}})\right)\leq 1.$$
\end{lemm}
\begin{proof}
The short exact sequence
$$\Sigma^{\sharp,+,\flat}(\lambda,\mathscr{L}_1,\mathscr{L}_2)\hookrightarrow\Sigma^{\sharp,+,\flat}(\lambda,\underline{\mathscr{L}})\twoheadrightarrow \overline{L}(\lambda)\otimes_E\left(v_{P_2}^{\infty}\oplus v_{P_1}^{\infty}\right)$$
induces an injection
\begin{equation}\label{3equation 13.1}
\mathrm{Ext}^1_{\mathrm{GL}_3(\mathbb{Q}_p),\lambda}\left(\overline{L}(\lambda), ~\Sigma^{\sharp,+,\flat}(\lambda,\underline{\mathscr{L}})\right)\hookrightarrow \mathrm{Ext}^1_{\mathrm{GL}_3(\mathbb{Q}_p),\lambda}\left(\overline{L}(\lambda), ~\overline{L}(\lambda)\otimes_E\left(v_{P_2}^{\infty}\oplus v_{P_1}^{\infty}\right)\right)
\end{equation}
by Lemma~\ref{3lemm: ext11}. Note that we have
$$\mathrm{dim}_E\mathrm{Ext}^1_{\mathrm{GL}_3(\mathbb{Q}_p),\lambda}\left(\overline{L}(\lambda), ~\overline{L}(\lambda)\otimes_E\left(v_{P_2}^{\infty}\oplus v_{P_1}^{\infty}\right)\right)=2$$
by Proposition~\ref{3prop: locally algebraic extension}. Assume first that (\ref{3equation 13.1}) is a surjection, and thus we can pick a representation $W$ represented by a non-zero element lying in the preimage of $\overline{L}(\lambda)\otimes_Ev_{P_2}^{\infty}$ under (\ref{3equation 13.1}). We observe that the very existence of $W$ implies that
\begin{equation}\label{3nonvanishing of ext 13.1}
\mathrm{Ext}^1_{\mathrm{GL}_3(\mathbb{Q}_p),\lambda}\left(W_2,~\Sigma^{\sharp,+,\flat}(\lambda,\mathscr{L}_1,\mathscr{L}_2)\right)\neq 0.
\end{equation}
We define
$$\Sigma_i^{+,\flat}(\lambda,\mathscr{L}_i):=\Sigma_i^+(\lambda,\mathscr{L}_i)/\overline{L}(\lambda)\otimes_E\mathrm{St}_3^{\infty}$$
and thus we have an embedding
$$\Sigma_i^{+,\flat}(\lambda,\mathscr{L}_i)\hookrightarrow\Sigma^{\sharp,+,\flat}(\lambda,\mathscr{L}_1,\mathscr{L}_2)$$
for each $i=1,2$.
We notice that the quotient $\Sigma^{\sharp,+,\flat}(\lambda,\mathscr{L}_1,\mathscr{L}_2)/\Sigma_1^{+,\flat}(\lambda,\mathscr{L}_1)$ fits into a short exact sequence
$$\left(\begin{xy}
(0,0)*+{v_{P_2}^{\rm{an}}(\lambda)}="a"; (16,0)*+{\overline{L}(\lambda)}="b";
{\ar@{-}"a";"b"};
\end{xy}\right)\hookrightarrow\Sigma^{\sharp,+,\flat}(\lambda,\mathscr{L}_1,\mathscr{L}_2)/\Sigma_1^{+,\flat}(\lambda,\mathscr{L}_1)\twoheadrightarrow C_{s_2,s_2}.$$
Hence it remains to show the equality
\begin{equation}\label{3equation 13.3}
\mathrm{Ext}^1_{\mathrm{GL}_3(\mathbb{Q}_p),\lambda}\left(W_2,~\begin{xy}
(0,0)*+{v_{P_2}^{\rm{an}}(\lambda)}="a"; (16,0)*+{\overline{L}(\lambda)}="b";
{\ar@{-}"a";"b"};
\end{xy}\right)=0
\end{equation}
and the equality
\begin{equation}\label{3equation 13.4}
\mathrm{Ext}^1_{\mathrm{GL}_3(\mathbb{Q}_p),\lambda}\left(W_2, ~C_{s_2,s_2}\right)=0
\end{equation}
to finish the proof of
\begin{equation}\label{3equation 13.2}
\mathrm{Ext}^1_{\mathrm{GL}_3(\mathbb{Q}_p),\lambda}\left(W_2, ~\Sigma^{\sharp,+,\flat}(\lambda,\mathscr{L}_1,\mathscr{L}_2)/\Sigma_1^{+,\flat}(\lambda,\mathscr{L}_1)\right)=0.
\end{equation}
The vanishing (\ref{3equation 13.4}) follows from Lemma~\ref{3lemm: nonvanishing ext1} and part (i) of Proposition~\ref{3prop: formal devissages}.
It follows from Proposition~\ref{3prop: locally algebraic extension}, Lemma~\ref{3lemm: nonvanishing ext1} and a simple devissage that
\begin{equation}\label{3equation 13.5}
\mathrm{Ext}^1_{\mathrm{GL}_3(\mathbb{Q}_p),\lambda}\left(\overline{L}(\lambda)\otimes_Ev_{P_2}^{\infty},~C^1_{s_1,1}\right)=\mathrm{Ext}^1_{\mathrm{GL}_3(\mathbb{Q}_p),\lambda}\left(\overline{L}(\lambda),~\begin{xy}
(0,0)*+{C^1_{s_1,1}}="a"; (16,0)*+{\overline{L}(\lambda)}="b";
{\ar@{-}"a";"b"};
\end{xy}\right)=0.
\end{equation}
Hence if
$$\mathrm{Ext}^1_{\mathrm{GL}_3(\mathbb{Q}_p),\lambda}\left(W_2,~\begin{xy}
(0,0)*+{C^1_{s_1,1}}="a"; (16,0)*+{\overline{L}(\lambda)}="b";
{\ar@{-}"a";"b"};
\end{xy}\right)\neq0$$
then there exists a uniserial representation of the form
$$\begin{xy}
(0,0)*+{C^1_{s_1,1}}="a"; (16,0)*+{\overline{L}(\lambda)}="b"; (35,0)*+{\overline{L}(\lambda)\otimes_Ev_{P_2}^{\infty}}="c";
{\ar@{-}"a";"b"}; {\ar@{-}"b";"c"};
\end{xy}$$
which contradicts (\ref{3equation 13.5}) and Lemma~\ref{3lemm: special vanishing 1}. As a result, we have shown that
$$\mathrm{Ext}^1_{\mathrm{GL}_3(\mathbb{Q}_p),\lambda}\left(W_2,~\begin{xy}
(0,0)*+{C^1_{s_1,1}}="a"; (16,0)*+{\overline{L}(\lambda)}="b";
{\ar@{-}"a";"b"};
\end{xy}\right)=0$$
which together with Proposition~\ref{3prop: locally algebraic extension} and part (i) of Proposition~\ref{3prop: formal devissages} implies (\ref{3equation 13.3}) and hence (\ref{3equation 13.2}) as well concerning (\ref{3equation 13.4}). Therefore we can combine (\ref{3equation 13.2}) with Lemma~\ref{3lemm: ext6} and conclude that
$$\mathrm{Ext}^1_{\mathrm{GL}_3(\mathbb{Q}_p),\lambda}\left(W_2,~\Sigma^{\sharp,+,\flat}(\lambda,\mathscr{L}_1,\mathscr{L}_2)\right)=0$$
which contradicts (\ref{3nonvanishing of ext 13.1}). Consequently, the injection (\ref{3equation 13.1}) must be strict and we finish the proof.
\end{proof}

According to Lemma~\ref{3lemm: ext11}, the short exact sequence
$$\Sigma^{\sharp,+}(\lambda, \mathscr{L}_1, \mathscr{L}_2)\hookrightarrow\Sigma^{\sharp,+}(\lambda,\underline{\mathscr{L}})\twoheadrightarrow \overline{L}(\lambda)\otimes_E(v_{P_2}^{\infty}\oplus v_{P_1}^{\infty})$$
induces a long exact sequence:
\begin{multline}\label{3second key sequence}
\mathrm{Ext}^1_{\mathrm{GL}_3(\mathbb{Q}_p),\lambda}\left(\overline{L}(\lambda), ~\Sigma^{\sharp,+}(\lambda,\underline{\mathscr{L}})\right)\hookrightarrow\mathrm{Ext}^1_{\mathrm{GL}_3(\mathbb{Q}_p),\lambda}\left(\overline{L}(\lambda), ~\overline{L}(\lambda)\otimes_E(v_{P_2}^{\infty}\oplus v_{P_1}^{\infty})\right)\\
\xrightarrow{f}\mathrm{Ext}^2_{\mathrm{GL}_3(\mathbb{Q}_p),\lambda}\left(\overline{L}(\lambda), ~\Sigma^{\sharp,+}(\lambda,\mathscr{L}_1,\mathscr{L}_2)\right)
\end{multline}
\begin{prop}\label{3prop: restriction}
We have
$$\mathrm{dim}_E\mathrm{Ext}^1_{\mathrm{GL}_3(\mathbb{Q}_p),\lambda}(\overline{L}(\lambda), ~\Sigma^{\sharp,+,\flat}(\lambda,\underline{\mathscr{L}}))=1$$
and the image of $f$ is not contained in the image of the natural injection
$$\mathrm{Ext}^2_{\mathrm{GL}_3(\mathbb{Q}_p),\lambda}\left(\overline{L}(\lambda),~\overline{L}(\lambda)\otimes_E\mathrm{St}_3^{\infty}\right)\hookrightarrow\mathrm{Ext}^2_{\mathrm{GL}_3(\mathbb{Q}_p),\lambda}\left(\overline{L}(\lambda), ~\Sigma^{\sharp,+}(\lambda,\mathscr{L}_1,\mathscr{L}_2)\right).$$
\end{prop}
\begin{proof}
We use the shorten notation for the two dimensional space
$$M:=\mathrm{Ext}^1_{\mathrm{GL}_3(\mathbb{Q}_p),\lambda}\left(\overline{L}(\lambda), ~\overline{L}(\lambda)\otimes_E(v_{P_2}^{\infty}\oplus v_{P_1}^{\infty})\right).$$
We actually have the following commutative diagram
\begin{equation}
\begin{xy}
(0,0)*+{\mathrm{Ext}^1_{\mathrm{GL}_3(\mathbb{Q}_p),\lambda}\left(\overline{L}(\lambda), ~\Sigma^{\sharp,+}(\lambda,\underline{\mathscr{L}})\right)}="a"; (40,0)*+{M}="b"; (80,0)*+{\mathrm{Ext}^2_{\mathrm{GL}_3(\mathbb{Q}_p),\lambda}\left(\overline{L}(\lambda), ~\Sigma^{\sharp,+}(\lambda,\mathscr{L}_1,\mathscr{L}_2)\right)}="c"; (0,-15)*+{\mathrm{Ext}^1_{\mathrm{GL}_3(\mathbb{Q}_p),\lambda}\left(\overline{L}(\lambda), ~\Sigma^{\sharp,+,\flat}(\lambda,\underline{\mathscr{L}})\right)}="d"; (40,-15)*+{M}="e"; (80,-15)*+{\mathrm{Ext}^2_{\mathrm{GL}_3(\mathbb{Q}_p),\lambda}\left(\overline{L}(\lambda), ~\Sigma^{\sharp,+,\flat}(\lambda,\mathscr{L}_1,\mathscr{L}_2)\right)}="f"; (32,3)*+{i}="g"; (47,3)*+{f}="h"; (32,-12)*+{j}="i"; (47,-12)*+{g}="j"; (-3,-7.5)*+{h}="k"; (77,-7.5)*+{k}="l";
{\ar@{^{(}->}"a";"b"}; {\ar@{->}"b";"c"}; {\ar@{^{(}->}"a";"d"}; {\ar@{=}"b";"e"}; {\ar@{->}"c";"f"}; {\ar@{^{(}->}"d";"e"}; {\ar@{->}"e";"f"};
\end{xy}
\end{equation}
where the middle vertical map is just an equality. We know that $h$ is injective by the vanishing
$$\mathrm{Ext}^1_{\mathrm{GL}_3(\mathbb{Q}_p),\lambda}\left(\overline{L}(\lambda), ~\overline{L}(\lambda)\otimes_E\mathrm{St}_3^{\infty}\right)=0$$
and $k$ has a one dimensional image by (\ref{3equation 11.1}). Both $i$ and $j$ are injective due to (\ref{3vanishing of ext}) and (\ref{3vanishing of ext prime}). Therefore by a simple diagram chasing we have
\begin{multline*}
\mathrm{dim}_E\mathrm{Ext}^1_{\mathrm{GL}_3(\mathbb{Q}_p),\lambda}\left(\overline{L}(\lambda), ~\Sigma^{\sharp,+,\flat}(\lambda,\underline{\mathscr{L}})\right)\\
=\mathrm{dim}_EM-\mathrm{dim}_E\mathrm{Im}(g)\geq \mathrm{dim}_EM-\mathrm{dim}_E\mathrm{Im}(k)=2-1=1
\end{multline*}
by Lemma~\ref{3lemm: ext11} and therefore
$$\mathrm{dim}_E\mathrm{Ext}^1_{\mathrm{GL}_3(\mathbb{Q}_p),\lambda}\left(\overline{L}(\lambda), ~\Sigma^{\sharp,+,\flat}(\lambda,\underline{\mathscr{L}})\right)=1$$
by Lemma~\ref{3lemm: ext13}. Moreover, the map $g$ has a one dimensional image and hence $k\circ f$ has one dimensional image, meaning that the image of $f$ has dimension one or two and is not contained in $\mathrm{Ker}(k)$, which is exactly the image of
\begin{equation}\label{3degenerate ext2}
\mathrm{Ext}^2_{\mathrm{GL}_3(\mathbb{Q}_p),\lambda}\left(\overline{L}(\lambda), ~\overline{L}(\lambda)\otimes_E\mathrm{St}_3^{\infty}\right)\rightarrow\mathrm{Ext}^2_{\mathrm{GL}_3(\mathbb{Q}_p),\lambda}\left(\overline{L}(\lambda), ~\Sigma^{\sharp,+}(\lambda,\mathscr{L}_1,\mathscr{L}_2)\right)
\end{equation}
by (\ref{3equation 11.1}). In fact, the restriction of $f$ to the direct summand $\mathrm{Ext}^1_{\mathrm{GL}_3(\mathbb{Q}_p),\lambda}\left(\overline{L}(\lambda), ~\overline{L}(\lambda)\otimes_Ev_{P_i}^{\infty}\right)$ is given by the cup product map with a non-zero element in the line of $$\mathrm{Ext}^1_{\mathrm{GL}_3(\mathbb{Q}_p),\lambda}\left(\overline{L}(\lambda)\otimes_Ev_{P_i}^{\infty},~\Sigma^+(\lambda, \mathscr{L}_1, \mathscr{L}_2)\right)$$
given by the preimage of
$$E\left(\iota_1(D_0)+\mathscr{L}_i^{\prime}\kappa(b_{1,\mathrm{val}_p}\wedge b_{2,\mathrm{val}_p})\right)$$
via (\ref{3isomorphism for normalization}) by our definition of $\Sigma^{\sharp,+}(\lambda,\underline{\mathscr{L}})$ and it is obvious that $\iota_1(D_0)+\mathscr{L}_i^{\prime}\kappa(b_{1,\mathrm{val}_p}\wedge b_{2,\mathrm{val}_p})$ does not lie in the image of (\ref{3degenerate ext2}) which is exactly the line $E\kappa(b_{1,\mathrm{val}_p}\wedge b_{2,\mathrm{val}_p})$.
\end{proof}
\begin{prop}\label{3prop: criterion of existence}
We have
$$\mathrm{dim}_E\mathrm{Ext}^1_{\mathrm{GL}_3(\mathbb{Q}_p),\lambda}\left(\overline{L}(\lambda), ~\Sigma^{\sharp,+}(\lambda,\underline{\mathscr{L}})\right)=1$$
if and only if $\mathscr{L}_1^{\prime}=\mathscr{L}_2^{\prime}=\mathscr{L}_3$ for a certain $\mathscr{L}_3\in E$.
\end{prop}
\begin{proof}
It follows from (\ref{3second key sequence}) that
$$\mathrm{Ext}^1_{\mathrm{GL}_3(\mathbb{Q}_p),\lambda}\left(\overline{L}(\lambda), ~\Sigma^{\sharp,+}(\lambda,\underline{\mathscr{L}})\right)=1$$
if and only if the image of $f$ is one dimensional. Then we notice by the interpretation of $f$ as cup product in Proposition~\ref{3prop: restriction} that the image of
$$\mathrm{Ext}^1_{\mathrm{GL}_3(\mathbb{Q}_p),\lambda}\left(\overline{L}(\lambda), ~\overline{L}(\lambda)\otimes_Ev_{P_i}^{\infty}\right)$$
under $f$ is the line of
$$\mathrm{Ext}^2_{\mathrm{GL}_3(\mathbb{Q}_p),\lambda}\left(\overline{L}(\lambda), ~\Sigma^{\sharp,+}(\lambda,\mathscr{L}_1,\mathscr{L}_2)\right)$$
generated by
$$\iota_1(D_0)+\mathscr{L}_i^{\prime}\kappa(b_{1,\mathrm{val}_p}\wedge b_{2,\mathrm{val}_p})$$
for each $i=1,2$. Therefore the image of $f$ is one dimensional if and only if the two lines for $i=1,2$ coincide which means that $$\mathscr{L}_1^{\prime}=\mathscr{L}_2^{\prime}=\mathscr{L}_3$$
for a certain $\mathscr{L}_3\in E$.
\end{proof}
We use the notation $\Sigma^{\sharp,+}(\lambda, \mathscr{L}_1, \mathscr{L}_2, \mathscr{L}_3)$ for the representation $\Sigma^{\sharp,+}(\lambda,\underline{\mathscr{L}})$ when
$$\underline{\mathscr{L}}=(\mathscr{L}_1, \mathscr{L}_2, \mathscr{L}_3, \mathscr{L}_3).$$
We define $\Sigma^{\rm{min}}(\lambda, \mathscr{L}_1, \mathscr{L}_2, \mathscr{L}_3)$ as the unique representation (up to isomorphism) given by a non-zero element in $\mathrm{Ext}^1_{\mathrm{GL}_3(\mathbb{Q}_p),\lambda}\left(\overline{L}(\lambda), ~\Sigma^{\sharp,+}(\lambda, \mathscr{L}_1, \mathscr{L}_2, \mathscr{L}_3)\right)$ according to Proposition~\ref{3prop: criterion of existence}. Therefore by our definition $\Sigma^{\rm{min}}(\lambda, \mathscr{L}_1, \mathscr{L}_2, \mathscr{L}_3)$ has the following form
\begin{equation}\label{3naive picture for min}
\begin{xy}
(0,0)*+{\mathrm{St}_3^{\rm{an}}(\lambda)}="a"; (20,6)*+{v_{P_1}^{\rm{an}}(\lambda)}="b"; (20,-6)*+{v_{P_2}^{\rm{an}}(\lambda)}="c"; (40,8)*+{C_{s_1,s_1}}="d"; (40,-8)*+{C_{s_2,s_2}}="e"; (40,0)*+{\overline{L}(\lambda)}="f"; (60,8)*+{\overline{L}(\lambda)\otimes_Ev_{P_2}^{\infty}}="g"; (60,-8)*+{\overline{L}(\lambda)\otimes_Ev_{P_1}^{\infty}}="h"; (80,0)*+{\overline{L}(\lambda)}="i";
{\ar@{-}"a";"b"}; {\ar@{-}"a";"c"}; {\ar@{-}"b";"d"}; {\ar@{-}"c";"e"}; {\ar@{-}"b";"f"}; {\ar@{-}"c";"f"}; {\ar@{-}"d";"g"}; {\ar@{-}"e";"h"}; {\ar@{-}"g";"i"}; {\ar@{-}"h";"i"}; {\ar@{--}"b";"i"}; {\ar@{--}"c";"i"};
\end{xy}.
\end{equation}
It follows from Proposition~\ref{3prop: locally algebraic extension}, Proposition~\ref{3prop: criterion of existence}, the definition of $\Sigma^{\rm{min}}(\lambda, \mathscr{L}_1, \mathscr{L}_2, \mathscr{L}_3)$ and an easy devissage that
\begin{equation}\label{3vanishing of ext1 for min}
\mathrm{Ext}^1_{\mathrm{GL}_3(\mathbb{Q}_p),\lambda}\left(\overline{L}(\lambda), ~\Sigma^{\rm{min}}(\lambda, \mathscr{L}_1, \mathscr{L}_2, \mathscr{L}_3)\right)=0.
\end{equation}
\begin{rema}\label{3rema: noncanonical}
The definition of the invariant $\mathscr{L}_3\in E$ of $\Sigma^{\rm{min}}(\lambda, \mathscr{L}_1, \mathscr{L}_2, \mathscr{L}_3)$ obviously relies on the choice of a special $p$\text{-}adic dilogarithm function $D_0$ which is non-canonical. This is similar to the definition of the invariants $\mathscr{L}_1, \mathscr{L}_2\in E$ which relies on the choice of a special $p$\text{-}adic logarithm function $\mathrm{log}_0$.
\end{rema}
\begin{lemm}\label{3lemm: preparation for global}
We have
$$\mathrm{dim}_E\mathrm{Ext}^1_{\mathrm{GL}_3(\mathbb{Q}_p),\lambda}\left(W_0, ~\Sigma^{\sharp,+}(\lambda, \mathscr{L}_1, \mathscr{L}_2)\right)=2.$$
Moreover, if $V$ is a locally analytic representation determined by a line
$$M_V\subsetneq \mathrm{Ext}^1_{\mathrm{GL}_3(\mathbb{Q}_p),\lambda}\left(W_0, ~\Sigma^{\sharp,+}(\lambda, \mathscr{L}_1, \mathscr{L}_2)\right)$$
satisfying
$$M_V\neq \mathrm{Ext}^1_{\mathrm{GL}_3(\mathbb{Q}_p),\lambda}\left(W_0, ~\overline{L}(\lambda)\otimes_E\mathrm{St}_3^{\infty}\right)\hookrightarrow \mathrm{Ext}^1_{\mathrm{GL}_3(\mathbb{Q}_p),\lambda}\left(W_0, ~\Sigma^{\sharp,+}(\lambda, \mathscr{L}_1, \mathscr{L}_2)\right),$$
then there exists a unique $\mathscr{L}_3\in E$ such that
$$V\cong \Sigma^{\rm{min}}(\lambda, \mathscr{L}_1, \mathscr{L}_2, \mathscr{L}_3).$$
\end{lemm}
\begin{proof}
The short exact sequence
$$\overline{L}(\lambda)\otimes_E\left(v_{P_1}^{\infty}\oplus v_{P_2}^{\infty}\right)\hookrightarrow W_0\twoheadrightarrow \overline{L}(\lambda)$$
together with Lemma~\ref{3lemm: ext10} induce a commutative diagram
\begin{equation}
\begin{xy}
(0,0)*+{\mathrm{Ext}^1_{\mathrm{GL}_3(\mathbb{Q}_p),\lambda}\left(W_0, ~V^+\right)}="a"; (55,0)*+{\mathrm{Ext}^1_{\mathrm{GL}_3(\mathbb{Q}_p),\lambda}\left(V_1^{\rm{alg}}\oplus V_2^{\rm{alg}}, ~V^+\right)}="b"; (110,0)*+{\mathrm{Ext}^2_{\mathrm{GL}_3(\mathbb{Q}_p),\lambda}\left(\overline{L}(\lambda), ~V^+\right)}="e";
(0,-12)*+{\mathrm{Ext}^1_{\mathrm{GL}_3(\mathbb{Q}_p),\lambda}\left(W_0, ~V^{\sharp, +}\right)}="c"; (55,-12)*+{\mathrm{Ext}^1_{\mathrm{GL}_3(\mathbb{Q}_p),\lambda}\left(V_1^{\rm{alg}}\oplus V_2^{\rm{alg}}, ~V^{\sharp, +}\right)}="d"; (110,-12)*+{\mathrm{Ext}^2_{\mathrm{GL}_3(\mathbb{Q}_p),\lambda}\left(\overline{L}(\lambda), ~V^{\sharp, +}\right)}="f"; (23,2)*+{g_1}; (23,-10)*+{g_2}; (-3,-6)*+{h_1}; (52,-6)*+{h_2}; (107,-6)*+{h_3}; (85,2)*+{k_1}; (85,-10)*+{k_2};
{\ar@{->}"a";"b"}; {\ar@{->}"a";"c"}; {\ar@{^{(}->}"c";"d"}; {\ar@{^{(}->}"b";"d"}; {\ar@{->>}"b";"e"}; {\ar@{->}"d";"f"}; {\ar@{->}"e";"f"};
\end{xy}
\end{equation}
where we use shorten notation $V_i^{\rm{alg}}$ for $\overline{L}(\lambda)\otimes_Ev_{P_i}^{\infty}$, $V^+$ for $\Sigma^+(\lambda, \mathscr{L}_1, \mathscr{L}_2)$ and $V^{\sharp, +}$ for $\Sigma^{\sharp,+}(\lambda, \mathscr{L}_1, \mathscr{L}_2)$ to save space. We observe that $g_2$ is an injection due to Lemma~\ref{3lemm: ext10}, $k_1$ is a surjection by the proof of Proposition~\ref{3prop: main dim}, $h_3$ is an isomorphism by Proposition~\ref{3prop: locally algebraic extension} and an easy devissage and finally $h_2$ is an injection. Assume that $h_2$ is not surjective, then any representation given by a non-zero element in $\mathrm{Coker}(h_2)$ admits a quotient of the form
\begin{equation}\label{3special uniserial}
\begin{xy}
(0,0)*+{C^1_{s_i,1}}="a"; (16,0)*+{\overline{L}(\lambda)}="b"; (32,0)*+{V_i^{\rm{alg}}}="c";
{\ar@{-}"a";"b"}; {\ar@{-}"b";"c"};
\end{xy}
\end{equation}
for $i=1$ or $2$ due to Lemma~\ref{3lemm: nonvanishing ext1}. However, it follows from Lemma~\ref{3lemm: special vanishing 1} that there is no uniserial representation of the form (\ref{3special uniserial}), which implies that $h_2$ is indeed an isomorphism, and hence $k_2$ is surjective by a diagram chasing. Therefore we conclude that
\begin{multline*}
\mathrm{dim}_E\mathrm{Ext}^1_{\mathrm{GL}_3(\mathbb{Q}_p),\lambda}\left(W_0, ~V^{\sharp, +}\right)\\
=\mathrm{dim}_E\mathrm{Ext}^1_{\mathrm{GL}_3(\mathbb{Q}_p),\lambda}\left(V_1^{\rm{alg}}\oplus V_2^{\rm{alg}}, ~V^{\sharp, +}\right)-\mathrm{dim}_E\mathrm{Ext}^2_{\mathrm{GL}_3(\mathbb{Q}_p),\lambda}\left(\overline{L}(\lambda), ~V^{\sharp, +}\right)\\
=\mathrm{dim}_E\mathrm{Ext}^1_{\mathrm{GL}_3(\mathbb{Q}_p),\lambda}\left(V_1^{\rm{alg}}\oplus V_2^{\rm{alg}}, ~V^+\right)-\mathrm{dim}_E\mathrm{Ext}^2_{\mathrm{GL}_3(\mathbb{Q}_p),\lambda}\left(\overline{L}(\lambda), ~V^+\right)=4-2=2.
\end{multline*}
The final claim on the existence of a unique $\mathscr{L}_3$ follows from Proposition~\ref{3prop: criterion of existence}, our definition of $\Sigma^{\rm{min}}(\lambda, \mathscr{L}_1, \mathscr{L}_2, \mathscr{L}_3)$ and the observation that the restriction of $k_2$ to the direct summand $$\mathrm{Ext}^1_{\mathrm{GL}_3(\mathbb{Q}_p),\lambda}\left(V_i^{\rm{alg}}, ~V^{\sharp, +}\right)$$
induces isomorphisms
$$\mathrm{Ext}^1_{\mathrm{GL}_3(\mathbb{Q}_p),\lambda}\left(V_i^{\rm{alg}}, ~V^{\sharp, +}\right)\xrightarrow{\sim}\mathrm{Ext}^2_{\mathrm{GL}_3(\mathbb{Q}_p),\lambda}\left(\overline{L}(\lambda), ~V^{\sharp, +}\right)$$
which can be interpreted as the cup product morphism with the one dimensional space
$$\mathrm{Ext}^1_{\mathrm{GL}_3(\mathbb{Q}_p),\lambda}\left(\overline{L}(\lambda),~V_i^{\rm{alg}}\right)$$
for $i=1,2$.
\end{proof}
We define $\Sigma^{\sharp, +}_i(\lambda, \mathscr{L}_1, \mathscr{L}_2, \mathscr{L}_3)$ as the subrepresentation of $\Sigma^{\sharp, +}(\lambda, \mathscr{L}_1, \mathscr{L}_2, \mathscr{L}_3)$ that fits into the short exact sequence
$$\Sigma^{\sharp, +}_i(\lambda, \mathscr{L}_1, \mathscr{L}_2, \mathscr{L}_3)\hookrightarrow\Sigma^{\sharp, +}(\lambda, \mathscr{L}_1, \mathscr{L}_2, \mathscr{L}_3)\twoheadrightarrow \overline{L}(\lambda)\otimes_Ev_{P_i}^{\infty}$$
for each $i=1,2$. We use the notation $\mathcal{D}_i(\lambda, \mathscr{L}_1, \mathscr{L}_2, \mathscr{L}_3)^{\prime}$ for the object in the derived category $\mathcal{D}^b\left(\mathrm{Mod}_{D(\mathrm{GL}_3(\mathbb{Q}_p), E)}\right)$ associated with the complex
$$\left[W_{3-i}^{\prime}\longrightarrow\Sigma^{\sharp, +}_i(\lambda, \mathscr{L}_1, \mathscr{L}_2, \mathscr{L}_3)^{\prime}\right].$$
\begin{prop}\label{3prop: relation with derived object}
The object
$$\mathcal{D}_i(\lambda, \mathscr{L}_1, \mathscr{L}_2, \mathscr{L}_3)^{\prime}\in \mathcal{D}^b\left(\mathrm{Mod}_{D(\mathrm{GL}_3(\mathbb{Q}_p), E)}\right)$$
fits into the distinguished triangle
\begin{equation}\label{3distinguished triangle}
\overline{L}(\lambda)^{\prime}\longrightarrow \mathcal{D}_i(\lambda, \mathscr{L}_1, \mathscr{L}_2, \mathscr{L}_3)^{\prime}\longrightarrow \Sigma^{\sharp, +}(\lambda, \mathscr{L}_1, \mathscr{L}_2)^{\prime}[-1]\xrightarrow{~+1}
\end{equation}
for each $i=1,2$. Moreover, the element in
\begin{multline}\label{3canonical translation}
\mathrm{Ext}^2_{\mathrm{GL}_3(\mathbb{Q}_p),\lambda}\left(\overline{L}(\lambda),~\Sigma(\lambda, \mathscr{L}_1, \mathscr{L}_2)\right)\\
\xrightarrow{\sim}\mathrm{Ext}^2_{\mathrm{GL}_3(\mathbb{Q}_p),\lambda}\left(\overline{L}(\lambda),~\Sigma^{\sharp, +}(\lambda, \mathscr{L}_1, \mathscr{L}_2)\right)\\
\cong \mathrm{Hom}_{\mathcal{D}^b\left(\mathrm{Mod}_{D(\mathrm{GL}_3(\mathbb{Q}_p), E)}\right)}\left(\Sigma^{\sharp, +}(\lambda, \mathscr{L}_1, \mathscr{L}_2)^{\prime}[-2],~\overline{L}(\lambda)^{\prime}\right)
\end{multline}
associated with the distinguished triangle (\ref{3distinguished triangle}) is
\begin{equation}\label{3explicit element}
\iota_1(D_0)+\mathscr{L}_3\kappa(b_{1,\mathrm{val}_p}\wedge b_{2,\mathrm{val}_p}).
\end{equation}
\end{prop}
\begin{proof}
It follows from Proposition~3.2 of \cite{Schr11} that there is a unique (up to isomorphism) object
$$\mathcal{D}(\lambda, \mathscr{L}_1, \mathscr{L}_2, \mathscr{L}_3)^{\prime}\in \mathcal{D}^b\left(\mathrm{Mod}_{D(\mathrm{GL}_3(\mathbb{Q}_p), E)}\right)$$
that fits into a distinguished triangle
\begin{equation}\label{3formal triangle}
\overline{L}(\lambda)^{\prime}\longrightarrow \mathcal{D}(\lambda, \mathscr{L}_1, \mathscr{L}_2, \mathscr{L}_3)\longrightarrow \Sigma^{\sharp, +}(\lambda, \mathscr{L}_1, \mathscr{L}_2)[-1]\xrightarrow{~+1}
\end{equation}
such that the element in $\mathrm{Ext}^2_{\mathrm{GL}_3(\mathbb{Q}_p),\lambda}\left(\overline{L}(\lambda),~\Sigma(\lambda, \mathscr{L}_1, \mathscr{L}_2)\right)$ associated with (\ref{3formal triangle}) via (\ref{3canonical translation}) is (\ref{3explicit element}). It follows from $\mathbf{TR 2}$ ( cf. Section~10.2.1 of \cite{Wei94}) that
\begin{equation}\label{3formal triangle 2}
\mathcal{D}(\lambda, \mathscr{L}_1, \mathscr{L}_2, \mathscr{L}_3)^{\prime}\longrightarrow\Sigma^{\sharp, +}(\lambda, \mathscr{L}_1, \mathscr{L}_2)^{\prime}[-1]\longrightarrow\overline{L}(\lambda)^{\prime}[1]\xrightarrow{~+1}
\end{equation}
is another distinguished triangle. The isomorphism (\ref{3isomorphism for normalization}) can be reinterpreted as the isomorphism
\begin{multline}
\mathrm{Hom}_{\mathcal{D}^b\left(\mathrm{Mod}_{D(\mathrm{GL}_3(\mathbb{Q}_p), E)}\right)}\left(\Sigma^{\sharp, +}(\lambda, \mathscr{L}_1, \mathscr{L}_2)^{\prime}[-1],~\left(\overline{L}(\lambda)\otimes_Ev_{P_{3-i}}^{\infty}\right)^{\prime}\right)\\
\xrightarrow{\sim}\mathrm{Hom}_{\mathcal{D}^b\left(\mathrm{Mod}_{D(\mathrm{GL}_3(\mathbb{Q}_p), E)}\right)}\left(\Sigma^{\sharp, +}(\lambda, \mathscr{L}_1, \mathscr{L}_2)^{\prime}[-1],~\overline{L}(\lambda)^{\prime}[1]\right)
\end{multline}
induced by the composition with $\mathrm{Hom}_{\mathcal{D}^b\left(\mathrm{Mod}_{D(\mathrm{GL}_3(\mathbb{Q}_p), E)}\right)}\left(\left(\overline{L}(\lambda)\otimes_Ev_{P_{3-i}}^{\infty}\right)^{\prime},~\overline{L}(\lambda)^{\prime}[1]\right)$. As a result, each morphism
$$\Sigma^{\sharp, +}(\lambda, \mathscr{L}_1, \mathscr{L}_2)^{\prime}[-1]\rightarrow\overline{L}(\lambda)^{\prime}[1]$$
uniquely factors through a composition
$$\Sigma^{\sharp, +}(\lambda, \mathscr{L}_1, \mathscr{L}_2)^{\prime}[-1]\rightarrow\left(\overline{L}(\lambda)\otimes_Ev_{P_{3-i}}^{\infty}\right)^{\prime}\rightarrow\overline{L}(\lambda)^{\prime}[1] $$
which induces a commutative diagram with four distinguished triangles
\begin{equation}
\begin{xy}
(0,0)*+{\Sigma^{\sharp, +}(\lambda, \mathscr{L}_1, \mathscr{L}_2)^{\prime}[-1]}="a"; (20,30)*+{\left(\overline{L}(\lambda)\otimes_Ev_{P_{3-i}}^{\infty}\right)^{\prime}}="b"; (40,20)*+{\overline{L}(\lambda)^{\prime}[1]}="c"; (40,60)*+{\Sigma^{\sharp, +}_i(\lambda, \mathscr{L}_1, \mathscr{L}_2, \mathscr{L}_3)^{\prime}}="d"; (60,30)*+{\mathcal{D}(\lambda, \mathscr{L}_1, \mathscr{L}_2, \mathscr{L}_3)^{\prime}}="e"; (80,0)*+{W_{3-i}^{\prime}[1]}="f"; (50,75)*+{}="g"; (80,40)*+{}="h"; (100,-10)*+{}="i"; (95,-22.5)*+{}="j"; (42.5,70)*+{+1}="k"; (70,40)*+{+1}="l"; (92,-2)*+{+1}="m"; (90,-10)*+{+1}="n";
{\ar@{->}"a";"b"}; {\ar@{->}"a";"c"}; {\ar@{->}"b";"c"}; {\ar@{->}"b";"d"}; {\ar@{->}"c";"e"}; {\ar@{->}"c";"f"}; {\ar@{->}"d";"e"}; {\ar@{->}"e";"f"}; {\ar@{->}"d";"g"}; {\ar@{->}"e";"h"}; {\ar@{->}"f";"i"}; {\ar@{->}"f";"j"};
\end{xy}
\end{equation}
by $\mathbf{TR 4}$. Hence we deduce that
$$\Sigma^{\sharp, +}_i(\lambda, \mathscr{L}_1, \mathscr{L}_2, \mathscr{L}_3)^{\prime}\longrightarrow \mathcal{D}(\lambda, \mathscr{L}_1, \mathscr{L}_2, \mathscr{L}_3)^{\prime}\longrightarrow W_{3-i}^{\prime}[1]\xrightarrow{~+1}$$
or equivalently
$$W_{3-i}^{\prime}\longrightarrow\Sigma^{\sharp, +}_i(\lambda, \mathscr{L}_1, \mathscr{L}_2, \mathscr{L}_3)^{\prime}\longrightarrow \mathcal{D}(\lambda, \mathscr{L}_1, \mathscr{L}_2, \mathscr{L}_3)^{\prime}\xrightarrow{~+1}$$
is a distinguished triangle. On the other hand, it is easy to see that $\mathcal{D}_i(\lambda, \mathscr{L}_1, \mathscr{L}_2, \mathscr{L}_3)^{\prime}$ fits into the distinguished triangle
$$W_{3-i}^{\prime}\longrightarrow\Sigma^{\sharp, +}_i(\lambda, \mathscr{L}_1, \mathscr{L}_2, \mathscr{L}_3)^{\prime}\longrightarrow \mathcal{D}_i(\lambda, \mathscr{L}_1, \mathscr{L}_2, \mathscr{L}_3)^{\prime}\xrightarrow{~+1}$$
and thus we conclude that
$$\mathcal{D}_i(\lambda, \mathscr{L}_1, \mathscr{L}_2, \mathscr{L}_3)^{\prime}\cong\mathcal{D}(\lambda, \mathscr{L}_1, \mathscr{L}_2, \mathscr{L}_3)^{\prime}\in\mathcal{D}^b\left(\mathrm{Mod}_{D(\mathrm{GL}_3(\mathbb{Q}_p), E)}\right)$$
by the uniqueness in Proposition~3.2 of \cite{Schr11}. Hence we finish the proof.
\end{proof}
We define $\Sigma^{\rm{min},-}(\lambda, \mathscr{L}_1, \mathscr{L}_2, \mathscr{L}_3)$ as the unique subrepresentation of $\Sigma^{\rm{min}}(\lambda, \mathscr{L}_1, \mathscr{L}_2, \mathscr{L}_3)$ of the form
$$\begin{xy}
(0,0)*+{\mathrm{St}_3^{\rm{an}}(\lambda)}="a"; (20,4)*+{v_{P_1}^{\rm{an}}(\lambda)}="b"; (20,-4)*+{v_{P_2}^{\rm{an}}(\lambda)}="c"; (40,4)*+{C_{s_1,s_1}}="d"; (40,-4)*+{C_{s_2,s_2}}="e"; (60,4)*+{\overline{L}(\lambda)\otimes_Ev_{P_2}^{\infty}}="g"; (60,-4)*+{\overline{L}(\lambda)\otimes_Ev_{P_1}^{\infty}}="h";
{\ar@{-}"a";"b"}; {\ar@{-}"a";"c"}; {\ar@{-}"b";"d"}; {\ar@{-}"c";"e"}; {\ar@{-}"d";"g"}; {\ar@{-}"e";"h"};
\end{xy}$$
that fits into the short exact sequence
\begin{equation}\label{3short exact sequence definition}
\Sigma^{\rm{min},-}(\lambda, \mathscr{L}_1, \mathscr{L}_2, \mathscr{L}_3)\hookrightarrow\Sigma^{\rm{min}}(\lambda, \mathscr{L}_1, \mathscr{L}_2, \mathscr{L}_3)\twoheadrightarrow\overline{L}(\lambda)^{\oplus2}
\end{equation}
and $\Sigma^{\rm{min},--}(\lambda, \mathscr{L}_1, \mathscr{L}_2, \mathscr{L}_3)$ as the unique subrepresentation of $\Sigma^{\rm{min},-}(\lambda, \mathscr{L}_1, \mathscr{L}_2, \mathscr{L}_3)$ of the form
$$
\begin{xy}
(0,0)*+{\mathrm{St}_3^{\rm{an}}(\lambda)}="a"; (30,4)*+{\overline{L}(\lambda)\otimes_Ev_{P_1}^{\infty}}="b"; (30,-4)*+{\overline{L}(\lambda)\otimes_Ev_{P_2}^{\infty}}="c"; (60,4)*+{C_{s_1,s_1}}="d"; (60,-4)*+{C_{s_2,s_2}}="e";
{\ar@{-}"a";"b"}; {\ar@{-}"a";"c"}; {\ar@{-}"b";"d"}; {\ar@{-}"c";"e"};
\end{xy}
$$
that fits into the short exact sequence
\begin{equation}\label{3short exact sequence for min}
\Sigma^{\rm{min},--}(\lambda, \mathscr{L}_1, \mathscr{L}_2, \mathscr{L}_3)\hookrightarrow\Sigma^{\rm{min},-}(\lambda, \mathscr{L}_1, \mathscr{L}_2, \mathscr{L}_3)\twoheadrightarrow \left(\overline{L}(\lambda)\otimes_Ev_{P_1}^{\infty}\right)\oplus \left(\overline{L}(\lambda)\otimes_Ev_{P_2}^{\infty}\right)\oplus C^1_{s_2,1}\oplus C^1_{s_1,1}.
\end{equation}
The short exact sequence (\ref{3short exact sequence definition}) induces a long exact sequence
\begin{multline*}
\mathrm{Hom}_{\mathrm{GL}_3(\mathbb{Q}_p),\lambda}\left(\overline{L}(\lambda),~\overline{L}(\lambda)^{\oplus 2}\right)\hookrightarrow\mathrm{Ext}^1_{\mathrm{GL}_3(\mathbb{Q}_p),\lambda}\left(\overline{L}(\lambda),~\Sigma^{\rm{min},-}(\lambda, \mathscr{L}_1, \mathscr{L}_2, \mathscr{L}_3)\right)\\
\rightarrow\mathrm{Ext}^1_{\mathrm{GL}_3(\mathbb{Q}_p),\lambda}\left(\overline{L}(\lambda),~\Sigma^{\rm{min}}(\lambda, \mathscr{L}_1, \mathscr{L}_2, \mathscr{L}_3)\right)\rightarrow\mathrm{Ext}^1_{\mathrm{GL}_3(\mathbb{Q}_p),\lambda}\left(\overline{L}(\lambda),~\overline{L}(\lambda)^{\oplus 2}\right)
\end{multline*}
which easily implies that
$$\mathrm{dim}_E\mathrm{Ext}^1_{\mathrm{GL}_3(\mathbb{Q}_p),\lambda}\left(\overline{L}(\lambda),~\Sigma^{\rm{min},-}(\lambda, \mathscr{L}_1, \mathscr{L}_2, \mathscr{L}_3)\right)=2$$
by Proposition~\ref{3prop: locally algebraic extension} and (\ref{3vanishing of ext1 for min}). On the other hand, we notice that $\Sigma^{\rm{min},--}(\lambda, \mathscr{L}_1, \mathscr{L}_2, \mathscr{L}_3)$ admits a filtration whose only reducible graded piece is
$$\begin{xy}
(0,0)*+{C^1_{s_i,1}}="a"; (20,0)*+{\overline{L}(\lambda)\otimes_Ev_{P_i}^{\infty}}="b";
{\ar@{-}"a";"b"};
\end{xy}$$
and
$$\mathrm{Ext}^1_{\mathrm{GL}_3(\mathbb{Q}_p),\lambda}\left(\overline{L}(\lambda),~V\right)=0$$
for all graded pieces $V$ of such a filtration by Lemma~\ref{3lemm: nonvanishing ext1} and Lemma~\ref{3lemm: special vanishing 1}, which implies that
$$\mathrm{Ext}^1_{\mathrm{GL}_3(\mathbb{Q}_p),\lambda}\left(\overline{L}(\lambda),~\Sigma^{\rm{min},--}(\lambda, \mathscr{L}_1, \mathscr{L}_2, \mathscr{L}_3)\right)=0.$$
Therefore (\ref{3short exact sequence for min}) induces an injection of a two dimensional space into a four dimensional space
\begin{multline}
M^{\rm{min}}:=\mathrm{Ext}^1_{\mathrm{GL}_3(\mathbb{Q}_p),\lambda}\left(\overline{L}(\lambda),~\Sigma^{\rm{min},-}(\lambda, \mathscr{L}_1, \mathscr{L}_2, \mathscr{L}_3)\right)\\
\hookrightarrow M^+:=\mathrm{Ext}^1_{\mathrm{GL}_3(\mathbb{Q}_p),\lambda}\left(\overline{L}(\lambda),~\left(\overline{L}(\lambda)\otimes_Ev_{P_1}^{\infty}\right)\oplus \left(\overline{L}(\lambda)\otimes_Ev_{P_2}^{\infty}\right)\oplus C^1_{s_2,1}\oplus C^1_{s_1,1}\right).
\end{multline}
It follows from the definition of $\Sigma^{\rm{min},-}(\lambda, \mathscr{L}_1, \mathscr{L}_2, \mathscr{L}_3)$ that we have embeddings
$$\Sigma(\lambda, \mathscr{L}_1, \mathscr{L}_2)\hookrightarrow\Sigma^+(\lambda, \mathscr{L}_1, \mathscr{L}_2)\hookrightarrow\Sigma^{\rm{min},-}(\lambda, \mathscr{L}_1, \mathscr{L}_2, \mathscr{L}_3)$$
which allow us to identify
$$M^-:=\mathrm{Ext}^1_{\mathrm{GL}_3(\mathbb{Q}_p),\lambda}\left(\overline{L}(\lambda),~\Sigma(\lambda, \mathscr{L}_1, \mathscr{L}_2)\right)$$
with a line in $M^{\rm{min}}$. We use the number $1,2,3,4$ to index the four representations $\overline{L}(\lambda)\otimes_Ev_{P_1}^{\infty}$, $\overline{L}(\lambda)\otimes_Ev_{P_2}^{\infty}$, $C^1_{s_2,1}$ and $C^1_{s_1,1}$ respectively, and we use the notation $M_I$ for each subset $I\subseteq \{1,2,3,4\}$ to denote the corresponding subspace of $M^+$ with dimension the cardinality of $I$. For example, $M_{\{1,2\}}$ denotes the two dimensional subspace $$\mathrm{Ext}^1_{\mathrm{GL}_3(\mathbb{Q}_p),\lambda}\left(\overline{L}(\lambda),~\left(\overline{L}(\lambda)\otimes_Ev_{P_1}^{\infty}\right)\oplus \left(\overline{L}(\lambda)\otimes_Ev_{P_2}^{\infty}\right)\right)$$
of $M^+$.
\begin{lemm}\label{3lemm: distinguish mult two}
We have the following characterizations of $M^{\rm{min}}$ inside $M^+$:
$$M^{\rm{min}}\cap M_{\{i,j\}}=0\mbox{ for }\{i,j\}\neq \{3,4\},$$
$$M^{\rm{min}}\cap M_{\{1,3,4\}}=M^{\rm{min}}\cap M_{\{2,3,4\}}=M^{\rm{min}}\cap M_{\{3,4\}}=M^-,$$
and
$$M^{\rm{min}}=(M^{\rm{min}}\cap M_{\{1,2,3\}})\oplus (M^{\rm{min}}\cap M_{\{1,2,4\}}).$$
\end{lemm}
\begin{proof}
As $C^1_{s_1,1}$ and $C^1_{s_2,1}$ are in the cosocle of $\Sigma(\lambda, \mathscr{L}_1, \mathscr{L}_2)$, it is immediate that
$$M^-\subseteq M_{\{3,4\}}.$$
It follows from (\ref{3naive picture for min}) that
$$M^{\rm{min}}\not\subseteq M_{\{3,4\}}$$
and thus $M^{\rm{min}}\cap M_{\{3,4\}}$ is one dimensional which must coincide with $M^-$. The proof of Lemma~\ref{3lemm: upper bound} implies that $M\not\subseteq  M_{\{i,3,4\}}$ for $i=1,2$ and therefore $M\cap M_{\{i,3,4\}}$ is one dimensional, which implies that
$$M^{\rm{min}}\cap M_{\{i,3,4\}}=M^-$$
by the inclusion
$$M^{\rm{min}}\cap M_{\{3,4\}}\subseteq M^{\rm{min}}\cap M_{\{i,3,4\}}$$
for $i=1,2$. We observe ( cf. Lemma~\ref{3lemm: ext6}) that
$$M^-\cap M_{\{3\}}=M^-\cap M_{\{4\}}=0$$
and thus
$$M^{\rm{min}}\cap M_{\{i,j\}}=M^-\cap M_{\{i,j\}}=0$$
for each $\{i,j\}\neq \{3,4\}, \{1,2\}$. We define $\Sigma^{\rm{min},-,\prime}(\lambda, \mathscr{L}_1, \mathscr{L}_2, \mathscr{L}_3)$ as the unique subrepresentation of $\Sigma^{\rm{min},-}(\lambda, \mathscr{L}_1, \mathscr{L}_2, \mathscr{L}_3)$ that fits into the short exact sequence
$$\Sigma^{\rm{min},-,\prime}(\lambda, \mathscr{L}_1, \mathscr{L}_2, \mathscr{L}_3)\hookrightarrow\Sigma^{\rm{min},-}(\lambda, \mathscr{L}_1, \mathscr{L}_2, \mathscr{L}_3)\twoheadrightarrow C^1_{s_1,1}\oplus C^1_{s_2,1}\oplus C_{s_1s_2s_1,1}$$
and then define
$$\Sigma^{\rm{min},-,\prime,\flat}(\lambda, \mathscr{L}_1, \mathscr{L}_2, \mathscr{L}_3):=\Sigma^{\rm{min},-,\prime}(\lambda, \mathscr{L}_1, \mathscr{L}_2, \mathscr{L}_3)/\overline{L}(\lambda)\otimes_E\mathrm{St}_3^{\infty}.$$
It is obvious that $M^{\rm{min}}\cap M_{\{1,2\}}\neq 0$ if and only if
$$
\mathrm{Ext}^1_{\mathrm{GL}_3(\mathbb{Q}_p),\lambda}\left(\overline{L}(\lambda),~\Sigma^{\rm{min},-,\prime}(\lambda, \mathscr{L}_1, \mathscr{L}_2, \mathscr{L}_3)\right)\neq 0
$$
which implies that
\begin{equation}\label{3nonvanishing contradiction}
\mathrm{Ext}^1_{\mathrm{GL}_3(\mathbb{Q}_p),\lambda}\left(\overline{L}(\lambda),~\Sigma^{\rm{min},-,\prime,\flat}(\lambda, \mathscr{L}_1, \mathscr{L}_2, \mathscr{L}_3)\right)\neq 0
\end{equation}
as
$$\mathrm{Ext}^1_{\mathrm{GL}_3(\mathbb{Q}_p),\lambda}\left(\overline{L}(\lambda),~\overline{L}(\lambda)\otimes_E\mathrm{St}_3^{\infty}\right)=0$$
due to Proposition~\ref{3prop: locally algebraic extension}. We notice that we have a direct sum decomposition
$$\Sigma^{\rm{min},-,\prime,\flat}(\lambda, \mathscr{L}_1, \mathscr{L}_2, \mathscr{L}_3)=V_1\oplus V_2$$
where $V_i$ is a representation of the form
$$\begin{xy}
(0,0)*+{C^2_{s_i,1}}="a"; (30,4)*+{C^1_{s_{3-i}s_i,1}}="b"; (60,8)*+{C^2_{s_{3-i}s_i,1}}="c";  (30,-4)*+{\overline{L}(\lambda)\otimes_Ev_{P_i}^{\infty}}="d"; (60,0)*+{C_{s_i,s_i}}="e"; (90,4)*+{\overline{L}(\lambda)\otimes_Ev_{P_{3-i}}^{\infty}}="f";
{\ar@{-}"a";"b"}; {\ar@{-}"a";"d"}; {\ar@{-}"b";"c"}; {\ar@{-}"b";"e"}; {\ar@{-}"d";"e"}; {\ar@{-}"e";"f"};
\end{xy}.$$
Switching $V_1$ and $V_2$ if necessary, we can assume by (\ref{3nonvanishing contradiction}) that
$$\mathrm{Ext}^1_{\mathrm{GL}_3(\mathbb{Q}_p),\lambda}\left(\overline{L}(\lambda),~V_1\right)\neq 0.$$
On the other hand, we have an embedding
$$V_1\hookrightarrow
\begin{xy}
(0,0)*+{\Sigma_1^{+,\flat}(\lambda, \mathscr{L}_1)}="a"; (25,0)*+{\overline{L}(\lambda)\otimes_Ev_{P_2}^{\infty}}="b";
{\ar@{-}"a";"b"};
\end{xy}$$
which induces an embedding
$$\mathrm{Ext}^1_{\mathrm{GL}_3(\mathbb{Q}_p),\lambda}\left(\overline{L}(\lambda),~V_1\right)\hookrightarrow \mathrm{Ext}^1_{\mathrm{GL}_3(\mathbb{Q}_p),\lambda}\left(\overline{L}(\lambda),~\begin{xy}
(0,0)*+{\Sigma_1^{+,\flat}(\lambda, \mathscr{L}_1)}="a"; (25,0)*+{\overline{L}(\lambda)\otimes_Ev_{P_2}^{\infty}}="b";
{\ar@{-}"a";"b"};
\end{xy}\right)$$
and in particular
\begin{equation}\label{3nonvanishing contradiction prime}
\mathrm{Ext}^1_{\mathrm{GL}_3(\mathbb{Q}_p),\lambda}\left(\overline{L}(\lambda),~\begin{xy}
(0,0)*+{\Sigma_1^{+,\flat}(\lambda, \mathscr{L}_1)}="a"; (25,0)*+{\overline{L}(\lambda)\otimes_Ev_{P_2}^{\infty}}="b";
{\ar@{-}"a";"b"};
\end{xy}\right)\neq 0.
\end{equation}
The short exact sequences
$$\overline{L}(\lambda)\otimes_E\mathrm{St}_3^{\infty}\hookrightarrow\Sigma_1(\lambda, \mathscr{L}_1)\twoheadrightarrow\Sigma_1^{\flat}(\lambda, \mathscr{L}_1),~\overline{L}(\lambda)\otimes_E\mathrm{St}_3^{\infty}\hookrightarrow\Sigma_1^+(\lambda, \mathscr{L}_1)\twoheadrightarrow\Sigma_1^{+,\flat}(\lambda, \mathscr{L}_1)$$
induce isomorphisms
\begin{equation}\label{3isomorphism from flat}
\begin{xy}
(0,0)*+{\mathrm{Ext}^1_{\mathrm{GL}_3(\mathbb{Q}_p),\lambda}\left(\overline{L}(\lambda),~\Sigma_1(\lambda, \mathscr{L}_1)\right)}; (30,0)*+{\xrightarrow{\sim}}; (60,0)*+{\mathrm{Ext}^1_{\mathrm{GL}_3(\mathbb{Q}_p),\lambda}\left(\overline{L}(\lambda),~\Sigma_1^{\flat}(\lambda, \mathscr{L}_1)\right)}; (0.4,-6)*+{\mathrm{Ext}^1_{\mathrm{GL}_3(\mathbb{Q}_p),\lambda}\left(\overline{L}(\lambda),~\Sigma_1^+(\lambda, \mathscr{L}_1)\right)}; (30,-6)*+{\xrightarrow{\sim}}; (62,-6)*+{\mathrm{Ext}^1_{\mathrm{GL}_3(\mathbb{Q}_p),\lambda}\left(\overline{L}(\lambda),~\Sigma_1^{+,\flat}(\lambda, \mathscr{L}_1)\right)};
\end{xy}
\end{equation}
by Lemma~\ref{3lemm: vanishing locally algebraic}. Hence we deduce that
\begin{equation}\label{3vanishing for flat}
\mathrm{Ext}^1_{\mathrm{GL}_3(\mathbb{Q}_p),\lambda}\left(W_2,~\Sigma_1^{\flat}(\lambda, \mathscr{L}_1)\right)=\mathrm{Ext}^1_{\mathrm{GL}_3(\mathbb{Q}_p),\lambda}\left(W_2,~\Sigma_1^{+,\flat}(\lambda, \mathscr{L}_1)\right)=0
\end{equation}
from Lemma~\ref{3lemm: ext6} and (\ref{3isomorphism from flat}). The surjection $W_2\twoheadrightarrow\overline{L}(\lambda)$ induces an embedding
$$\mathrm{Ext}^1_{\mathrm{GL}_3(\mathbb{Q}_p),\lambda}\left(\overline{L}(\lambda),~\Sigma_1^{\flat}(\lambda, \mathscr{L}_1)\right)\hookrightarrow \mathrm{Ext}^1_{\mathrm{GL}_3(\mathbb{Q}_p),\lambda}\left(W_2,~\Sigma_1^{\flat}(\lambda, \mathscr{L}_1)\right)$$
which together with (\ref{3vanishing for flat}) imply that
$$\mathrm{Ext}^1_{\mathrm{GL}_3(\mathbb{Q}_p),\lambda}\left(\overline{L}(\lambda),~\Sigma_1^{\flat}(\lambda, \mathscr{L}_1)\right)=0$$
and hence
\begin{equation}\label{3vanishing for flat prime}
\mathrm{Ext}^1_{\mathrm{GL}_3(\mathbb{Q}_p),\lambda}\left(\overline{L}(\lambda),~\Sigma_1^{+,\flat}(\lambda, \mathscr{L}_1)\right)=0
\end{equation}
by (\ref{3vanishing of simple special ext}) and an easy devissage. It follows from (\ref{3vanishing for flat}) and (\ref{3vanishing for flat prime}) that
$$
\mathrm{Ext}^1_{\mathrm{GL}_3(\mathbb{Q}_p),\lambda}\left(\overline{L}(\lambda),~\begin{xy}
(0,0)*+{\Sigma_1^{+,\flat}(\lambda, \mathscr{L}_1)}="a"; (25,0)*+{\overline{L}(\lambda)\otimes_Ev_{P_2}^{\infty}}="b";
{\ar@{-}"a";"b"};
\end{xy}\right)=0
$$
which contradicts (\ref{3nonvanishing contradiction prime}). A a result, we have shown that
$$M^{\rm{min}}\cap M_{\{1,2\}}=0.$$
As $M^-\not\subseteq M_{\{1,2,i\}}$ for $i=3,4$, we deduce that both $M^{\rm{min}}\cap M_{\{1,2,3\}}$ and $M^{\rm{min}}\cap M_{\{1,2,4\}}$ are one dimensional. On the other hand, since we know that
$$(M^{\rm{min}}\cap M_{\{1,2,3\}})\cap(M^{\rm{min}}\cap M_{\{1,2,4\}})=M^{\rm{min}}\cap M_{\{1,2\}}=0,$$
we deduce the following direct sum decomposition
$$M^{\rm{min}}=(M^{\rm{min}}\cap M_{\{1,2,3\}})\oplus (M^{\rm{min}}\cap M_{\{1,2,4\}}).$$
\end{proof}
We use the notation $\overline{L}(\lambda)^i$ for copy of $\overline{L}(\lambda)$ inside $\overline{L}(\lambda)^{\oplus2}$ corresponding to the one dimensional space $M^{\rm{min}}\cap M_{\{1,2,i+2\}}$ inside $M^{\rm{min}}$, and therefore we have a surjection
\begin{equation}
\Sigma^{\rm{min}}(\lambda, \mathscr{L}_1, \mathscr{L}_2, \mathscr{L}_3)\twoheadrightarrow \left(
\begin{xy}
(0,0)*+{C^1_{s_2,1}}="a"; (16,0)*+{\overline{L}(\lambda)^1}="b";
{\ar@{-}"a";"b"};
\end{xy}\right)\oplus \left(\begin{xy}
(0,0)*+{C^1_{s_1,1}}="a"; (16,0)*+{\overline{L}(\lambda)^2}="b";
{\ar@{-}"a";"b"};
\end{xy}\right).
\end{equation}

As a result, the representation $\Sigma^{\rm{min}}(\lambda, \mathscr{L}_1, \mathscr{L}_2, \mathscr{L}_3)$ has the following form:
\begin{equation}\label{3simple picture for min}
\begin{xy}
(0,0)*+{\mathrm{St}_3^{\rm{an}}(\lambda)}="a"; (20,5)*+{v_{P_1}^{\rm{an}}(\lambda)}="b"; (40,9)*+{C_{s_1,s_1}}="d"; (60,9)*+{\overline{L}(\lambda)\otimes_Ev_{P_2}^{\infty}}="f"; (20,-5)*+{v_{P_2}^{\rm{an}}(\lambda)}="c"; (40,-9)*+{C_{s_2,s_2}}="e"; (60,-9)*+{\overline{L}(\lambda)\otimes_Ev_{P_1}^{\infty}}="g"; (90,3)*+{\overline{L}(\lambda)^1}="h1"; (90,-3)*+{\overline{L}(\lambda)^2}="h2";
{\ar@{-}"a";"b"}; {\ar@{-}"a";"c"}; {\ar@{-}"b";"d"}; {\ar@{-}"c";"e"}; {\ar@{-}"d";"f"}; {\ar@{-}"e";"g"}; {\ar@{-}"f";"h1"}; {\ar@{-}"b";"h1"}; {\ar@{-}"g";"h1"}; {\ar@{-}"f";"h2"}; {\ar@{-}"c";"h2"}; {\ar@{-}"g";"h2"};
\end{xy}.
\end{equation}
If we clarify the internal structure of $\mathrm{St}_3^{\rm{an}}(\lambda)$, $v_{P_1}^{\rm{an}}(\lambda)$ and $v_{P_2}^{\rm{an}}(\lambda)$ using Lemma~\ref{3lemm: structure of St}, then $\Sigma^{\rm{min}}(\lambda, \mathscr{L}_1, \mathscr{L}_2, \mathscr{L}_3)$ has the following form:
\begin{equation}\label{3main picture for min}
\begin{xy}
(-20,0)*+{\overline{L}(\lambda)\otimes_E\mathrm{St}_3^{\infty}}="a1";
(-6,9)*+{C^2_{s_1,1}}="a2"; (13.5,22)*+{C^1_{s_2s_1,1}}="a4"; (51,21)*+{C^2_{s_2s_1,1}}="a6"; (20,14)*+{\overline{L}(\lambda)\otimes_Ev_{P_1}^{\infty}}="b1"; (70,12)*+{C^1_{s_2,1}}="b2"; (42,30)*+{C_{s_1,s_1}}="d1"; (77,22)*+{\overline{L}(\lambda)\otimes_Ev_{P_2}^{\infty}}="f1"; (-6,-9)*+{C^2_{s_2,1}}="a3";(13.5,-22)*+{C^1_{s_1s_2,1}}="a5"; (51,-21)*+{C^2_{s_1s_2,1}}="a7"; (20,-14)*+{\overline{L}(\lambda)\otimes_Ev_{P_2}^{\infty}}="c1"; (70,-12)*+{C^1_{s_1,1}}="c2"; (42,-30)*+{C_{s_2,s_2}}="e1"; (77,-22)*+{\overline{L}(\lambda)\otimes_Ev_{P_1}^{\infty}}="g1"; (100,5)*+{\overline{L}(\lambda)^1}="h1"; (100,-5)*+{\overline{L}(\lambda)^2}="h2"; (73,-0)*+{C_{s_1s_2s_1,1}}="a8";
{\ar@{-}"a1";"a2"}; {\ar@{-}"a2";"a4"}; {\ar@{-}"a2";"b1"}; {\ar@{-}"a2";"a7"}; {\ar@{--}"a4";"b2"}; {\ar@{-}"a4";"a6"}; {\ar@{-}"a4";"d1"}; {\ar@{--}"a6";"c2"}; {\ar@{--}"a6";"a8"}; {\ar@{-}"b1";"b2"}; {\ar@{-}"b1";"d1"}; {\ar@{-}"b2";"h1"}; {\ar@{-}"d1";"f1"}; {\ar@{-}"f1";"h2"}; {\ar@{-}"f1";"h1"}; {\ar@{-}"a1";"a3"}; {\ar@{-}"a3";"a5"}; {\ar@{-}"a3";"c1"}; {\ar@{-}"a3";"a6"}; {\ar@{--}"a5";"c2"}; {\ar@{-}"a5";"e1"}; {\ar@{-}"a5";"a7"}; {\ar@{--}"a7";"b2"}; {\ar@{--}"a7";"a8"}; {\ar@{-}"c1";"c2"}; {\ar@{-}"c1";"e1"}; {\ar@{-}"c2";"h2"}; {\ar@{-}"e1";"g1"}; {\ar@{-}"g1";"h1"}; {\ar@{-}"g1";"h2"};
\end{xy}.
\end{equation}
\begin{rema}
It is actually possible to show that all the possibly split extensions illustrated in (\ref{3main picture for min}) are non-split. However, the proof is quite technical and not related to the $p$\text{-}adic dilogarithm function, and thus we decided not to include the proof here.
\end{rema}
We observe that $\Sigma^{\rm{min}}(\lambda, \mathscr{L}_1, \mathscr{L}_2, \mathscr{L}_3)$ admits a unique subrepresentation $\Sigma^{\mathrm{Ext}^1, -}(\lambda, \mathscr{L}_1, \mathscr{L}_2, \mathscr{L}_3)$ of the form
$$
\begin{xy}
(0,0)*+{\overline{L}(\lambda)\otimes_E\mathrm{St}_3^{\infty}}="a1";
(20,8)*+{C^2_{s_1,1}}="a2"; (45,12)*+{C^1_{s_2s_1,1}}="a4"; (45,4)*+{\overline{L}(\lambda)\otimes_Ev_{P_1}^{\infty}}="b1"; (70,8)*+{C_{s_1,s_1}}="d1"; (95,4)*+{\overline{L}(\lambda)\otimes_Ev_{P_2}^{\infty}}="f1"; (20,-8)*+{C^2_{s_2,1}}="a3";(45,-12)*+{C^1_{s_1s_2,1}}="a5"; (45,-4)*+{\overline{L}(\lambda)\otimes_Ev_{P_2}^{\infty}}="c1"; (70,-8)*+{C_{s_2,s_2}}="e1"; (95,-4)*+{\overline{L}(\lambda)\otimes_Ev_{P_1}^{\infty}}="g1";
{\ar@{-}"a1";"a2"}; {\ar@{-}"a2";"a4"}; {\ar@{-}"a2";"b1"}; {\ar@{-}"a4";"d1"}; {\ar@{-}"b1";"d1"}; {\ar@{-}"d1";"f1"}; {\ar@{-}"a1";"a3"}; {\ar@{-}"a3";"a5"}; {\ar@{-}"a3";"c1"}; {\ar@{-}"a5";"e1"}; {\ar@{-}"c1";"e1"}; {\ar@{-}"e1";"g1"};
\end{xy}
$$
which can be uniquely extend to a representation $\Sigma^{\mathrm{Ext}^1}(\lambda, \mathscr{L}_1, \mathscr{L}_2, \mathscr{L}_3)$ of the form:
\begin{equation}\label{3main picture for Ext1}
\begin{xy}
(0,0)*+{\overline{L}(\lambda)\otimes_E\mathrm{St}_3^{\infty}}="a1";
(20,8)*+{C^2_{s_1,1}}="a2"; (45,12)*+{C^1_{s_2s_1,1}}="a4"; (45,4)*+{\overline{L}(\lambda)\otimes_Ev_{P_1}^{\infty}}="b1"; (70,8)*+{C_{s_1,s_1}}="d1"; (95,4)*+{\overline{L}(\lambda)\otimes_Ev_{P_2}^{\infty}}="f1"; (20,-8)*+{C^2_{s_2,1}}="a3";(45,-12)*+{C^1_{s_1s_2,1}}="a5"; (45,-4)*+{\overline{L}(\lambda)\otimes_Ev_{P_2}^{\infty}}="c1"; (70,-8)*+{C_{s_2,s_2}}="e1"; (95,-4)*+{\overline{L}(\lambda)\otimes_Ev_{P_1}^{\infty}}="g1"; (95,12)*+{C^1_{s_2s_1,s_2s_1}}="i1"; (95,-12)*+{C^1_{s_1s_2,s_1s_2}}="j1"; (120,8)*+{C^2_{s_1,s_1s_2}}="i2"; (120,-8)*+{C^2_{s_2,s_2s_1}}="j2";
{\ar@{-}"a1";"a2"}; {\ar@{-}"a2";"a4"}; {\ar@{-}"a2";"b1"}; {\ar@{-}"a4";"d1"}; {\ar@{-}"b1";"d1"}; {\ar@{-}"d1";"f1"}; {\ar@{-}"a1";"a3"}; {\ar@{-}"a3";"a5"}; {\ar@{-}"a3";"c1"}; {\ar@{-}"a5";"e1"}; {\ar@{-}"c1";"e1"}; {\ar@{-}"e1";"g1"}; {\ar@{-}"d1";"i1"}; {\ar@{-}"e1";"j1"}; {\ar@{-}"i1";"i2"}; {\ar@{-}"j1";"j2"}; {\ar@{-}"f1";"i2"}; {\ar@{-}"g1";"j2"};
\end{xy}
\end{equation}
according to Section~4.4 and 4.6 of \cite{Bre17} together with our Lemma~\ref{3lemm: existence of diamond}. Finally, we define $\Sigma^{\rm{min},+}(\lambda, \mathscr{L}_1, \mathscr{L}_2, \mathscr{L}_3)$ as the amalgamate sum of $\Sigma^{\rm{min}}(\lambda, \mathscr{L}_1, \mathscr{L}_2, \mathscr{L}_3)$ and $\Sigma^{\mathrm{Ext}^1}(\lambda, \mathscr{L}_1, \mathscr{L}_2, \mathscr{L}_3)$ over $\Sigma^{\mathrm{Ext}^1, -}(\lambda, \mathscr{L}_1, \mathscr{L}_2, \mathscr{L}_3)$.
\begin{rema}\label{3quotient independent of invariants}
It is actually possible to prove (by several technical computations of $\mathrm{Ext}$\text{-}groups) that the quotient
$$\Sigma^{\rm{min},+}(\lambda, \mathscr{L}_1, \mathscr{L}_2, \mathscr{L}_3)/\overline{L}(\lambda)\otimes_E\mathrm{St}_3^{\infty}$$
and the quotient
$$\Sigma^{\rm{min}}(\lambda, \mathscr{L}_1, \mathscr{L}_2, \mathscr{L}_3)/\overline{L}(\lambda)\otimes_E\mathrm{St}_3^{\infty}$$
are independent of the choices of $\mathscr{L}_1, \mathscr{L}_2, \mathscr{L}_3\in E$.
\end{rema}
\section{Local-global compatibility}\label{3section: local-global}
We are going to borrow most of the notation and assumptions from Section~6 of \cite{Bre17}. We fix embeddings $\iota_{\infty}:\overline{\mathbb{Q}}\hookrightarrow\mathbb{C}$, $\iota_p: \overline{\mathbb{Q}}$, an imaginary quadratic CM extension $F$ of $\mathbb{Q}$ and a unitary group $G/\mathbb{Q}$ attached to the extension $F/\mathbb{Q}$ such that $G\times_{\mathbb{Q}}F\cong\mathrm{GL}_3$ and $G(\mathbb{R})$ is compact. If $\ell$ is a finite place of $\mathbb{Q}$ which splits completely in $F$, we have isomorphisms $\iota_{G,w}: G(\mathbb{Q}_{\ell})\xrightarrow{\sim}G(F_w)\cong\mathrm{GL}_3(F_w)$ for each finite place $w$ of $F$ over $\ell$. We assume that $p$ splits completely in $F$, and we fix a finite place $w_0$ of $F$ dividing $p$ and therefore $G(\mathbb{Q}_p)\cong G(F_{w_0})\cong \mathrm{GL}_3(\mathbb{Q}_p)$.

We fix an open compact subgroup $U^p\subsetneq G(\mathbb{A}^{\infty,p}_{\mathbb{Q}})$ of the form $U^p=\prod_{\ell\neq p}U_{\ell}$ where $U_{\ell}$ is an open compact subgroup of $G(\mathbb{Q}_{\ell})$. For each finite extension $E$ of $\mathbb{Q}_p$ inside $\overline{\mathbb{Q}_p}$, we consider the following $\mathcal{O}_E$-lattice inside a $p$\text{-}adic Banach space:
\begin{equation}
\widehat{S}(U^p,\mathcal{O}_E):=\{f : G(\mathbb{Q})\backslash G(\mathbb{A}^{\infty}_{\mathbb{Q}})/U^p\rightarrow \mathcal{O}_E, ~f \text{ continuous}\}
\end{equation}
and note that $\widehat{S}(U^p,E):=\widehat{S}(U^p,\mathcal{O}_E)\otimes_{\mathcal{O}_E}E$. The right translation of $G(\mathbb{Q}_p)$ on $G(\mathbb{Q})\backslash G(\mathbb{A}^{\infty}_{\mathbb{Q}})/U^p$ induces a $p$\text{-}adic continuous action of $G(\mathbb{Q}_p)$ on $\widehat{S}(U^p,\mathcal{O}_E)$ which makes $\widehat{S}(U^p,E)$ an admissible Banach representation of $G(\mathbb{Q}_p)$ in the sense of \cite{ST02}. We use the notation $\widehat{S}(U^p,E)^{\rm{alg}}\subseteq\widehat{S}(U^p,E)^{\rm{an}}$ following Section~6 of \cite{Bre17} for the subspaces of locally $\mathbb{Q}_p$-algebraic vectors and locally $\mathbb{Q}_p$-analytic vectors inside $\widehat{S}(U^p,E)$ respectively. Moreover, we have the following decomposition:
\begin{equation}\label{3global decomposition}
\widehat{S}(U^p,E)^{\rm{alg}}\otimes_E\overline{\mathbb{Q}_p}\cong\bigoplus_{\pi}(\pi_f^{v_0})^{U_p}\otimes_{\overline{\mathbb{Q}}}(\pi_{v_0}\otimes_{\overline{\mathbb{Q}}}W_p)
\end{equation}
where the direct sum is over the automorphic representations $\pi$ of $G(\mathbb{A}_{\mathbb{Q}})$ over $\mathbb{C}$ and $W_p$ is the $\mathbb{Q}_p$-algebraic representation of $G(\mathbb{Q}_p)$ over $\overline{\mathbb{Q}_p}$ associated with the algebraic representation $\pi_{\infty}$ of $G(\mathbb{R})$ over $\mathbb{C}$ via $\iota_p$ and $\iota_{\infty}$. In particular, each distinct $\pi$ appears with multiplicity one ( cf. the paragraph after (55) of \cite{Bre17} for further references).

We use the notation $D(U^p)$ for the set of finite places $\ell$ of $\mathbb{Q}$ that are different from $p$, split completely in $F$ and such that $U_{\ell}$ is a maximal open compact subgroup of $G(\mathbb{Q}_{\ell})$. Then we consider the commutative polynomial algebra $\mathbb{T}(U^p):=E[T^{(j)}_w]$ generated by the variables $T^{(j)}_w$ indexed by $j\in\{1,\cdots,n\}$ and $w$ a finite place of $F$ over a place $\ell$ of $\mathbb{Q}$ such that $\ell\in D(U^p)$. The algebra $\mathbb{T}(U^p)$ acts on $\widehat{S}(U^p,E)$, $\widehat{S}(U^p,E)^{\rm{alg}}$ and $\widehat{S}(U^p,E)^{\rm{an}}$ via the usual double coset operators. The action of $\mathbb{T}(U^p)$ commutes with that of $G(\mathbb{Q}_p)$.

We fix now $\alpha\in E^{\times}$, hence a Deligne--Fontaine module $\underline{D}$ over $\mathbb{Q}_p=F_{w_0}$ of rank three of the form
\begin{equation}\label{3Deligne Fontaine}
\underline{D}=Ee_2\oplus Ee_1\oplus Ee_0,\mbox{ with }
\left\{\begin{array}{ccc}
\varphi(e_2)&=&\alpha e_2\\
\varphi(e_1)&=&p^{-1}\alpha e_1\\
\varphi(e_0)&=&p^{-2}\alpha e_0\\
\end{array}\right.
\mbox{ and }
\left\{\begin{array}{ccc}
N(e_2)&=&e_1\\
N(e_1)&=&e_0\\
N(e_0)&=&0\\
\end{array}\right.
\end{equation}
and finally a tuple of Hodge--Tate weights $\underline{k}=(k_1>k_2>k_3)$. If $\rho: \mathrm{Gal}(\overline{F}/F)\rightarrow \mathrm{GL}_3(E)$ is an absolute irreducible continuous representation which is unramified at each finite place $w$ lying over a finite place $\ell\in D(U^p)$, we can associate to $\rho$ a maximal ideal $\mathfrak{m}_{\rho}\subseteq \mathbb{T}(U^p)$ with residual field $E$ by the usual method described in the middle paragraph on Page 58 of \cite{Bre17}. We use the notation $\star_{\mathfrak{m}_\rho}$ for spaces of localization and $\star[\mathfrak{m}_\rho]$ for torsion subspaces where $\star\in\{\widehat{S}(U^p,E), \widehat{S}(U^p,E)^{\rm{alg}}, \widehat{S}(U^p,E)^{\rm{an}}\}$.

We assume that there exists $U^p$ and $\rho$ such that
\begin{enumerate}
\item $\rho$ is absolutely irreducible and unramified at each finite place $w$ of $F$ over a place $\ell$ of $\mathbb{Q}$ satisfying $\ell\in D(U^p)$;
\item $\widehat{S}(U^p,E)^{\rm{alg}}[\mathfrak{m}_{\rho}]\neq 0$ (hence $\rho$ is automorphic and $\rho_{w_0}:=\rho|_{\mathrm{Gal}(\overline{F_{w_0}}/F_{w_0})}$ is potentially semi-stable);
\item $\rho_{w_0}$ has Hodge--Tate weights $\underline{k}$ and gives the Deligne--Fontaine module $\underline{D}$.
\end{enumerate}
By identifying $\widehat{S}(U^p,E)^{\rm{alg}}$ with a representation of $\mathrm{GL}_3(\mathbb{Q}_p)$ via $\iota_{G,w_0}$, we have the following isomorphism up to normalization from (\ref{3global decomposition}) and \cite{Ca14}:
\begin{equation}
\widehat{S}(U^{v_0},E)^{\rm{alg}}[\mathfrak{m}_{\rho}]\cong\left(\overline{L}(\lambda)\otimes_E\mathrm{St}_3^{\infty}\otimes_E(\mathrm{ur}(\alpha)\otimes_E\varepsilon^2)\circ\mathrm{det}\right)^{\oplus d(U^p, \rho)}
\end{equation}
for all $(U^p, \rho)$ satisfying the conditions (i), (ii) and (iii), where $\lambda=(\lambda_1,\lambda_2,\lambda_3)=(k_1-2,k_2-1,k_3)$ and $d(U^p,\rho)\geq 1$ is an integer depending only on $U^p$ and $\rho$.

\begin{theo}\label{3theo: main}
We consider $U^p=\prod_{\ell\neq p}U_{\ell}$ and $\rho: \mathrm{Gal}(\overline{F}/F)\rightarrow\mathrm{GL}_3(E)$ such that
\begin{enumerate}
\item $\rho$ is absolutely irreducible and unramified at each finite place $w$ of $F$ lying above $D(U^p)$;
\item $\widehat{S}(U^p, E)^{\rm{alg}}[\mathfrak{m}_{\rho}]\neq 0$;
\item $\rho$ has Hodge--Tate weights $\underline{k}$ and gives the Deligne--Fontaine module $\underline{D}$ as in (\ref{3Deligne Fontaine});
\item the filtration on $\underline{D}$ is non-critical in the sense of (ii) of Remark~6.1.4 of \cite{Bre17};
\item only one automorphic representation $\pi$ contributes to $\widehat{S}(U^p, E)^{\rm{alg}}[\mathfrak{m}_{\rho}]$.
\end{enumerate}
Then there exists a unique choice of $\mathscr{L}_1, \mathscr{L}_2, \mathscr{L}_3\in E$ such that:
\begin{multline}
\mathrm{Hom}_{\mathrm{GL}_3(\mathbb{Q}_p)}\left(\Sigma^{\rm{min},+}(\lambda, \mathscr{L}_1, \mathscr{L}_2, \mathscr{L}_3)\otimes_E(\mathrm{ur}(\alpha)\otimes_E\varepsilon^2)\circ\mathrm{det}, ~\widehat{S}(U^p, E)^{\rm{an}}[\mathfrak{m}_{\rho}]\right)\\
\xrightarrow{\sim}\mathrm{Hom}_{\mathrm{GL}_3(\mathbb{Q}_p)}\left(\overline{L}(\lambda)\otimes_E\mathrm{St}_3^{\infty}\otimes_E(\mathrm{ur}(\alpha)\otimes_E\varepsilon^2)\circ\mathrm{det}, ~\widehat{S}(U^p, E)^{\rm{an}}[\mathfrak{m}_{\rho}]\right).\\
\end{multline}
\end{theo}

We recall several useful results from \cite{Bre17} and \cite{BH18}.
\begin{prop}\label{3prop: injectivity}
Suppose that $U^p=\prod_{\ell\neq p}U_{\ell}$ is a sufficiently small open compact subgroup of $G(\mathbb{A}_{\mathbb{Q}}^{\infty,p})$, $\widehat{S}(U^p, E)^{\rm{an}}\hookrightarrow \Pi\twoheadrightarrow \Pi_1$ a short exact sequence of admissible locally analytic representations of $\mathrm{GL}_3(\mathbb{Q}_p)$, $\chi: T(\mathbb{Q}_p)\rightarrow E^{\times}$ a locally analytic character and $\eta: U(\mathfrak{t})\rightarrow E$ its derived character, then we have  $T(\mathbb{Q}_p)^+$-equivariant short exact sequences of finite dimensional $E$-spaces
$$(\widehat{S}(U^p, E)^{\rm{an}})^{\overline{N}(\mathbb{Z}_p)}[\mathfrak{t}=\eta]\hookrightarrow \Pi^{\overline{N}(\mathbb{Z}_p)}[\mathfrak{t}=\eta]\twoheadrightarrow \Pi_1^{\overline{N}(\mathbb{Z}_p)}[\mathfrak{t}=\eta]$$
and
$$(\widehat{S}(U^p, E)^{\rm{an}})^{\overline{N}(\mathbb{Z}_p)}[\mathfrak{t}=\eta]_{\chi}\hookrightarrow \Pi^{\overline{N}(\mathbb{Z}_p)}[\mathfrak{t}=\eta]_{\chi}\twoheadrightarrow \Pi_1^{\overline{N}(\mathbb{Z}_p)}[\mathfrak{t}=\eta]_{\chi}$$
where $T(\mathbb{Q}_p)^+$ is a submonoid of $T(\mathbb{Q}_p)$ defined by
$$T(\mathbb{Q}_p)^+:=\{t\in T(\mathbb{Q}_p)\mid t\overline{N}(\mathbb{Z}_p)t^{-1}\subseteq \overline{N}(\mathbb{Z}_p)\}.$$
\end{prop}
\begin{proof}
This is Proposition~6.3.3 of \cite{Bre17} and Proposition~4.1 of \cite{BH18}.
\end{proof}

\begin{prop}\label{3prop: socle}
We fix $U^p$ and $\rho$ as in Theorem~\ref{3theo: main}. For a locally analytic character $\chi: T(\mathbb{Q}_p)\rightarrow E^{\times}$, we have
$$\mathrm{Hom}_{T(\mathbb{Q}_p)^+}\left(\chi\otimes_E(\mathrm{ur}(\alpha)\otimes_E\varepsilon^2)\circ\mathrm{det}, ~(\widehat{S}(U^p, E)^{\rm{an}}[\mathfrak{m}_{\rho}])^{\overline{N}(\mathbb{Z}_p)}\right)\neq 0$$
if and only if $\chi=\delta_{\lambda}$.
\end{prop}
\begin{proof}
This is Proposition~6.3.4 of \cite{Bre17}.
\end{proof}

We recall the notation $i_{B}^{\mathrm{GL}_3}(\chi_w^{\infty})$ for a smooth principal series for each $w\in W$ from Section~\ref{3subsection: main notation}. Given three locally analytic representations $V_i$ for $i=1,2,3$ and two surjections $V_1\twoheadrightarrow V_2$ and $V_3\twoheadrightarrow V_2$, we use the notation $V_1\times_{V_2}V_3$ for the representation given by the fiber product of $V_1$ and $V_3$ over $V_2$ with natural surjections $V_1\times_{V_2}V_3\twoheadrightarrow V_1$ and $V_1\times_{V_2}V_3\twoheadrightarrow V_3$. We also use the shorten notation $V^{\rm{alg}}$ for the maximally locally algebraic subrepresentation of a locally analytic representation $V$. We recall that $U^p$ is \emph{sufficiently small} if there exists $\ell\neq p$ such that $U_{\ell}$ has no non-trivial element with finite order.
\begin{prop}\label{3prop: adjunction}
We fix $U^p$ and $\rho$ as in Theorem~\ref{3theo: main} and assume moreover that $U^p$ is a sufficiently small open compact subgroup of $G(\mathbb{A}_{\mathbb{Q}}^{\infty,p})$. We also fix a non-split short exact sequence $V_1\hookrightarrow V_2\twoheadrightarrow V_3$ of finite length representations inside the category $\mathrm{Rep}^{\mathcal{OS}}_{\mathrm{GL}_3(\mathbb{Q}_p), E}$ such that $V_1\otimes_E(\mathrm{ur}(\alpha)\otimes_E\varepsilon^2)\circ\mathrm{det}$ embeds into $\widehat{S}(U^p, E)^{\rm{an}}[\mathfrak{m}_{\rho}]$. We conclude that:
\begin{enumerate}
\item if $V_3$ is irreducible and not locally algebraic, then we have an embedding
 $$V_2\otimes_E(\mathrm{ur}(\alpha)\otimes_E\varepsilon^2)\circ\mathrm{det}\hookrightarrow\widehat{S}(U^p, E)^{\rm{an}}[\mathfrak{m}_{\rho}];$$
\item if there is a surjection
$$\overline{L}(\lambda)\otimes_Ei_{B}^{\mathrm{GL}_3}(\chi_w^{\infty})\twoheadrightarrow V_3$$
for a certain $w\in W$, then there exists a certain quotient $V_4$ of $V_2\times_{V_3}\left(\overline{L}(\lambda)\otimes_Ei_{B}^{\mathrm{GL}_3}(\chi_w^{\infty})\right)$ satisfying $$\mathrm{soc}_{\mathrm{GL}_3(\mathbb{Q}_p)}(V_4)=V_4^{\rm{alg}}=\overline{L}(\lambda)\otimes_E\mathrm{St}_3^{\infty}$$
such that we have an embedding
$$V_4\otimes_E(\mathrm{ur}(\alpha)\otimes_E\varepsilon^2)\circ\mathrm{det}\hookrightarrow\widehat{S}(U^p, E)^{\rm{an}}[\mathfrak{m}_{\rho}].$$
\end{enumerate}
\end{prop}
\begin{proof}
This is an immediate generalization (or rather formalization) of Section~6.4 of \cite{Bre17}. More precisely, part (i) (resp. (ii)) generalizes the \'Etape 1 (resp. the \'Etape 2) of Section~6.4 of \cite{Bre17}.
\end{proof}

\begin{proof}[proof of Theorem~\ref{3theo: main}]
We may assume that $\alpha=1$ for simplicity of notation thanks to Lemma~\ref{3lemm: det twist}. According to the \'Etape 1 and 2 of Section~6.2 of \cite{Bre17}, we may assume without loss of generality that $U^p$ is sufficiently small and it is sufficient to show that there exists a unique choice of $\mathscr{L}_1, \mathscr{L}_2, \mathscr{L}_3\in E$ such that
\begin{equation}\label{3global embedding}
\mathrm{Hom}_{\mathrm{GL}_3(\mathbb{Q}_p)}\left(\Sigma^{\rm{min}, +}(\lambda, \mathscr{L}_1, \mathscr{L}_2, \mathscr{L}_3)\otimes_E(\mathrm{ur}(\alpha)\otimes_E\varepsilon^2)\circ\mathrm{det}, ~\widehat{S}(U^p, E)^{\rm{an}}[\mathfrak{m}_{\rho}]\right)\neq 0.
\end{equation}
We borrow the notation $\Pi^i(\underline{k}, \underline{D})$ from Theorem~6.2.1 of \cite{Bre17}. We observe from (\ref{3main picture for min}) that $\Sigma^{\rm{min},+}(\lambda, \mathscr{L}_1, \mathscr{L}_2, \mathscr{L}_3)$ contains a unique subrepresentation $\Sigma^{\mathrm{Ext}^1}(\lambda, \mathscr{L}_1, \mathscr{L}_2, \mathscr{L}_3)$ of the form
\begin{equation}\label{3special form ext 1}
\begin{xy}
(0,0)*+{\overline{L}(\lambda)\otimes_E\mathrm{St}_3^{\infty}}="a"; (30,6)*+{\Pi^1(\underline{k}, \underline{D})}="b"; (30,-6)*+{\Pi^2(\underline{k}, \underline{D})}="c";
{\ar@{-}"a";"b"}; {\ar@{-}"a";"c"};
\end{xy}.
\end{equation}
Moreover, $\Sigma^{\rm{min},+}(\lambda, \mathscr{L}_1, \mathscr{L}_2, \mathscr{L}_3)$ is uniquely determined by $\Sigma^{\mathrm{Ext}^1}(\lambda, \mathscr{L}_1, \mathscr{L}_2, \mathscr{L}_3)$ up to isomorphism. It is known by \'Etape 3 of Section~6.2 of \cite{Bre17} that there is at most one choice of $\mathscr{L}_1, \mathscr{L}_2, \mathscr{L}_3\in E$ such that
$$\mathrm{Hom}_{\mathrm{GL}_3(\mathbb{Q}_p)}\left(\Sigma^{\mathrm{Ext}^1}(\lambda, \mathscr{L}_1, \mathscr{L}_2, \mathscr{L}_3)\otimes_E(\mathrm{ur}(\alpha)\otimes_E\varepsilon^2)\circ\mathrm{det}, ~\widehat{S}(U^p, E)^{\rm{an}}[\mathfrak{m}_{\rho}]\right)\neq 0,$$
and thus there is at most one choice of $\mathscr{L}_1, \mathscr{L}_2, \mathscr{L}_3\in E$ such that (\ref{3global embedding}) holds. As a result, it remains to show the existence of $\mathscr{L}_1, \mathscr{L}_2, \mathscr{L}_3\in E$ that satisfies (\ref{3global embedding}). We notice that $\Sigma^{\rm{min},+}(\lambda, \mathscr{L}_1, \mathscr{L}_2, \mathscr{L}_3)$ admits an increasing filtration $\mathrm{Fil}_{\bullet}$ satisfying the following conditions
\begin{enumerate}
\item the representations $\Sigma^{\rm{min}}(\lambda, \mathscr{L}_1, \mathscr{L}_2, \mathscr{L}_3)$ and $\Sigma^{\sharp,+}(\lambda, \mathscr{L}_1, \mathscr{L}_2)$ ( cf. their definition after Proposition~\ref{3prop: main dim} and Proposition~\ref{3prop: criterion of existence}) appear as two consecutive terms of the filtration;
\item each graded piece is either locally algebraic or irreducible.
\end{enumerate}
As a result, the only reducible graded pieces of this filtration is the quotient
$$\Sigma^{\rm{min}}(\lambda, \mathscr{L}_1, \mathscr{L}_2, \mathscr{L}_3)/\Sigma^{\sharp,+}(\lambda, \mathscr{L}_1, \mathscr{L}_2)\cong W_0.$$
Then we can prove the existence of $\mathscr{L}_1, \mathscr{L}_2, \mathscr{L}_3\in E$ satisfying (\ref{3global embedding}) by reducing to the isomorphism
\begin{multline}\label{3induction on filtration}
\mathrm{Hom}_{\mathrm{GL}_3(\mathbb{Q}_p)}\left(\mathrm{Fil}_{k+1}\Sigma^{\rm{max}}(\lambda, \mathscr{L}_1, \mathscr{L}_2, \mathscr{L}_3)\otimes_E(\mathrm{ur}(\alpha)\otimes_E\varepsilon^2)\circ\mathrm{det}, ~\widehat{S}(U^p, E)^{\rm{an}}[\mathfrak{m}_{\rho}]\right)\\
\xrightarrow{\sim} \mathrm{Hom}_{\mathrm{GL}_3(\mathbb{Q}_p)}\left(\mathrm{Fil}_k\Sigma^{\rm{max}}(\lambda, \mathscr{L}_1, \mathscr{L}_2, \mathscr{L}_3)\otimes_E(\mathrm{ur}(\alpha)\otimes_E\varepsilon^2)\circ\mathrm{det}, ~\widehat{S}(U^p, E)^{\rm{an}}[\mathfrak{m}_{\rho}]\right)
\end{multline}
for each $k\in\mathbb{Z}$. If
$$
\mathrm{Gr}_k:=\mathrm{Fil}_{k+1}\Sigma^{\rm{min}}(\lambda, \mathscr{L}_1, \mathscr{L}_2, \mathscr{L}_3)/\mathrm{Fil}_k\Sigma^{\rm{min}}(\lambda, \mathscr{L}_1, \mathscr{L}_2, \mathscr{L}_3)$$
is not locally algebraic, then (\ref{3induction on filtration}) is true in this case by part (i) of Proposition~\ref{3prop: adjunction}. The only locally algebraic graded pieces of the filtration except $\overline{L}(\lambda)\otimes_E\mathrm{St}_3^{\infty}$ are $\overline{L}(\lambda)\otimes_Ev_{P_1}^{\infty}$, $\overline{L}(\lambda)\otimes_Ev_{P_2}^{\infty}$ and $W_0$. The isomorphism (\ref{3induction on filtration}) when the graded piece $\mathrm{Gr}_k$ equals $\overline{L}(\lambda)\otimes_Ev_{P_1}^{\infty}$ or $\overline{L}(\lambda)\otimes_Ev_{P_2}^{\infty}$ has been treated in \'Etape 2 of Section~6.4 of \cite{Bre17}. As a result, it remains to show that
\begin{multline}\label{3last isomorphism}
\mathrm{Hom}_{\mathrm{GL}_3(\mathbb{Q}_p)}\left(\Sigma^{\rm{min}}(\lambda, \mathscr{L}_1, \mathscr{L}_2, \mathscr{L}_3)\otimes_E(\mathrm{ur}(\alpha)\otimes_E\varepsilon^2)\circ\mathrm{det}, ~\widehat{S}(U^p, E)^{\rm{an}}[\mathfrak{m}_{\rho}]\right)\\
\xrightarrow{\sim} \mathrm{Hom}_{\mathrm{GL}_3(\mathbb{Q}_p)}\left(\Sigma^{\sharp,+}(\lambda, \mathscr{L}_1, \mathscr{L}_2)\otimes_E(\mathrm{ur}(\alpha)\otimes_E\varepsilon^2)\circ\mathrm{det}, ~\widehat{S}(U^p, E)^{\rm{an}}[\mathfrak{m}_{\rho}]\right)
\end{multline}
to finish the proof of Theorem~\ref{3theo: main}. It follows from results in Section~5.3 of \cite{Bre17} ( cf. (53) of \cite{Bre17}) that $i_{B}^{\mathrm{GL}_3}(\chi_{s_1s_2s_1}^{\infty})$ has the form
$$
\begin{xy}
(0,0)*+{\mathrm{St}_3^{\infty}}="a"; (20,4)*+{v_{P_1}^{\infty}}="b"; (20,-4)*+{v_{P_2}^{\infty}}="c"; (40,0)*+{1_3}="d";
{\ar@{-}"a";"b"}; {\ar@{-}"a";"c"}; {\ar@{-}"b";"d"}; {\ar@{-}"c";"d"};
\end{xy}
$$
and thus there is a surjection
$$\overline{L}(\lambda)\otimes_Ei_{B}^{\mathrm{GL}_3}(\chi_{s_1s_2s_1}^{\infty})\twoheadrightarrow W_0.$$
According to part (ii) of Proposition~\ref{3prop: adjunction}, we only need to show that any quotient $V$ of
$$V^{\diamond}:=\Sigma^{\rm{min}}(\lambda, \mathscr{L}_1, \mathscr{L}_2, \mathscr{L}_3)\times_{W_0}\left(\overline{L}(\lambda)\otimes_Ei_{B}^{\mathrm{GL}_3}(\chi_{s_1s_2s_1}^{\infty})\right)$$
such that
\begin{equation}\label{3necessary condition}
\mathrm{soc}_{\mathrm{GL}_3(\mathbb{Q}_p)}(V)=V^{\rm{alg}}=\overline{L}(\lambda)\otimes_E\mathrm{St}_3^{\infty}
\end{equation}
must have the form
$$\Sigma^{\rm{min}}(\lambda, \mathscr{L}_1, \mathscr{L}_2, \mathscr{L}^{\prime}_3)$$
for certain $\mathscr{L}^{\prime}_3\in E$. We recall from Proposition~\ref{3prop: criterion of existence} and our definition of $\Sigma^{\rm{min}}(\lambda, \mathscr{L}_1, \mathscr{L}_2, \mathscr{L}_3)$ afterwards that $\Sigma^{\rm{min}}(\lambda, \mathscr{L}_1, \mathscr{L}_2, \mathscr{L}_3)$ fits into a short exact sequence
\begin{equation}
\Sigma^{\sharp, +}(\lambda, \mathscr{L}_1, \mathscr{L}_2)\hookrightarrow\Sigma^{\rm{min}}(\lambda, \mathscr{L}_1, \mathscr{L}_2, \mathscr{L}_3)\twoheadrightarrow W_0
\end{equation}
and thus $V^\diamond$ fits (by definition of fiber product) into a short exact sequence
\begin{equation}
\Sigma^{\sharp, +}(\lambda, \mathscr{L}_1, \mathscr{L}_2)\hookrightarrow V^{\diamond}\twoheadrightarrow i_{B}^{\mathrm{GL}_3}(\chi_{s_1s_2s_1}^{\infty})
\end{equation}
and in particular
$$\mathrm{soc}_{\mathrm{GL}_3(\mathbb{Q}_p)}(V^{\diamond})=\left(\overline{L}(\lambda)\otimes_E\mathrm{St}_3^{\infty}\right)^{\oplus2}.$$
Hence the condition (\ref{3necessary condition}) implies that $V$ fits into a short exact sequence
$$\overline{L}(\lambda)\otimes_E\mathrm{St}_3^{\infty}\xrightarrow{j} V^{\diamond}\twoheadrightarrow V$$
and that
$$j\left(\overline{L}(\lambda)\otimes_E\mathrm{St}_3^{\infty}\right)\cap \Sigma^{\sharp, +}(\lambda, \mathscr{L}_1, \mathscr{L}_2)=0\subseteq V^{\diamond}$$
which induces an injection
$$\Sigma^{\sharp, +}(\lambda, \mathscr{L}_1, \mathscr{L}_2)\hookrightarrow V.$$
Therefore $V$ fits into a short exact sequence
$$\Sigma^{\sharp, +}(\lambda, \mathscr{L}_1, \mathscr{L}_2)\hookrightarrow V\twoheadrightarrow W_0$$
and thus corresponds to a line $M_V$ inside
$$\mathrm{Ext}^1_{\mathrm{GL}_3(\mathbb{Q}_p),\lambda}\left(W_0,~\Sigma^{\sharp, +}(\lambda, \mathscr{L}_1, \mathscr{L}_2)\right)$$
which is two dimensional by Lemma~\ref{3lemm: preparation for global}. Moreover, the condition (\ref{3necessary condition}) implies that $M_V$ is different from the line
$$\mathrm{Ext}^1_{\mathrm{GL}_3(\mathbb{Q}_p),\lambda}\left(W_0,~\overline{L}(\lambda)\otimes_E\mathrm{St}_3^{\infty}\right)\hookrightarrow\mathrm{Ext}^1_{\mathrm{GL}_3(\mathbb{Q}_p),\lambda}\left(W_0,~\Sigma^{\sharp, +}(\lambda, \mathscr{L}_1, \mathscr{L}_2)\right).$$
Hence it follows from Lemma~\ref{3lemm: preparation for global} that there exists $\mathscr{L}^{\prime}_3\in E$ such that
$$V\cong \Sigma^{\rm{min}}(\lambda, \mathscr{L}_1, \mathscr{L}_2, \mathscr{L}^{\prime}_3).$$
\end{proof}

\begin{coro}\label{3coro: criterion}
If a locally analytic representation $\Pi$ of the form (\ref{3special form ext 1}) is contained in $\widehat{S}(U^p, E)^{\rm{an}}[\mathfrak{m}_{\rho}]$ for a certain $U^p$ and $\rho$ as in Theorem~\ref{3theo: main}, then there exists $\mathscr{L}_1, \mathscr{L}_2, \mathscr{L}_3\in E$ uniquely determined by $\Pi$ such that
$$\Pi\hookrightarrow \Sigma^{\rm{min},+}(\lambda, \mathscr{L}_1, \mathscr{L}_2, \mathscr{L}_3).$$
\end{coro}
\begin{proof}
We fix $U^p$ and $\rho$ such that the embedding
\begin{equation}\label{3embedding as assumption}
\Pi\hookrightarrow \widehat{S}(U^p, E)^{\rm{an}}[\mathfrak{m}_{\rho}]
\end{equation}
exists. Then (\ref{3embedding as assumption}) restricts to an embedding
$$
\overline{L}(\lambda)\otimes_E\mathrm{St}_3^{\infty}\hookrightarrow \widehat{S}(U^p, E)^{\rm{an}}[\mathfrak{m}_{\rho}]$$
which extends to an embedding
\begin{equation}\label{3embedding as conclusion}
\Sigma^{\rm{min},+}(\lambda, \mathscr{L}_1, \mathscr{L}_2, \mathscr{L}_3)\hookrightarrow \widehat{S}(U^p, E)^{\rm{an}}[\mathfrak{m}_{\rho}]
\end{equation}
for a unique choice of $\mathscr{L}_1, \mathscr{L}_2, \mathscr{L}_3\in E$ according to Theorem~\ref{3theo: main}. The embedding (\ref{3embedding as conclusion}) induces by restriction an embedding
$$\Sigma^{\mathrm{Ext}^1}(\lambda, \mathscr{L}_1, \mathscr{L}_2, \mathscr{L}_3)\hookrightarrow \widehat{S}(U^p, E)^{\rm{an}}[\mathfrak{m}_{\rho}]$$
and therefore we have
$$\Pi\cong \Sigma^{\mathrm{Ext}^1}(\lambda, \mathscr{L}_1, \mathscr{L}_2, \mathscr{L}_3)$$
by Theorem~6.2.1 of \cite{Bre17}. In particular, we deduce an embedding
$$\Pi\hookrightarrow \Sigma^{\rm{min},+}(\lambda, \mathscr{L}_1, \mathscr{L}_2, \mathscr{L}_3)$$
for certain invariants $\mathscr{L}_1, \mathscr{L}_2, \mathscr{L}_3\in E$ determined by $\Pi$.
\end{proof}

\bibliographystyle{alpha}

\end{document}